\documentclass[11pt,reqno,a4paper]{amsart}
\usepackage[centertags]{amsmath}
\usepackage{amsfonts}
\usepackage{amssymb}
\usepackage{amsthm}
\usepackage{amsmath}
\usepackage{times}
\usepackage[top=2.5cm, bottom=2.5cm, left=2.5cm, right=2.5cm]{geometry}
\usepackage{color}
\usepackage{graphicx}
\usepackage{hyperref}
\hypersetup{
    colorlinks=true,
    linkcolor=blue,
    filecolor=magenta,
    urlcolor=cyan,
}

\numberwithin{equation}{section}

\newcommand{\cE}{{\mathcal E}}

\newcommand{\cC}{{\mathcal C}_{\beta}}
\newcommand{\cB}{{\mathcal B}_{\beta}}
\newcommand{\cA}{\mathcal{A}_{\beta}}
\newcommand{\cAo}{\mathcal{A}_{\beta_1}}
\newcommand{\cAt}{\mathcal{A}_{\beta_2}}
\newcommand{\cG}{\mathcal{G}_{\beta}}

\newcommand{\tC}{\mathtt{C}_{\beta}}

\newcommand{\cV}{\mathcal{V}}
\newcommand{\cH}{\mathcal{H}}
\newcommand{\cU}{\mathcal{U}}
\newcommand{\cL}{\mathcal{L}}
\newcommand{\cR}{\mathcal{R}}

\newcommand{\dOm}{{\partial \Omega}}

\newcommand{\N}{\mathbb{N}}
\newcommand{\Z}{\mathbb{Z}}
\newcommand{\C}{\mathbb{C}}
\newcommand{\R}{\mathbb{R}}
\newcommand{\Q}{\mathbb{Q}}
\renewcommand{\S}{\mathbb{S}}

\renewcommand{\epsilon}{\varepsilon}

\newcommand{\conormal}{\mathrm{n}}

\newcommand{\sdist}{\tilde d}
\newcommand{\dist}{d}
\newcommand{\distance}{\mathrm{dist}}

\newcommand{\cyl}[2]{\ensuremath{C_{#1}^{#2}}}
\newcommand{\height}{K}

\newcommand{\Class}{X}

\newcommand{\loc}{\mathrm{loc}}
\newcommand{\Rn}{\mathbb{R}^n}
\newcommand{\Trace}{\mathrm{Tr}}

\newcommand{\strictlyincluded}{\subset\subset}

\newcommand{\tr}{\mathrm{tr}}

\newcommand{\res}{\mathop{\hbox{\vrule height 7pt width 0.5pt depth 0pt
\vrule height 0.5pt width 6pt depth 0pt}}\nolimits}

\newcommand{\supp}{{\ensuremath{\mathrm{supp\,}}}}

\newcommand{\p}{\partial}
\newcommand{\esssup}{\mathop{\mathrm{ess\, sup}}}
\renewcommand{\det}{\mathop{\mathrm{det}}}
\newcommand{\diam}{\mathop{\mathrm{diam}}}
\newcommand{\diag}{\mathop{\mathrm{diag}}}
\renewcommand{\div}{\mathop{\mathrm{div}}}

\newcommand{\Lip}{\mathop{\mathrm{Lip}}}

\newcommand{\Id}{\mathop{\mathrm{Id}}}

\theoremstyle{plain}
\newtheorem{theorem}{Theorem}[section]
\newtheorem{lemma}[theorem]{Lemma}
\newtheorem{definition}[theorem]{Definition}
\newtheorem{proposition}[theorem]{Proposition}
\newtheorem{corollary}[theorem]{Corollary}

\theoremstyle{definition}
\newtheorem{example}[theorem]{Example}
\newtheorem{hypothesis}[theorem]{Hypothesis}
\newtheorem{remark}[theorem]{Remark}

\date{\today}

\begin{document}

\title{Minimizing movements for
mean curvature flow of droplets with prescribed contact angle}

\author{G. Bellettini$\!^{1,2}$, Sh.Yu. Kholmatov$\!^{2,3,4}$}  

\address{$^1$Universit\'a degli studi di Siena\\
Dipartimento di Ingegneria
dell'Informazione e Scienze Matematiche\\
via Roma 56, 53100 Siena, Italy,\\
}
\email{bellettini@diism.unisi.it}

\address{$^2$International Centre for Theoretical Physics (ICTP)\\
Strada Costiera 11, 34151 Trieste, Italy,\\
$^3$Scuola Internazionale Superiore di
Studi Avanzati (SISSA)\\
Via Bonomea 265, 34136 Trieste, Italy\\
$^4$University of Vienna Department of Mathematics\\
Oskar-Morgenstern-Platz 1, 1090, Vienna, Austria}
\email{sholmat@sissa.it}

\jot=12pt

\begin{abstract}
We study the mean curvature motion of a  droplet  flowing
by mean curvature  on a horizontal
hyperplane with a possibly nonconstant prescribed contact angle.
Using the  minimizing movements method we show the existence of
a weak evolution, and its compatibility with a  distributional
solution. We also  prove various comparison results.
\end{abstract}

\maketitle


\jot=12pt

\section{Introduction}

Historically, capillarity problems attracted attention 
because of their applications in physics, for instance in
the study of wetting phenomena \cite{CQ:05,GBQ:05}, 
energy minimizing drops and their adhesion
properties \cite{Qu:08, AS:05, DGO:07, LCAM07},
as well as because of their connections  with minimal surfaces, see e.g. 
\cite{Finn:, BB:book} and references therein.

In this paper we are interested in the study of the evolution of a droplet
flowing on a horizontal  hyperplane
under curvature driven forces
with a prescribed (possibly nonconstant) contact angle. Although there are
results in the literature describing the  static and dynamic behaviours
of droplets \cite{AS:11, SB:2001, BF:97},
not too much seems to be known concerning  their mean curvature motion.
Various results have been obtained
for mean curvature flow
of hypersurfaces with  Dirichlet boundary conditions \cite{Hui:89,Stone:90, OU:1991,OU:895}
and zero-Neumann  boundary condition \cite{AW:94,Guan:94, St:96, KKR:95}.
It is also worthwhile to recall that,
when the contact angle is constant, the evolution is related
to the so-called mean curvature flow of surface clusters, also called space partitions
(networks,  in the plane):
in two dimensions local well-posedness 
has been shown in \cite{BR:1993}, and
authors of \cite{KL:2001} derived global existence of the 
motion of grain boundaries
close to an equilibrium configuration. See also
\cite{MNT:04} for related results.
In higher space dimensions   short time existence
for symmetric partitions of space into three phases
with graph-type interfaces
 has been derived in \cite{Fr:20-2,Fr:2010}.
Very recently, authors of \cite{DGK:14} have shown 
short time existence
of the mean curvature flow of three surface clusters.

If we describe the evolving droplet by a set $E(t)\subset \Omega,$ $t\ge0$
the time,
where $\Omega = \R^n\times(0,+\infty)$ is the upper half-space in $\R^{n+1},$
the evolution problem we are interested in reads as
\begin{equation}\label{mcf}
V=H_{E(t)}\qquad \text{on $\Omega\cap \p E(t)$}
\end{equation}
where $V$ is the normal velocity and $H_{E(t)}$ is the mean curvature of  $\p E(t),$
supplied with the contact angle condition
on the contact set (the boundary of the wetted area):
\begin{equation}\label{contact_angle}
\nu_{E(t)}\cdot e_{n+1} = \beta \qquad \text{on 
$\overline{\Omega\cap\p E(t)}\cap\p\Omega,$}
\end{equation}
where $\nu_{E(t)}$ is the outer unit normal to
$\overline{\Omega\cap \p E(t)}$ at $\p\Omega,$
and $\beta:\p\Omega\to [-1,1]$
is the cosine of the prescribed contact angle.
We do not allow $\p E(t)$ to be tangent to $\p\Omega,$
i.e. we suppose $|\beta|\le 1-2\kappa$   on $\p\Omega$ for some
$\kappa\in(0,\frac12].$
Following \cite{KKR:95}, in  Appendix \ref{sec:short_time_existence}
we show local well-posedness
of
\eqref{mcf}-\eqref{contact_angle}.

Short time existence describes  the motion only up to the
first singularity time.
In order to continue the flow through singularities
one needs a notion of weak solution. Concerning the case without
boundary,
there are various notions of generalized solutions, such as
Brakke's varifold-solution \cite{Br:1978},
the viscosity solution (see \cite{Giga:06} and references therein),
the Almgren-Taylor-Wang \cite{ATW93} and Luckhaus-Sturzenhecker
\cite{LS:95}  solution,
the minimal  barrier solution
(see
\cite{Bell:2013} and references therein);
see also \cite{Il:94, ESS:1992} for other different approaches.

In the present paper we want to adapt the scheme proposed in \cite{ATW93,
LS:95}, and later extended to the notions of
{\it minimizing movement} and
{\it generalized minimizing movement} (shortly GMM)
by De Giorgi \cite{DG:93} (see also \cite{LAln,AGS:05}) to solve \eqref{mcf}-\eqref{contact_angle}.
Let us recall the definition.

\begin{definition}\label{def:general_GMM}
Let $S$ be a topological space, $F: S\times S\times [1,+\infty)\times \Z
\to[-\infty,+\infty]$ be a functional and $u:[0,+\infty)\to S.$ We say that $u$
is a generalized minimizing movement associated to $F,S$ (shortly GMM)
starting from
$a\in S$ and we write $u\in GMM(F,S,\Z,a),$ if there exist
$w:[1,+\infty)\times \Z \to S$ and a diverging sequence $\{\lambda_j\}$
such that
$$
\lim\limits_{j\to+\infty} w(\lambda_j,[\lambda_jt]) = u(t)\quad \text{for any $t\ge0,$}
$$
and the functions $w(\lambda,k),$ $\lambda\ge1,$ $k\in\Z,$ 
are defined  inductively as $w(\lambda,k)=a$ for $k\le0$ and
$$
F(w(\lambda,k+1),w(\lambda,k),\lambda,k) = 
\min\limits_{s\in S} F(s,w(\lambda,k),\lambda,k)
\qquad \forall k \geq 0.
$$
\end{definition}

If $GMM(F,S,\Z,a)$ consists of a unique element it is called
a minimizing movement starting from $a.$

In the sequel, we take $S=BV(\Omega,\{0,1\}),$  
$F=\cA:BV(\Omega,\{0,1\})\times BV(\Omega,\{0,1\})
\times [1,+\infty)\times\Z\to(-\infty,+\infty]$ defined by
$$
\cA(E,E_0,\lambda)=\cC(E,\Omega) +\lambda\int_{E\Delta E_0} \dist_{E_0}\,dx,
$$
where $E_0\in BV(\Omega,\{0,1\})$ is the initial set,
$\dist_{E_0}$ is the distance to $\Omega\cap\p E_0$ and
$$
\cC(E,\Omega) = P(E,\Omega) - \int_{\p\Omega} \beta \chi_E\,d\cH^n
$$
is the capillary functional. If
$\Omega= \Rn$ (hence when the term $\int_{\p\Omega} \beta \chi_E\,d\cH^n$
is not present),
the weak evolution (GMM) has been studied
in
\cite{ATW93}
and
\cite{LS:95},
 see also \cite{MT99} for the Dirichlet case. Further when no ambiguity appears
we use $GMM(E_0)$ to denote the GMM starting from $E_0\in BV(\Omega,\{0,1\}).$

After setting in Section
\ref{sec:preliminaries} the notation, and 
some properties of finite
perimeter sets, 
in Section \ref{sec:capillarity_functionals}
we study  the functional $\cC(\cdot,\Omega)$ and
its level-set counterpart   $\tC(\cdot,\Omega),$
including  lower semicontinuity and coercivity,
which will be useful in Section \ref{sec:comparison_principles}.
In particular,  
the map $E\mapsto \cA(E,E_0,\lambda)$
is $L^1(\Omega)$-lower semicontinuous
if and only if $\|\beta\|_\infty\le1$ (Lemma \ref{lem: lsc_of_F0}).
 Although we can also establish
the coercivity of $\cA(\cdot,E_0,\lambda)$ (Proposition
\ref{prop:coercivity_of_the_capillarity_functional}),
compactness theorems in $BV$ cannot be applied
because of the unboundedness of $\Omega.$
However, in
Theorem \ref{teo: unconstrained minimizer}
we prove  that if $E_0\in BV(\Omega,\{0,1\})$ is bounded
and $\|\beta\|_{\infty}<1,$ then
there is a minimizer in $BV(\Omega,\{0,1\})$
of $\cA(\cdot,E_0,\lambda)$, and any minimizer is bounded.
In Lemma \ref{lem:behav_big_lambda} we study the behaviour of minimizers
as $\lambda\to+\infty.$
In Proposition
\ref{prop:minimizer_of_F_0} we show
existence of constrained minimizers of $\cC(\cdot,\Omega)$,
which will be used in the proof of existence of GMMs
and in comparison principles.
In Appendix \ref{sec:error_est} we need to generalize such 
existence and uniform boundedness 
results to minimizers of functionals of type 
$\cC(\cdot,\Omega)+\cV$  under suitable hypotheses 
on $\cV. $

In Section \ref{sec:density_estimates} we study the regularity of minimizers
$\cA(\cdot,E_0,\lambda)$ (Theorem \ref{teo: regularity}). We point out
the uniform density estimates for minimizers
of $\cA(\cdot,E_0,\lambda)$ and constrained minimizers of $\cC(\cdot,\Omega)$
(Theorem \ref{teo:lower_density_est} and Proposition \ref{prop:dens_est_for_cap}),
which are the main ingredients in the existence proof of GMMs (Section
\ref{sec:existence_of_GMM}), and in the proof of  coincidence with
distributional solutions
(Section \ref{sec:weak_curvature}).

In Section \ref{sec:comparison_principles} we prove the following comparison
principle for minimizers of $\cA(\cdot,E_0,\lambda)$ (Theorem \ref{teo:E_0_and_F_0}):
{\it if  $E_0, F_0$ are bounded,
$E_0\subseteq F_0$,
$\|\beta_1\|_\infty, \|\beta_2\|_\infty<1$
and $\beta_1 \leq \beta_2,$ then
\begin{itemize}
\item[a)] there exists  a minimizer $F^*_\lambda$
of  $\cAt(\cdot,F_0,\lambda)$  containing any minimizer of
$\cAo(\cdot,E_0,\lambda);$
\item[b)] there exists  a minimizer ${E_\lambda}_*$
of  $\cAo(\cdot,E_0,\lambda)$ contained in
any minimizer of $\cAt(\cdot,F_0,\lambda);$
\end{itemize}
if in addition
$
\distance(\Omega \cap \p E_0, \Omega \cap \p F_0)>0,
$
then any minimizer $E_\lambda$ and $F_\lambda$ of
$\cAo(\cdot,E_0,\lambda)$ and
 $\cAt(\cdot,F_0,\lambda)$ respectively,
 satisfy
$
E_\lambda\subseteq F_\lambda.
$
}
As a corollary, we show that
if $E^+$ is a bounded minimizer of $\cC(\cdot,\Omega)$
in the collection $\cE(E^+)$  of all finite perimeter sets containing $E^+,$
and if
 $\|\beta\|_\infty<1$, then
for any $E_0\subseteq E^+$,
a minimizer $E_\lambda$ of $\cA(\cdot,E_0,\lambda)$
satisfies $E_\lambda\subseteq E^+$ (Proposition \ref{prop:boundedness_of_minimizers}).

In Section \ref{sec:existence_of_GMM}
we apply
the scheme in Definition \ref{def:general_GMM}
to the functional
$\cA(\cdot, E_0, \lambda)$:
as in \cite{LS:95,MSS:2016}
we build a
locally $\frac{1}{2}$-H\"older continuous
generalized minimizing movement
$t\in [0,+\infty)\mapsto E(t)\in BV(\Omega,\{0,1\})$
starting from a bounded set $E_0\in BV(\Omega,\{0,1\})$
(Theorem \ref{teo:existence_of_GMM}).
Moreover, using the results of Section \ref{sec:comparison_principles},
we prove that any GMM starting from a bounded set stays    bounded.
In general, for two GMMs one cannot expect a comparison principle
(for example in the presence of fattening).
However, the notions of {\it maximal} and {\it minimal}
GMMs (Definition \ref{def:max_min_GMM})
are always comparable if the initial
sets are comparable  (Theorem \ref{teo:comp_princ_for_GMM}).
This  requires regularity of  minimizers of $\cA(\cdot, E_0, \lambda)$
and $\cC(\cdot,\Omega),$ see Sections \ref{sec:Almgren_Taylor_Wang_functional} and
\ref{sec:density_estimates}.
Finally, in Section \ref{sec:weak_curvature} we prove that,
under a suitable conditional convergence assumption and if $1\le n\le 6,$ our
GMM solution is, in fact, a {\it distributional solution} to
\eqref{mcf}-\eqref{contact_angle}.

\section{Some preliminaries}\label{sec:preliminaries}

\subsection{Notation}

$\chi_F$ stands for the characteristic function of the
Lebesgue measurable set $F\subseteq\R^{n+1}$ and $|F|$ denotes its Lebesgue measure.
The set of $L^1(\Omega)$-functions having
bounded total variation in
an open set $\Omega\subseteq\R^{n+1}$ is denoted by $BV(\Omega),$ and
$$
BV(\Omega,\{0,1\}):= \{E\subseteq \Omega:\,\, \chi_E \in BV(\Omega)\}.
$$
Given $E\subseteq BV(\Omega,\{0,1\})$ we denote by $P(E,\Omega)$
the {\it perimeter} of $E$ in $\Omega,$ i.e.
$P(E,\Omega):=\int_\Omega|D\chi_E|,$ by $\p^*E$ the essential boundary of
$E,$ and by $\nu_E(x)$ the measure-theoretical exterior normal to $E$
at $x\in\p^*E.$
Since Lebesgue equivalent sets  in $\Omega$  have
the same perimeter in $\Omega,$  we  assume that any set $E\subset\Omega$
we consider coincides with the set
$$\left\{x\in\R^{n+1}:\,\,\lim\limits_{r\to0+}\frac{|B_r(x)\cap E|}{|B_r(x)|}=1\right\}$$
of points of density one, where $B_r(x)$ is the ball of radius $r>0$
centered at $x.$  Recall that  $\overline{\p^*E}=\p E.$ 
For simplicity, set $P(E,\R^{n+1}) = P(E).$
We say that $E\subset\R^{n+1}$ has  locally finite
perimeter in $\R^{n+1},$ if $P(E,\Omega')<+\infty$ for every
bounded open  set $\Omega'\subset\R^{n+1}.$ The collection of all
sets of locally finite perimeter is denoted by $BV_\loc(\Omega,\{0,1\}).$
We refer to \cite{Gius84, AFP:00} for a  complete information
about $BV$-functions and sets  of finite perimeter.

For a fixed nonempty $E_0\in BV(\Omega,\{0,1\})$ set
\begin{equation}\label{coveringset}
\cE(E_0):=\{E\in BV(\Omega,\{0,1\}):\,\,E_0\subseteq E\},
\end{equation}
which is $L^1(\Omega)$-closed. 

Given $\rho> 0$ and $l>0$ let $\cyl{\rho}{l}=\hat B_{\rho}\times(0,l)$  
stand for the truncated cylinder in $\R^{n+1}$ of height $l,$ 
whose basis is an open ball 
$\hat B_{\rho}\subset\R^n$ centered at the origin of radius $\rho>0;$
also set $\Omega_l: = \R^n\times(0,l).$


\subsection{Some properties of sets of finite perimeter}

By \cite[Theorem II]{DG:54-1}, for every $E\in BV_\loc(\Omega,\{0,1\})$ 
the additive set function
$O\mapsto \int_O|D\chi_E|$ defined on the open sets $O\subseteq \Omega$
extends to a measure $B\mapsto \int_B|D\chi_E|$
defined on the Borel $\sigma$-algebra of $\Omega.$
Moreover, $P(\cdot,\Omega)$ is strongly subadditive, i.e. 
\begin{equation}\label{famfor}
 P(E\cap F,\Omega)+P(E\cup F,\Omega)\le P(E,\Omega)+P(F,\Omega)
\quad \text{for any $E,F\in BV(\Omega,\{0,1\}).$}
\end{equation}

Let $\Omega$ be an open set  with Lipschitz boundary and
$E\in BV_\loc(\R^{n+1},\{0,1\}).$
We denote the  interior and exterior traces
of the set $E$ on $\p \Omega$ respectively by $\chi_E^+$ and $\chi_E^-$
and  we recall that $\chi_E^\pm\in L_\loc^1(\p \Omega).$
Moreover, the integration by parts formula holds \cite{DG:54-1}:
\begin{equation}\label{integ_by_parts}
\begin{aligned}
\int_\Omega \chi_E\div g\, dx = & -\int_\Omega g\cdot D\chi_E +
\int_{\p \Omega} (\chi_E^+- \chi_E^-) g\cdot \nu_\Omega \,d\cH^n \qquad \forall g\in
C_c^1(\R^{n+1},\R^{n+1}),
\end{aligned}
\end{equation}
where $\nu_\Omega$ is the outer unit normal to $\p \Omega.$

If $V\subseteq \Omega$ is an  open set with Lipschitz boundary, then
$$P(E,\Omega) = P(E,V) + P(E,\Omega\setminus \overline{V}) +
\int_{\Omega\cap \p V} |\chi_E^+ - \chi_E^-|\, d\cH^n.$$

The trace set of $E\subseteq \Omega$ on $\p\Omega$ is denoted by
$\Trace(E).$ With a slight abuse of  notation we set
$\chi_{\Trace(E)}=\chi_E.$
Note that
\begin{equation*} 
 P(E,\overline{\Omega}):= P(E,\Omega) + \int_{\p \Omega} \chi_E\,d\cH^n=P(E).
\end{equation*}

In general, even if $E\in BV(\Omega,\{0,1\}),$ the traces
$\chi_E^\pm$ are in  $L_\loc^1(\p\Omega),$ but not in $L^1(\p\Omega).$
For instance, if $\Omega = \Big(\R\times (0,+\infty)\Big)
\cup A \subset \R^2$  and $A=\bigcup\limits_{m=2}^{+\infty}
(m-\frac{1}{m^2},m+\frac{1}{m^2})\times(-1,0],$
then $E=A\in BV(\Omega,\{0,1\}),$ whereas $\cH^1(\Trace(E)) = +\infty.$
In Lemma \ref{lem:betacond}  we show that  
$\chi_E\in L^1(\p\Omega)$ for any   $E\in BV(\Omega,\{0,1\}),$ 
provided that $\Omega$ is a half-space.

From now on we fix $\Omega:=\R^n\times(0,+\infty);$
we often identify $\p\Omega = \R^n\times\{0\}$ with $\R^n,$ so that
$E\subset \p\Omega$ means $E\subset \R^n,$ and
$\pi:\Omega\to\p\Omega$ denotes the projection
$$\pi(\hat x,x_{n+1}):=\hat x,\quad x=(\hat x,x_{n+1})\in \Omega.$$

\subsection{Controlling the trace of a set by its perimeter}

The following lemma shows that the $L^1(\p\Omega)$-norm
of the trace of $E\in BV(\Omega,\{0,1\})$ is controlled
by $P(E,\Omega).$

\begin{lemma}\label{lem:betacond}
For any $E\in BV(\Omega,\{0,1\})$ and for
any $\beta\in L^\infty(\p\Omega)$
the relations
\begin{equation}\label{eq:proj_onto_boundary}
\left|\int_{\p\Omega} \beta\,  \chi_E\,d\cH^n \right|
\le \int_{\Omega} |\beta\circ \pi|\, |D\chi_E|\le
 \|\beta\|_\infty\, P(E,\Omega).
\end{equation}
hold. In particular, $P(E)<+\infty.$
\end{lemma}

\begin{proof}
The last inequality of \eqref{eq:proj_onto_boundary} is immediate.
The first inequality is enough to be shown for $\beta\ge0.$ 

If  $\beta$ is locally Lipschitz, then \eqref{eq:proj_onto_boundary}
follows from the divergence theorem. Indeed, suppose that $\supp(\beta)$ is compact. Since
  $\div ((\beta\circ \pi) e_{n+1})=0,$ we have
\begin{align*}
0 = \int_E \div ((\beta\circ\pi)e_{n+1})\, dx =
\int_{\Omega\cap \p^* E} (\beta\circ\pi)\, \nu_E \cdot e_{n+1}\,d\cH^n
-\int_{\p\Omega} \beta \chi_E\, d\cH^n.
\end{align*}
Hence nonnegativity of $\beta$ implies that
\begin{equation}\label{beta_lipchitz}
 \int_{\p\Omega} \beta\,\chi_E\,d\cH^n \le
\int_{\Omega\cap\p^* E} \beta\circ\pi\,d\cH^n= \int_{\Omega} \beta\circ\pi\,|D\chi_E|.
\end{equation}
If $\supp(\beta)$ is not compact,  we use $\eta_k(|x|)\beta(x)$ in \eqref{beta_lipchitz}
instead of $\beta(x),$ where $\eta_k:[0,+\infty)\to[0,+\infty)$
is Lipschitz, linear in $[k,k+1],$ $\eta_k=1$ in $[0,k]$
and $\eta_k=0$ in $[k+1,+\infty).$
Now \eqref{eq:proj_onto_boundary} follows from the monotone convergence theorem.
In particular, when $\beta\equiv1$ we have
\begin{equation*} 
P(E) = P(E,\Omega) +\int_{\p\Omega}\chi_Ed\cH^n \le 2P(E,\Omega).
\end{equation*}

Assume that $\beta=\chi_{\hat O}$ for some  open
set $\hat O\subseteq \p\Omega.$
Consider a sequence $\{\beta_k\}$  of nonnegative locally Lipschitz functions
converging $\cH^n$-almost everywhere to $\beta$ on $\p\Omega$ such that
$\beta_k\le \beta$ and
$\supp \beta_k\subseteq \overline{\hat O}.$
By Fatou's lemma we get
\begin{equation*} 
\begin{aligned}
\int_{\p\Omega} \beta\chi_Ed\cH^n \le \liminf\limits_{k\to+\infty}
  \int_{\p\Omega} \beta_k\chi_Ed\cH^n \le \liminf\limits_{k\to+\infty}
\int_{\Omega} \beta_k\circ\pi\,|D\chi_E| \le  \int_{\Omega} \beta\circ\pi\,|D\chi_E|.
\end{aligned}
\end{equation*}

Finally, if $\beta\in L^\infty(\p\Omega)$ is any nonnegative function, 
then the statement 
of the lemma follows by an approximation argument.
\end{proof}

From Lemma \ref{lem:betacond}
it follows that $E\in BV(\Omega,\{0,1\})$
if and only if $E\in BV(\R^{n+1},\{0,1\}).$

\begin{remark}
If $u\in BV(\Omega),$ then its trace belongs to $L^1(\p\Omega).$
Indeed,  it is well-known that
\begin{equation}\label{coarea_volume}
\int_{\Omega} |u|dx =  \int_{-\infty}^0 \int_{\Omega} \chi_{ \{u<t\} }(x)\,dxdt
+\int_0^{+\infty} \int_\Omega \chi_{\{u>t\}}(x)\,dxdt,
\end{equation}
\begin{equation}\label{coarea_area}
\int_\Omega |Du| =
\int_{-\infty}^0 P(\{u<t\},\Omega)\,dt
+\int_0^{+\infty} P(\{u>t\},\Omega)\,dt,
\end{equation}
in particular, $\{u>t\},\{u<s\} \in BV(\Omega)$ for a.e. $t>0$ and $s<0.$
Using \eqref{eq:proj_onto_boundary} with $\beta\equiv1,$ for a.e. $t>0$ and $s<0$
we get
$$
\int_{\p\Omega} \chi_{\{u>t\}} d\cH^n\le P(\{u>t\},\Omega),\qquad
\int_{\p\Omega} \chi_{\{u<s\}} d\cH^n\le P(\{u<s\},\Omega)
$$
and we obtain
$$
\int_{\p\Omega} |u|\,d\cH^n \le \int_\Omega|Du|.
$$
Notice that for every $\beta\in L^\infty(\p\Omega)$ one has also
\begin{equation}\label{coarea_trace}
\int_{\p \Omega} \beta u \,d\cH^n = - \int_{-\infty}^0 \int_{\p\Omega}
\beta \chi_{ \{u<t\} }\,d\cH^ndt
+\int_0^{+\infty} \int_{\p\Omega} \beta \chi_{\{u>t\} }\,d\cH^ndt.
\end{equation}
\end{remark}

The following lemma is the analog to comparison theorem 
in \cite[page 216]{LAln}\footnote{For
any $E \in BV(\R^{n+1},\{0,1\})$ and any
closed convex set $C\subseteq \R^{n+1}$
the inequality $P(E \cap C) \leq P(E)$ holds;
equality occurs if and only if $\vert E \setminus C
\vert=0$.}.

\begin{lemma}\label{muhim_corollary}
Let $E_0$ be a closed convex set such that $\nu_{E_0} \cdot e_{n+1}\ge0$ 
$\cH^n$-a.e. on $\Omega\cap \p E_0.$
 Then $P(E_0,\Omega)\le P(E,\Omega)$
for every $E\in \cE(E_0).$
\end{lemma}

\section{Capillary functionals}\label{sec:capillarity_functionals}

Let $\beta\in L^\infty(\p\Omega).$
The capillary functional $\cC(\cdot,\Omega):BV(\Omega,\{0,1\})\to\R$
and its ``level set'' version $\tC(\cdot,\Omega):BV(\Omega)\to\R$ are defined as
\begin{equation}\label{define:capil}
\cC(E,\Omega) :=
 P(E,\Omega) - \int_{\p \Omega} \beta\, \chi_E\,d\cH^n,
\end{equation}
and
\begin{equation*}
\tC(u,\Omega):=\int_\Omega |Du| - \int_{\p\Omega} \beta ud\cH^n,
\end{equation*}
respectively.
Note that $\tC(\cdot,\Omega)$ is convex,
$\tC(u,\Omega) = \mathtt{C}_{-\beta}(-u,\Omega)$ for any $u\in BV(\Omega)$, and
$\cC(E,\Omega) = \tC(\chi_E,\Omega)$ for any $E\in BV(\Omega,\{0,1\}).$
Moreover, when $\|\beta\|_\infty\le1,$ by \eqref{eq:proj_onto_boundary} the functional
$\cC(\cdot,\Omega)$ is nonnegative, and
the same holds for $\tC(\cdot,\Omega)$ as by \eqref{coarea_volume}-\eqref{coarea_trace}
one has
\begin{equation}\label{eq:coarea_J}
\tC(u,\Omega) = \int_{-\infty}^0  \mathcal{C}_{-\beta}( \{u<t\}, \Omega)\,dt
+ \int_0^{+\infty}   \cC(\{u>t\},\Omega)\,dt.
\end{equation}

The functional $\tC(\cdot,\Omega)$ will be  useful for the comparison
principles (Section \ref{sec:comparison_principles}).

\subsection{Coercivity and lower semicontinuity}

The next lemma is a localized version of \cite[Lemma 4]{LCAM07}, which is
needed to prove   coercivity of $\cC(\cdot,\Omega)$ and $\tC(\cdot,\Omega)$ and
will be frequently used in the proofs (see for example the proofs
of Theorem \ref{teo: unconstrained minimizer1} and Theorem \ref{teo:lower_density_est}).

\begin{lemma}\label{lem:Lower_bound_of_F0}
Assume that $\|\beta\|_\infty\le1$ and $E\in BV(\Omega,\{0,1\}).$
Then for any open set $A\subseteq\Omega$ with $A\in BV_\loc(\R^{n+1},\{0,1\})$ and
\begin{equation}\label{zero_perimeter_outside}
\cH^n\Big([\pi^{-1}( \pi(A) ) \setminus A]\cap \Omega \cap \p^*E\Big)=0
\end{equation}
the inequality
\begin{equation}\label{eq:lower_bound}
P(E,A) - \int_{\p\Omega}  \beta\, \chi_{E\cap A} \,d\cH^n \ge
\dfrac{1- \esssup \beta}{2}\,\left[P(E,A)+\int_{\p\Omega}
\chi_{E\cap A} \,d\cH^n \right]
\end{equation}
holds.
\end{lemma}

\begin{proof}
Let us first show that if $F\subset \Omega$ has locally finite perimeter
in $\R^{n+1},$ then
\begin{equation}\label{eq:trace_vs_projection}
\chi_F \le \chi_{\pi(F)}\quad \text{$\cH^n$-a.e. on $\p\Omega.$}
\end{equation}
Set $\hat G:= \{\hat x\in\Trace(F): \,\chi_{\pi(F)}(\hat x) =0\}.$
For any $\epsilon>0$ take an open set $\hat O\subseteq \p\Omega$ such that
$\hat G\subseteq \hat O$ and $\cH^n(\hat O\setminus \hat G)<\epsilon.$
Since $\cH^n(\pi(F)\cap \hat G) =0,$ one has
\begin{align*}
|F\cap\pi^{-1}(\hat G)|= & \int_{\pi^{-1}(\hat G)} \chi_F\,dx =
\int_0^{+\infty} dx_{n+1} \int_{\hat G} \chi_F(\hat x,x_{n+1})d\cH^n(\hat x) \\
=& \int_0^{+\infty} \cH^n(\hat G \cap \{(\hat x,0):\,\, (\hat x,x_{n+1})\in F\}) dx_{n+1}=
\int_0^{+\infty} \cH^n(\hat G \cap \pi(F)) dx_{n+1}=0.
\end{align*}
Let $\hat B_\rho\subset\R^n$ denote the ball of
radius $\rho>0$ centered at the origin.
Recall that for any $\gamma>0$ the following estimate \cite[page 35]{Gius84} holds:
$$
\int_{\hat O\cap \hat B_\rho} \chi_F d\cH^n \le
P(F,(\hat O\cap \hat B_\rho)\times(0,\gamma)) +
\frac1\gamma\int_{(\hat O\cap\hat B_\rho)\times(0,\gamma)} \chi_F\,dx.
$$
Then using $\hat G\subseteq \Trace(F),$ we establish
\begin{align*}
\cH^n(\hat G\cap \hat B_\rho) \le &
\int_{\hat O\cap \hat B_\rho} \chi_F d\cH^n
\le  P(F,(\hat O\cap \hat B_\rho)\times(0,\gamma)) \\
&+
\frac1\gamma\int_{(\hat G\cap \hat B_\rho)\times(0,\gamma)} \chi_F\,dx +
\frac1\gamma\int_{((\hat O\setminus\hat G)\cap \hat B_\rho)\times(0,\gamma)} \chi_F\,dx\\
\le& P(F,\hat O\times(0,\gamma)) + \frac1\gamma\, |F\cap \pi^{-1}(\hat G)|  +
\cH^n(\hat O\setminus \hat G)
 < P(F,\hat O\times(0,\gamma)) + \epsilon.
\end{align*}
Now letting $\epsilon,\gamma\to0^+$ we get $\cH^n(\hat G\cap\hat B_\rho) = 0$ and
\eqref{eq:trace_vs_projection} follows from letting $\rho\to+\infty.$

We have
\begin{equation}\label{periii}
\begin{aligned}
\int_\Omega \chi_{\pi(A)}\circ \pi \,\dfrac{1+\beta\circ \pi  }{2}\,|D\chi_E| = &
\int_{\pi^{-1}(\pi(A))} \dfrac{1+\beta\circ \pi }{2}\,|D\chi_E|
=  \int_A \dfrac{1+\beta\circ \pi}{2}\,|D\chi_E|,
\end{aligned}
\end{equation}
where in the second equality we used \eqref{zero_perimeter_outside}.
Moreover, from \eqref{eq:trace_vs_projection} with $F=A$
 we get
\begin{equation}\label{tracee}
\begin{aligned}
\int_{\p\Omega} \dfrac{1+\beta}{2} \chi_{E\cap A} \,d\cH^n =&
\int_{\p\Omega}  \chi_A  \,\dfrac{1+\beta}{2} \, \chi_E \,d\cH^n
\le
\int_{\p\Omega} \chi_{\pi(A)} \,\dfrac{1+\beta}{2} \,\chi_E \,d\cH^n.
\end{aligned}
\end{equation}
Now, using Lemma \ref{lem:betacond} with $\beta$ replaced with
$(1+\beta)\chi_{\pi(A)}/2,$   from \eqref{periii} and \eqref{tracee}
we obtain
\begin{equation}\label{halfilla}
 \int_{\p\Omega} \dfrac{1+\beta}{2}
 \chi_{E\cap A} \,d\cH^n\le \int_A \dfrac{1+\beta\circ \pi }{2}\,|D\chi_E|.
\end{equation}
Finally, adding the identities
\begin{align*}
P(E,A)=&\int_A|D\chi_E| = \int_A\frac{1-\beta\circ \pi }{2}\,|D\chi_E|+
\int_A \frac{1+\beta\circ\pi}{2}\,|D\chi_E|,\\
-\int_{\p \Omega} \beta\,\chi_{E\cap A} \,d\cH^n =&
\int_{\p \Omega}\frac{1-\beta}{2}\,\chi_{E\cap A} \,d\cH^n-
\int_{\p \Omega}\frac{1+\beta}{2}\,\chi_{E\cap A} \,d\cH^n,
\end{align*}
and using \eqref{halfilla} we deduce
$$
P(E,A) - \int_{\p\Omega}  \beta\,  \chi_{E\cap A} \,d\cH^n\ge
 \int_A \frac{1-\beta\circ\pi}{2}\,|D\chi_E| + \int_{\p\Omega} \frac{1-\beta}{2} \chi_{E\cap A} \,d\cH^n.
$$
This relation yields \eqref{eq:lower_bound}.
\end{proof}

\begin{proposition}[\textbf{Coercivity of the capillary functionals}]
\label{prop:coercivity_of_the_capillarity_functional}
If $-1\le \beta\le 1-2\kappa$   $\cH^n$-a.e. on $\p\Omega$ for some $\kappa\in[0,\frac12],$
then
\begin{equation}\label{eq:bound_F-0}
\kappa P(E) \le \cC(E,\Omega) \leq P(E) \qquad \forall
E\in BV(\Omega,\{0,1\}).
\end{equation}
Moreover, if $\|\beta\|_\infty\le1-2\kappa$ for some $\kappa\in [0,\frac12],$
then
\begin{equation}\label{eq:bound_J-0}
\kappa \int_{\overline\Omega} |Du| \le \tC(u,\Omega)
\le \int_{\overline\Omega}|Du|  \qquad \forall u\in BV(\Omega).
\end{equation}

\end{proposition}

\begin{proof}
The inequality $\kappa P(E) \le \cC(E,\Omega)$ follows
from Lemma \ref{lem:Lower_bound_of_F0} with $A=\Omega.$
Moreover, by virtue of Lemma \ref{lem:betacond},
\begin{equation}\label{eq:bound_F-0_zero}
\|\beta\|_\infty\le 1 \quad \Longrightarrow \quad
\cC(E,\Omega) \le   P(E)
\qquad \forall
E\in BV(\Omega,\{0,1\}).
\end{equation}
Now \eqref{eq:bound_J-0} follows from the inequalities
$$
\kappa P(\{u<t\},\Omega) +\kappa \int_{\p\Omega} \chi_{\{u<t\}}\,d\cH^n \le
\mathcal{C}_{-\beta}(\{u<t\},\Omega)
\le P(\{u<t\},\Omega)+\int_{\p\Omega} \chi_{\{u<t\}}\,d\cH^n
$$
for a.e. $t<0$ and
$$
\kappa P(\{u>t\},\Omega)
+\kappa \int_{\p\Omega}  \chi_{\{u>t\}}\,d\cH^n \le
\cC(\{u>t\},\Omega)
\le  P(\{u>t\},\Omega)+
\int_{\p\Omega} \chi_{\{u>t\}}\,d\cH^n
$$
for a.e. $t>0,$ from \eqref{coarea_volume}-\eqref{coarea_trace}, \eqref{eq:coarea_J} and
by \cite[Remark  2.14]{Gius84}, possibly after extending $u$ to $0$ outside $\Omega.$
\end{proof}

\begin{remark}\label{rem:posit_negat_part}
From the proof of Proposition \ref{prop:coercivity_of_the_capillarity_functional}
it follows that if $u\ge0,$ then \eqref{eq:bound_J-0} holds for any
$\beta\in L^\infty(\p\Omega)$  with $-1\le \beta\le 1-2\kappa;$
if $u\le 0,$ \eqref{eq:bound_J-0} is valid whenever $-1+2\kappa\le \beta\le 1.$
\end{remark}

\begin{remark}\label{emtyset_minimizer}
If $\beta>1$ on a set of infinite $\cH^n$-measure,
then $\cC(\cdot,\Omega)$ is unbounded from below.
Note also that if $\|\beta\|_\infty\le1,$ then
$\emptyset$ is the unique minimizer  of $\cC(\cdot,\Omega)$ in
$BV(\Omega,\{0,1\}).$ Indeed, clearly, 
$$
0=\cC(\emptyset,\Omega) = \min\limits_{E\in BV(\Omega,\{0,1\})} \cC(E,\Omega).
$$
If there were  a minimizer $E\ne\emptyset$ of $\cC(\cdot,\Omega),$ 
there would exist  $l>0$ such that $|E\setminus \Omega_l|>0.$ Now
since  $\Trace(E) = \Trace(E\cap\overline{\Omega_l}),$ by 
\cite[page 216]{LAln} we have
$$
0 =\cC(E,\Omega) >\cC(E\cap \overline{\Omega_l},\Omega) \ge0,
$$
a contradiction.
\end{remark}

\begin{lemma}[\textbf{Lower semicontinuity}]\label{lem: lsc_of_F0}
Assume that  $\beta\in L^\infty(\partial \Omega).$
Then the functionals
$\cC(\cdot,\Omega)$
and
$\tC(\cdot,\Omega)$
are $L^1(\Omega)$-lower semicontinuous if and only if $\|\beta\|_\infty\le1.$
\end{lemma}

\begin{proof}
Assume that $\|\beta\|_\infty\le1.$
In this case the
lower semicontinuity of $\cC(\cdot,\Omega)$ is proven in \cite[Lemma 2]{LCAM07}.
Let us prove the lower semicontinuity of $\tC(\cdot,\Omega).$ Take $u_k,u\in BV(\Omega)$ such that
$u_k\to u$ in $L^1(\Omega).$ By \eqref{coarea_volume} we may assume that
$\int_{\Omega} |\{u_k<t\}\Delta \{u<t\}|\,dx\to 0$ as  $k\to+\infty$
for a.e. $t\in\R.$ Then using the nonnegativity of summands,
the lower semicontinuity of $\cC(\cdot,\Omega)$ and
Fatou's Lemma in \eqref{eq:coarea_J} we establish
\begin{align*}
\liminf\limits_{k\to+\infty} \tC(u_k,\Omega) \ge &  \liminf\limits_{k\to+\infty}
\int_{-\infty}^0  \mathcal{C}_{-\beta}(\{u_k<t\},\Omega)\,dt
+ \liminf\limits_{k\to+\infty}\int_0^{+\infty}
 \cC( \{u_k>t\},\Omega)\,dt\\
\ge &  \int_{-\infty}^0 \liminf\limits_{k\to+\infty}
\mathcal{C}_{-\beta}(\{u_k<t\},\Omega)\,dt
+ \int_0^{+\infty}  \liminf\limits_{k\to+\infty}
 \cC(\{u_k>t\},\Omega)\,dt\\
\ge&\int_{-\infty}^0  \mathcal{C}_{-\beta}(\{u<t\},\Omega)\,dt
+ \int_0^{+\infty}   \cC(\{u>t\},\Omega)\,dt=\tC(u,\Omega).
\end{align*}

Now assume that $\|\beta\|_\infty>1,$ i.e. the set
$
\{\hat x\in \p\Omega:\,\,|\beta(\hat x)|>1\}
$
has positive $\cH^n$-measure.
Let for some $\epsilon,\delta_0>0$ the set
$\hat A:  = \{\beta>1+\epsilon\}$ satisfy $|\hat A|\ge \delta_0.$
By Lusin's theorem, for any $k>\frac{4\|\beta\|_\infty}{\epsilon\delta_0}$
there exists $\beta_k\in C(\p\Omega)$ such that
$
\cH^n(\{\beta \ne \beta_k\})<\frac1k$ and $\|\beta_k\|_\infty \le\|\beta\|_\infty.
$
Let $k$ be so large that
 $\cH^n(\{\beta_k>1+\epsilon\}) \ge \delta_0/2$  and choose an open
set
$\hat O\subset \{\beta_k>1+\epsilon\}$ of finite perimeter
such that $\delta_0/4\le \cH^n(\hat O)<+\infty.$
Define the sequence of sets $E_m: = \hat O\times (0,\frac1m)\subset \Omega.$
Clearly, $E_m\to\emptyset$ in $L^1(\Omega)$ as $m\to+\infty.$
Then, indicating by $P(\hat O)$ the perimeter of $\hat O$ in $\R^n,$ from the relations
\begin{align*}
\cC(E_m,\Omega) = & \frac1m\,P(\hat O) +\cH^n(\hat O) -\int_{\hat O} \beta d\cH^n\\
 \le &
\frac1m\,P(\hat O) +\cH^n(\hat O) - \int_{\hat O} \beta_k d\cH^n
+\int_{\hat O} |\beta - \beta_k|d\cH^n\\
\le & \frac1m\,P(\hat O) -\epsilon \cH^n(\hat O) + 2\|\beta\|_\infty
\cH^n(\hat O\cap\{\beta\ne \beta_k\}) \le  \frac1m\,P(\hat O) -\frac{\epsilon\delta_0}{4},
\end{align*}
we establish
$$
\liminf\limits_{m\to+\infty} \cC(E_m,\Omega) \le  -\frac{\epsilon\delta_0}{4} < 0 =
\cC(\emptyset,\Omega).
$$
Since $\tC(\chi_E,\Omega) = \cC(E,\Omega),$ one has also
$
\liminf\limits_{m\to+\infty} \tC(\chi_{E_m},\Omega) < 0 =
\tC(0,\Omega).
$
Hence  $\cC(\cdot,\Omega)$ and $\tC(\cdot,\Omega)$ are not $L^1(\Omega)$-lower
semicontinuous.

Finally, the case $\hat B:  = \{\beta<-1-\epsilon\}$ is of a positive 
measure can be treated in a similar way.
\end{proof}

\begin{remark}\label{rem:lsc_cap_f_omega_bounded}
If $\Omega$ is an arbitrary bounded open set with Lipschitz boundary and $\|\beta\|_\infty\le1$, 
then the
lower semicontinuity of $\cC(\cdot,\Omega)$ is a consequence of \cite[Theorem 3.4]{ADT:2015}.
In this case $\cC(\cdot,\Omega)$ is bounded from below by $-\cH^n(\p\Omega).$
Hence again Fatou's lemma and \eqref{eq:coarea_J} yield
lower semicontinuity of $\tC(\cdot,\Omega).$

\end{remark}

\section{Capillary Almgren-Taylor-Wang-type functional}\label{sec:Almgren_Taylor_Wang_functional}

In the sequel, for a given nonempty set $F\subseteq \Omega,$
$\dist_F$ stands for the distance function from the boundary of $\p F$ in $\Omega:$
$$\dist_F(x):=\distance(x,\Omega\cap \p F).$$
The function
$$\sdist_F(x):=
\begin{cases}
- \dist_F(x) &\text{if}\,\,x\in F,\\
\dist_F(x) &\text{if}\,\, x\in \Omega\setminus F,\\
\end{cases}
$$
is called the {\it signed distance function} from  $\p F$ in $\Omega$
negative inside $F.$
The distance from the empty set is assumed to be equal to $+\infty.$

Notice that for $E,F\subseteq\Omega,$ $F\ne\emptyset,$
\begin{equation*}
\begin{aligned}
\int_{E\Delta F} \dist_F\,dx=
\int_{E\setminus F} \sdist_F\,dx -
\int_{F\setminus E} \sdist_F\,dx
= \int_{E} \sdist_F\,dx -
\int_{F} \sdist_F\,dx,
\end{aligned}
\end{equation*}
provided $\displaystyle\int_{E\cap F}\dist_Fdx < +\infty.$
Moreover, we assume  $\displaystyle\int_{E\Delta F}\dist_Fdx :=0$
whenever $|E\Delta F|=0.$

Given
$\beta\in L^\infty(\p\Omega),$
$E_0\in BV(\Omega,\{0,1\})$
and
$\lambda\ge1,$ recalling the definition of
$\cC(\cdot,\Omega)$ in \eqref{define:capil},
we define the {\it capillary Almgren-Taylor-Wang-type}
functional
$\cA(\cdot,E_0,\lambda):
BV(\Omega,\{0,1\})\to [-\infty,+\infty]$
with contact angle $\beta$, as
\begin{equation}\label{eq:ATW}
\cA(E,E_0,\lambda):= \cC(E,\Omega) + \lambda\int_{E\Delta E_0} \dist_{E_0}\,dx,
\end{equation}
so that
\begin{equation}\label{representing_ATW}
\begin{aligned}
\cA(E,E_0,\lambda) = & P(E,\Omega)+\lambda\int_{E} \sdist_{E_0} \,dx -
\int_{\p \Omega} \beta\, \chi_E \,d\cH^n -
\lambda\int_{E_0} \sdist_{E_0}\,dx
\end{aligned}
\end{equation}
whenever $\displaystyle \int_{E\cap E_0}\dist_{E_0}dx < +\infty.$

\subsection{Existence of minimizers of the functional
$\cA(\cdot,E_0,\lambda)$} \label{subsec:exist_min_ATW}

We always suppose that $\lambda\ge1$ and
in this section we assume that
\begin{equation}\label{hyp:main1}
\begin{cases}
E_0\in BV(\Omega,\{0,1\}) {\rm ~is ~nonempty ~and~ bounded~},
\\
\beta\in L^\infty(\dOm) {\rm ~and}~
\exists
\kappa \in (0,\frac{1}{2}] :
-1\le \beta \le 1-2\kappa ~\text{$\cH^n$-a.e on}~\p\Omega.
\end{cases}
\end{equation}
Hence, there exists a cylinder $\cyl{D}{H}=\hat B_D\times(0,H)$ containing $E_0$
whose basis is an open ball $\hat B_D\subset\R^n$  of radius $D>0$ and height
\begin{equation*} 
H=1+\max\{x_{n+1}:\,\,x=(x',x_{n+1})\in \overline{E_0}\}.
\end{equation*}

Define
\begin{equation}\label{eq:universal constant}
R_0:=R_0(n,\kappa,E_0) = D+1+\max\Big\{8^{n^2+n+1}\,
\left(\frac{P(E_0)}{\kappa}\right)^{\frac{n+1}{n}},
\,4\mu(\kappa,n)\Big\},
\end{equation}
 where
$
\mu(\kappa,n) = \left(1/\kappa +2\right)^{\frac{n+1}{n}}.
$
The proof of the next result is essentially postponed 
to   \ref{sec:error_est}, since the main 
idea does not differ too much from \cite{LCAM07}.

\begin{theorem}[\textbf{Existence of minimizers and uniform bound}]
\label{teo: unconstrained minimizer}
Suppose that  \eqref{hyp:main1} holds. Then the minimum
problem
\begin{equation}\label{min_problem}
\inf_{E \in BV(\Omega,\{0,1\})} \cA(E,E_0,\lambda)
\end{equation}
has a solution $E_\lambda$. Moreover,
 any minimizer is contained in
$\cyl{R_0}{H}$.
\end{theorem}

\begin{proof}
Let $f=\lambda \sdist_{E_0}$ and 
$$
\cV:BV(\Omega,\{0,1\})\to(-\infty,+\infty],\quad  \cV(E):=\int_E fdx.
$$
Then $\cV$ satisfies Hypothesis \ref{hyp:2} and by Remark \ref{rem:muhim_coef}
$\cR_0\le R_0.$
Now the proof 
directly follows from Theorem  \ref{teo: unconstrained minimizer1}.
\end{proof}

\begin{remark}
If $E_0=\emptyset,$ then \eqref{min_problem} has a unique
solution $E_\lambda=\emptyset.$
Moreover, for some choices of
$\lambda\ge1$ and $\emptyset\ne E_0\in BV(\Omega,\{0,1\}),$
the empty set solves  \eqref{min_problem}. For example, let
$B_\rho$ be the ball centered at $x$ such that $x_{n+1}\ge 4\rho+4.$
If $\lambda\rho\le n,$ then  as in \cite{Cham:04,BChN:05}, one can show that
$E_\lambda=\emptyset$ is the unique minimizer of $\cA(\cdot,B_\rho,\lambda).$
\end{remark}

\begin{remark}\label{rem:unconst_min}
Let $F$ minimize  $\cA(\cdot,E_0,\lambda)$
in $BV(\cyl{R_0}{H},\{0,1\}).$ Then $F$ is an
unconstrained minimizer, i.e.
\begin{equation}\label{min_in_cyilindrin}
\cA(F,E_0,\lambda) = \min\limits_{E\in BV(\Omega,\{0,1\})} \cA(E,E_0,\lambda).
\end{equation}
Indeed, let $E_\lambda$ be any minimizer of $\cA(\cdot,E_0,\lambda).$ Clearly,
$\cA(F,E_0,\lambda) \ge \cA(E_\lambda,E_0,\lambda).$ On the other hand, by Theorem
\ref{teo: unconstrained minimizer} $E_\lambda\subseteq \cyl{R_0}{H}$ and by minimality of
$F$ in $\cyl{R_0}{H}$ we have $\cA(F,E_0,\lambda) \le \cA(E_\lambda,E_0,\lambda),$
which implies \eqref{min_in_cyilindrin}.
\end{remark}

Recalling Remark \ref{emtyset_minimizer} and definition 
\eqref{coveringset} of $\cE(E_0)$
we have also the following result.

\begin{proposition}[\bf Existence of 
constrained minimizers of $\cC$]\label{prop:minimizer_of_F_0}
Under assumptions  \eqref{hyp:main1} the constrained minimum problem
\begin{equation}\label{const.min.capil}
\inf\limits_{E\in BV(\Omega,\{0,1\}),\,E\in \cE(E_0)} \cC(E,\Omega)
\end{equation}
has a solution. In addition,  any minimizer $E^+$ 
satisfies $E^+\subseteq \cyl{R_0}{H},$
where $R_0$ is given by \eqref{eq:universal constant}, and
$E^+$ is  also a solution of
$$
\inf\limits_{E\in BV(\Omega,\{0,1\}),\,E\in \cE(E^+)} \cC(E,\Omega).
$$
\end{proposition}

\begin{proof}
Set
\begin{equation}\label{example_for_V}
\cV:BV(\Omega,\{0,1\})\to[0,+\infty],\quad 
\cV(E):=
\begin{cases}
0 & \text{if}\,\, E\in \cE(E_0),\\
+\infty & \text{if}\,\,   E\in BV(\Omega,\{0,1\})\setminus \cE(E_0).
\end{cases} 
\end{equation}
Then $\cV$ satisfies Hypothesis \ref{hyp:2} and
$\cR_0\le R_0.$ Now existence of a minimizer $E^+$
of $\cC(\cdot,\Omega)$ in $\cE(E_0)$ and the inclusion
$E^+\subseteq \cyl{R_0}{H}$ follow from Theorem 
\ref{teo: unconstrained minimizer1}.
To show  the last statement we observe that the inclusion 
$E_0\subseteq E^+$ implies $\cE(E^+)\subseteq \cE(E_0).$ 
Hence the minimality
of $E^+$ yields the inequality 
$\cC(E^+,\Omega)\le \cC(E,\Omega)$ for any $E\in \cE(E^+).$
\end{proof}

Solutions of \eqref{const.min.capil} will be called constrained
minimizers of $\cC(\cdot,\Omega)$ in $\cE(E_0).$

\begin{example}\label{exa:decreas}
Suppose that $E_0\subset\Omega$  is a closed convex set so that
$\nu_{E_0} \cdot e_{n+1}\ge 0$ $\cH^n$-a.e. on $\Omega\cap \p E_0.$
Then for every $\beta\in L^\infty(\p\Omega,[-1,0])$ the set $E_0$ is a 
constrained  minimizer of $\cC(\cdot,\Omega)$ in $\cE(E_0).$
Indeed, by Lemma \ref{muhim_corollary} 
$P(E_0,\Omega)\le P(E,\Omega)$ for all $E\in\cE(E_0),$ therefore 
$$
\cC(E,\Omega)-\cC(E_0,\Omega) = P(E,\Omega) - P(E_0,\Omega) +
\int_{\p\Omega} (-\beta) \chi_{E\setminus E_0}d\cH^n\ge0.
$$
\end{example}
The following lemma shows the behaviour of $E_\lambda$ as $\lambda\to+\infty.$
\begin{lemma}[\bf Asymptotics of $E_\lambda$ as time goes to $0^+$]\label{lem:behav_big_lambda}
Assume    \eqref{hyp:main1}  and $|\overline{E_0}\setminus E_0|=0.$
Then  any minimizer
$E_\lambda$
satisfies:
\begin{itemize}
\item[a)] $\lim\limits_{\lambda\to+\infty} |E_\lambda\Delta E_0| =0,$
\item[b)] $\lim\limits_{\lambda\to+\infty} \cC(E_\lambda,\Omega) = \cC(E_0,\Omega),$
\item[c)] $\lim\limits_{\lambda\to+\infty} \lambda \int_{E_\lambda\Delta E_0} \dist_{E_0}\,dx =0.$
\end{itemize}
\end{lemma}

\begin{proof}
a) We have
$$
\kappa P(E_\lambda) \le\cA(E_\lambda,E_0,\lambda)  \le 
\cA(E_0,E_0,\lambda) = \cC(E_0,\Omega) \le P(E_0).
$$
Moreover, from $\cA(E_\lambda,E_0,\lambda)\le P(E_0) $
and \eqref{eq:proj_onto_boundary} we get
$
\lambda\int_{E_\lambda\Delta E_0} \dist_{E_0}\,dx \le P(E_0),
$
hence
\begin{equation}\label{muha}
\lim\limits_{\lambda\to+\infty} \int_{E_\lambda\Delta E_0} \dist_{E_0}\,dx=0.
\end{equation}

Recall from Theorem \ref{teo: unconstrained minimizer} that $E_\lambda\subseteq \cyl{R_0}{H}$ for all $\lambda\ge1.$
Hence, by compactness,  from every diverging
sequence $\{\lambda_i\}$
 we can select a subsequence $\{\lambda_{i_k}\}$ such that
$$E_{\lambda_{i_k}}\to E_\infty\quad \text{in}\,\, L^1(\Omega)$$
for some $E_\infty\in BV(\cyl{R_0}{H},\{0,1\}).$
{}From \eqref{muha}
we deduce that  $\int_{E_\infty\Delta E_0} \dist_{E_0}
\,dx = 0,$ and thus, since
$\dist_{E_0}\ge0$
and 	
by assumption  $|\overline{E_0}\setminus E_0|=0,$ we get
$|E_\infty\Delta E_0|=0.$ Now arbitrariness of $\{\lambda_j\}$ implies a).

b) Clearly, $\cC(E_\lambda,\Omega)\le \cA(E_\lambda,E_0,\lambda)
\le \cC(E_0,\Omega)$ for all $\lambda\ge1.$
Then by a) and by the $L^1(\Omega)$-lower semicontinuity of $\cC(\cdot,\Omega)$
(Lemma \ref{lem: lsc_of_F0}) we establish
$$\cC(E_0,\Omega) \le \liminf\limits_{\lambda\to+\infty} \cC(E_\lambda,\Omega) \le
 \limsup\limits_{\lambda\to+\infty} \cC(E_\lambda,\Omega) \le
\cC(E_0,\Omega),$$
and b) follows.

c) follows from b) and
nonnegativity of  $\lambda \int_{E_\lambda\Delta E_0} \dist_{E_0}\,dx,$
since
$$
\limsup\limits_{\lambda\to+\infty} \lambda \int_{E_\lambda\Delta E_0} \dist_{E_0}\,dx \le
\lim\limits_{\lambda\to+\infty} [\cC(E_0,\Omega) - \cC(E_\lambda,\Omega)] = 0.
$$
\end{proof}

\section{Density estimates and regularity of minimizers}\label{sec:density_estimates}

In this section we assume that
\begin{equation}\label{hyp:main}
\begin{cases}
E_0\in BV(\Omega,\{0,1\}) {\rm ~is ~nonempty ~and~ bounded~},
\\
\beta\in L^\infty(\dOm) {\rm ~and}~
\exists
\kappa \in (0,\frac{1}{2}] :~
\|\beta\|_\infty \le 1-2\kappa.
\end{cases}
\end{equation}

Define
\begin{equation}\label{const:uniform_density_est}
R(n,\kappa):=
\left(2^{n+3}\,
\frac{\omega_n+(n+1)\omega_{n+1}}{\omega_{n+1}\kappa^{n+1} } \right)^{\frac12},\,\,
\gamma(n,\kappa):=
\frac{\kappa(n+1)}{\sqrt{R(n,\kappa)^2+4\kappa(n+1)}+R(n,\kappa) },
\end{equation}
and
\begin{equation}\label{coef_density}
 C(n,\kappa):= (n+1)\omega_{n+1}+2\omega_n + \frac{\kappa(n+1)}{2}\,\omega_{n+1},\quad
c(n,\kappa):=c_{n+1}\left(\frac{\kappa}{4}
\right)^n,
\end{equation}
where $c_{n+1}$ is the relative isoperimetric constant for the ball, i.e. 
$$
c_{n+1} \,\min\{|B_r\cap F|,|B_r\setminus F|\}^{ \frac{n}{n+1} }\le 
P(F, B_r),\qquad r>0,\,\, F\in BV(B_r,\{0,1\}).
$$
The aim of this section is to prove the following uniform density estimates for minimizers
of $\cA(\cdot,E_0,\lambda),$ needed to prove regularity
of minimizers (Theorem \ref{teo: regularity}) and Proposition \ref{prop:uniform_estimate}.

\begin{theorem}\label{teo:lower_density_est}
Assume that $E_0$ and
$\beta$ are as in  \eqref{hyp:main} and
$E_\lambda\in BV(\Omega,\{0,1\})$ is a minimizer of $\cA(\cdot, E_0,\lambda).$
Then either $E_\lambda=\emptyset$ or
\begin{equation}\label{eq:volume_est}
\left(\frac\kappa4\right)^{n+1} \le  \dfrac{|E_\lambda \cap B_r(x)|}
{\omega_{n+1}r^{n+1}} \le
1- \left(\frac\kappa4\right)^{n+1},
\end{equation}
\begin{equation}\label{eq:perimeter_est}
c(n,\kappa) \le \dfrac{P(E_\lambda, B_r(x))}{r^n}\le C(n,\kappa) 
\end{equation}
for every $x\in \p E_\lambda$ and
$r\in (0,  \frac{\gamma(n,\kappa)}{\lambda^{1/2}} ).$
In particular,
\begin{equation}\label{eq:ess_boundary}
\cH^n(\p E_\lambda\setminus\p^*E_\lambda)=0.
\end{equation}
\end{theorem}

We postpone the proof after several auxiliary results.
First we show a weaker version of Theorem \ref{teo:lower_density_est};
the difference stands in that Proposition \ref{prop:lower_density_est_weak} holds
for $r\le O(\frac{1}{\lambda})$ and $O(\frac{1}{\lambda})$
depends on $E_0,$ whereas Theorem \ref{teo:lower_density_est} is valid
for $r\le O(\frac{1}{\lambda^{1/2}})$ and $O(\frac{1}{\lambda^{1/2}})$
is independent of $E_0.$

\begin{proposition}\label{prop:lower_density_est_weak}
Under the assumptions of Theorem \ref{teo:lower_density_est},
setting
$$\Lambda: = \Lambda(\lambda,n,\kappa,P(E_0)) =
\lambda \diam(\hat B_{D+R_0+1}\times(-1,H+1)),$$
for any nonempty $E_\lambda,$ $x\in \p E_\lambda $ and
$r\in (0, \min \{1,\frac{\kappa(n+1)}{2\Lambda}\}),$
the density estimates \eqref{eq:volume_est}-\eqref{eq:perimeter_est} hold.
\end{proposition}

\begin{proof}[\bf Proof]
For completeness we give the full proof of the proposition
using the  methods of \cite{LS:95, MSS:2016}.
We recall that one could also employ the density estimates for
almost minimizers of the capillary functional (see for instance
\cite[Lemma 2.8]{PhM2014}).

Set $r_0: =\min \{1,\frac{\kappa(n+1)}{2\Lambda}\},$ and fix  $x\in \p^* E_\lambda.$
Let $B_r:=B_r(x)$ be the ball of radius $r\in (0,r_0)$ centered at $x,$ we can choose $r$
such that
$$
\cH^n(\p B_r\cap \p E_\lambda) = 0.
$$
First we show that $E_\lambda$ satisfies
\begin{equation}\label{tttyyy}
\kappa P(E_\lambda\cap B_r)
\le  2\cH^n(E_\lambda\cap \p B_r ) + \Lambda |E_\lambda\cap B_r|.
\end{equation}

Comparing $\cA(E_\lambda,E_0,\lambda)$ with
$\cA(E_\lambda\setminus B_r,E_0,\lambda),$
for a.e. $s\in(r,r_0)$  we establish
\begin{equation*} 
\begin{aligned}
P(E_\lambda, B_s\cap\Omega) -& \int_{B_r\cap \p\Omega } \beta\chi_{E_\lambda\cap B_r}\,d\cH^n
+\lambda \int_{E_\lambda\cap B_r} \sdist_{E_0}dy
\\
\le & P(E_\lambda,(B_s\setminus \overline B_r)\cap\Omega) +\cH^n(E_\lambda\cap \p B_r).
\end{aligned}
\end{equation*}
Sending $s\to r^+ $ we get
\begin{align}\label{qwqwww}
P(E_\lambda, B_r\cap\Omega) -\int_{B_r\cap \p\Omega } \beta\chi_{E_\lambda}\,d\cH^n
 + \lambda \int_{E_\lambda\cap B_r}\sdist_{E_0}dy \le & \cH^n(E_\lambda\cap \p B_r).
\end{align}
By Theorem \ref{teo: unconstrained minimizer} $E_\lambda\subseteq \cyl{R_0}{H}$
and thus, since $r_0\le 1,$ for any $y\in  B_r$
\begin{equation}\label{diam_bound}
\lambda|\sdist_{E_0}(y)|\le \lambda \diam(\hat B_{D+R_0+1}\times(-1,H+1)) =\Lambda.
\end{equation}
Moreover, using \eqref{eq:bound_F-0} for $E_\lambda\cap B_r$ we get
\eqref{tttyyy}:
\begin{align*}
\kappa P(E_\lambda\cap B_r) \le  & P(E_\lambda, B_r\cap\Omega) +\cH^n(E_\lambda\cap \p B_r)
-\int_{B_r\cap \p\Omega } \beta\chi_{E_\lambda}\,d\cH^n  \\
\le & 2\cH^n(E_\lambda\cap \p B_r ) + \Lambda |E_\lambda\cap B_r|.
\end{align*}
Now by the isoperimetric inequality,
\begin{equation}\label{isop_ineq}
P(E_\lambda\cap B_r) \ge (n+1)\omega_{n+1}^{\frac{1}{n+1}}|E_\lambda\cap B_r|^{\frac{n}{n+1}}.
\end{equation}
Set $m(r):= |E_\lambda\cap B_r|.$ Then $m $ is absolutely continuous,
$m(0)=0,$ $m(r)>0$ for all $r>0$
 and
$m'(r) = \cH^n(E_\lambda\cap\p B_r) $ for a.e. $r\in (0,r_0).$ Consequently,
 \eqref{tttyyy} and \eqref{isop_ineq} give
\begin{equation}\label{eq:diff_eq_for_de}
\kappa (n+1)\omega_{n+1}^{\frac{1}{n+1}} m(r)^{\frac{n}{n+1}} \le
2m'(r) + \Lambda m(r) = 2m'(r) + \Lambda m(r)^{\frac{n}{n+1}}m(r)^{\frac{1}{n+1}}.
\end{equation}
Since $m(r)\le \omega_{n+1}r^{n+1}$ and $r\le \frac{\kappa(n+1)}{2\Lambda},$
from the last inequality we obtain
$$
\frac\kappa4\, (n+1)\omega_{n+1}^{\frac{1}{n+1}} m(r)^{\frac{n}{n+1}} \le
m'(r).
$$
Integrating we get the lower volume density estimate
$$
m(r)\ge \left(\frac\kappa4\right)^{n+1} \omega_{n+1}r^{n+1},\quad \forall r\in (0,r_0).
$$

Let us  prove the upper volume density estimate  in  \eqref{eq:volume_est}.
Since $E_\lambda\subseteq\Omega$ if $x\in\p\Omega\cap\p^*E_\lambda,$  
the inequality
\begin{equation}\label{trivial_ineq}
\frac{|B_r\setminus E_\lambda|}{\omega_{n+1}r^{n+1}} 
\ge \frac12 > \left(\frac\kappa4\right)^{n+1}\qquad \forall r>0 
\end{equation}
is trivial.
So assume that $x\in \Omega\cap \p^*E_\lambda.$ Since
$\cA(E_\lambda,E_0,\lambda)\le\cA((E_\lambda\cup B_r)\cap\Omega,E_0,\lambda),$
arguing as in the proof of \eqref{qwqwww} we get
\begin{align}\label{ungtomon}
P(E_\lambda, B_r\cap\Omega) + \int_{\p\Omega } \beta\chi_{(B_r\cap \Omega)\setminus E_\lambda}\,d\cH^n
\le & \cH^n((\Omega\setminus E_\lambda)\cap \p B_r) + \lambda
\int_{(B_r\cap \Omega)\setminus E_\lambda}\sdist_{E_0}dy.
\end{align}
From the  isoperimetric inequality, \eqref{eq:bound_F-0},
\eqref{ungtomon} and also  \eqref{diam_bound},
it follows that
\begin{equation}\label{qaydantopay}
\begin{aligned}
\kappa(n+1)\omega_{n+1}^{\frac{1}{n+1}}&|(B_r\setminus E_\lambda)\cap\Omega|^\frac{n}{n+1} \le
\kappa P((B_r\setminus E_\lambda)\cap\Omega)\le
\mathcal{C}_{-\beta}((B_r\setminus E_\lambda)\cap\Omega,\Omega)\\
\le &P(E_\lambda, B_r\cap\Omega) + \int_{\p\Omega } \beta\chi_{(B_r\cap \Omega)\setminus E_\lambda}\,d\cH^n
+ \cH^n((\Omega\setminus E_\lambda)\cap \p B_r)\\
\le & 2\cH^n((\Omega\setminus E_\lambda)\cap \p B_r) + \Lambda |(B_r\setminus E_\lambda)\cap \Omega|.
\end{aligned}
\end{equation}
Repeating the same arguments as before we establish
$$
\frac{|B_r\setminus E_\lambda|}{\omega_{n+1}r^{n+1}} \ge
\frac{|(B_r\setminus E_\lambda)\cap\Omega|}{\omega_{n+1}r^{n+1}} \ge \left(\frac\kappa4\right)^{n+1}
\qquad \forall r\in (0,r_0).
$$

Let us now show \eqref{eq:perimeter_est}. From \eqref{qwqwww} we get
\begin{align*}
P(E_\lambda,B_r) = & P(E_\lambda, B_r\cap\Omega)  +\int_{B_r\cap\p\Omega}\chi_{E_\lambda}\,d\cH^n\\
\le & \cH^n(E_\lambda\cap \p B_r) +\int_{B_r\cap \p\Omega} (1+\beta) \chi_{E_\lambda}\,d\cH^n
 + \Lambda |E_\lambda\cap B_r|\\
\le &(n+1)\omega_{n+1}r^n +2\omega_nr^n +\omega_{n+1}r^n (\Lambda r)\\
\le & \left[(n+1)\omega_{n+1}+2\omega_n + \omega_{n+1}\frac{\kappa(n+1)}{2}\right] r^{n}
\end{align*}
for a.e $r\in (0,r_0).$ Since $P(E_\lambda,\cdot)$ is a nonnegative measure,
this inequality holds for all $r\in (0,r_0).$ This proves the upper
perimeter estimate in \eqref{eq:perimeter_est}.

The lower perimeter density estimate in
\eqref{eq:perimeter_est} follows from \eqref{eq:volume_est} and the
relative isoperimetric
inequality (see for example \cite[page 152]{AFP:00}).
\end{proof}

\begin{theorem}[\bf Regularity of minimizers up to the
boundary]\label{teo: regularity}
Assume that $E_0$ and
$\beta$ satisfy \eqref{hyp:main}. Then any nonempty minimizer
$E_\lambda$
is   open  in $\R^{n+1}$ and
$\Omega\cap \p^*E_\lambda$ is
an $n$-dimensional manifold of class   $C^{2,\alpha}$ for a suitable
$\alpha\in(0,1)$,
and
$\cH^s((\p E_\lambda\setminus \p^*E_\lambda)\cap \Omega)=0$ for all $s>n-7.$
Moreover, if $\beta\in \Lip(\p\Omega),$ then
\begin{itemize}
\item[a)] $\cH^n((\p E_\lambda\cap\p\Omega)\Delta(\Trace(E_\lambda)))=0;$
\item[b)] $\p E_\lambda\cap \p\Omega$ is a set of  finite perimeter in $\p\Omega$
and
$$
\cH^{n-1}(\p (\p E_\lambda\cap\p\Omega)\setminus \p^*(\p E_\lambda\cap\p\Omega))=0,
$$
where $\p (\p E_\lambda\cap\p\Omega)$ denotes  the
 boundary of $\p E_\lambda\cap\p\Omega$ in $\p\Omega.$
Moreover,  if $M_\lambda=\overline{\Omega\cap \p E_\lambda},$ then
\begin{equation*}  
\p (\p E_\lambda\cap\p\Omega) = M_\lambda \cap \p\Omega.
\end{equation*}
\item[c)] There exists a relatively closed set $\Sigma\subset M_\lambda$ with
 $\cH^{n-1}(\Sigma\cap\p\Omega)=0$ such that
in a neighborhood of any $x\in (M_\lambda\cap \p\Omega)
\setminus \Sigma$ the set $M_\lambda$ is a $C^{1,1/2}$-manifold
with boundary, and
$$\nu_{E_\lambda} \cdot e_{n+1} =  \beta\quad \text{on}\,\,
 (M_\lambda\cap \p \Omega)\setminus\Sigma.$$
\end{itemize}

\end{theorem}

\begin{proof}
Since $E_\lambda$ is a minimizer of $\cA(\cdot,E_0,\lambda)$
in every ball $B\subset \Omega,$  we can apply \cite[Theorem 5.2]{Mas74}
to prove  that $E_\lambda$ is  open and
$\Omega\cap \p^*E_\lambda$ is $C^{2,\alpha}$  with
$\cH^s((\p E_\lambda\setminus \p^*E_\lambda)\cap \Omega)=0$ for all $s>n-7.$
Moreover, if $\beta\in \Lip(\p\Omega),$ by \eqref{diam_bound}
the remaning assertions  follow from \cite[Lemma 2.16,  Theorem 1.10]{PhM2014}.
\end{proof}

\begin{remark}\label{rem:uncons_dens_est}
(Compare with \cite[Remark 1.4]{LS:95} and \cite{MSS:2016}.)

a) Assume that $x\in \overline{E_\lambda}$ and $r>0$ are such that $B_r(x)\cap E_0 = \emptyset.$
Then $\dist_{E_0}\ge 0$ in $E_\lambda\cap B_r(x)$ and from \eqref{qwqwww}
we get
\begin{align}\label{eq:d_est}
P(E_\lambda, B_r\cap\Omega) -\int_{B_r\cap \p\Omega} \beta\chi_{E_\lambda}\,d\cH^n
\le & \cH^n(E_\lambda\cap \p B_r).
\end{align}
Then proceeding as in the proof of Proposition  \ref{prop:lower_density_est_weak} we get
$
|E_\lambda\cap B_r| \ge\left(\kappa/2\right)^{n+1}\omega_{n+1}r^{n+1}.
$
Moreover, from \eqref{eq:d_est} it follows that
$$
P(E_\lambda,B_r\cap\Omega) \le \cH^n(E_\lambda\cap \p B_r)+
\int_{B_r\cap \p\Omega}\chi_{E_\lambda}\,d\cH^n\le
\big[(n+1)\omega_{n+1} + \omega_n\big]r^n.
$$

b) Similarly, if $x\in \overline{ E_\lambda}$ and $B_r(x)\cap (\Omega\setminus E_0) = \emptyset,$
then
$
|B_r\setminus E_\lambda| \ge\left(\kappa/2\right)^{n+1}\omega_{n+1}r^{n+1}.
$

Observe that in both cases $r$ {\it need not  be} in
 $(0,\min \{1,\frac{\kappa(n+1)}{2\Lambda}\})$
and the assumption $x\in \p E_\lambda$ is not necessary.
\end{remark}

The following proposition is the analog  of \cite[Lemma 2.1]{LS:95} and \cite[Proposition 3.2.1]{MSS:2016}.

\begin{proposition}[\bf $L^\infty$-bound for
the distance function]\label{prop:uniform_L_infty_est}
Assume that $E_0$ and
$\beta$ are as in  \eqref{hyp:main} and
$E_\lambda\in BV(\Omega,\{0,1\})$ is a minimizer of $\cA(\cdot, E_0,\lambda).$ Then
\begin{equation}\label{eq:uniform_bound_dist}
\sqrt{\lambda} \|\dist_{E_0}\|_{L^\infty(E_\lambda\Delta E_0)} \le R(n,\kappa).
\end{equation}
\end{proposition}

\begin{proof}
Let  $R:=  R(n,\kappa).$ Suppose by contradiction that
there exist $\epsilon>0,$  $\lambda\ge1$ and
$x\in E_\lambda\Delta E_0$
such that $\dist_{E_0}(x) > (R+\epsilon)\lambda^{-1/2}.$
Consider first the case $x\in  E_\lambda\setminus E_0.$
By regularity of $E_\lambda$
(Theorem \ref{teo: regularity}) we may assume that
$x\in \p E_\lambda\setminus E_0.$ Note that $B_\rho \cap E_0 =\emptyset,$
where  $B_\rho:=B_\rho(x),$ $\rho=(R+\epsilon)\lambda^{-1/2}/2.$
Since $\cA(E_\lambda,E_0,\lambda) \le
\cA(E_\lambda\setminus B_\rho,E_0,\lambda),$
and $\sdist_{E_0}(y)=\dist_{E_0}(y)\ge \rho $ for any $y\in B_\rho\cap E_\lambda,$
from \eqref{qwqwww} we establish
$$
\frac{(R+\epsilon)\lambda^{1/2}}{2}|E_\lambda\cap B_\rho| \le
\lambda\int_{E_\lambda\cap B_\rho} \sdist_{E_0}dy \le \cH^n(E_\lambda\cap \p B_\rho)
+\int_{B_\rho\cap \p\Omega} \beta\chi_{E_\lambda}\,d\cH^n\le
 [\omega_{n+1}(n+1)+\omega_n]\rho^n.
$$
This and Remark \ref{rem:uncons_dens_est} (a) yield\footnote{
Since the upper bound for the radii in Proposition
\ref{prop:lower_density_est_weak} is of order $O(\frac{1}{\lambda}),$
in general, we cannot apply it with $\rho.$}
$$
\omega_{n+1}\,\frac{(R+\epsilon)\kappa^{n+1}}{2^{n+2}} \lambda^{1/2} \rho^{n+1} \le [\omega_{n+1}(n+1)+\omega_n]\rho^n,
$$
or equivalently, recalling the definition of $\rho$
$$
(R+\epsilon)^2 \le 2^{n+3}\,
\frac{\omega_n+(n+1)\omega_{n+1}}{\omega_{n+1}\kappa^{n+1} } = R^2,
$$
which is a contradiction.
A similar contradiction is obtained
when $x\in E_0\setminus E_\lambda.$
\end{proof}

\begin{corollary}\label{cor:behav_big_lambda}
Assume    \eqref{hyp:main1}  and $|\overline{E_0}\setminus E_0|=0.$
If $\|\beta\|_\infty<1,$ then any minimizer
$E_\lambda$
satisfies $\overline{\Omega\cap \p E_\lambda}
\overset{K}{\to} \overline{\Omega\cap \p E_0}$ as $\lambda\to+\infty,$ 
where $\overset{K}{\to}$ denotes
 Kuratowski convergence \cite{Kurat:66}. 
\end{corollary}

\begin{proof}
It suffices to show that every diverging sequence $\{\lambda_j\}$
has a subsequence $\{\lambda_j'\}$ such that  
$$K-\lim\limits_{j\to+\infty} \overline{\Omega\cap \p E_{\lambda_j'}} =
\overline{\Omega\cap \p E_0}.$$
Choose any sequence $\lambda_j\to+\infty.$ By compactness of closed
sets in Kuratowski convergence \cite[page 340]{Kurat:66}, 
there exists a closed
set $C \subset \overline \Omega$
such that up to a not relabelled subsequence
$\overline{\Omega\cap \p E_{\lambda_j}}
\overset{K}{\to} C$
as $j\to+\infty.$
Let us show first that $\overline{\Omega\cap\p E_0}\subseteq C.$
Take any $x\in \R^{n+1}\setminus C$; we may suppose that $x\in\Omega.$
Since $C$ is closed, there exists a ball $B_\rho(x)$ such that
$B_\rho(x)\cap C=\emptyset.$ Since
$\overline{\Omega\cap \p E_{\lambda_j}}\overset{K}{\to} C$ as $j\to+\infty,$
we have $B_\rho(x)\cap \overline{\Omega\cap\p E_{\lambda_j}}=\emptyset$
for $j\ge1$ large enough. Therefore, $P(E_{\lambda_j}, B_\rho(x)\cap\Omega)=0,$
and by a) and lower semicontinuity,  $P(E_0, B_\rho(x)\cap\Omega)=0.$
This yields $B_{\rho/2}(x)\cap \overline{\Omega\cap \p E_0}=\emptyset$
and thus $\R^{n+1}\setminus C\subseteq
\R^{n+1}\setminus \overline{\Omega\cap \p E_0}.$

Now suppose that there exists $x\in C\setminus\overline{\Omega\cap \p E_0}.$
Then there exists $\rho>0$ such that
$B_\rho(x)\cap \overline{\Omega\cap \p E_0}=\emptyset.$ Since $x\in C,$
there exists $x_j\in \overline{\Omega\cap \p E_{\lambda_j}}$ such that
$x_j\to x.$ Choose $j\in \mathbb N$ so large that $x_j\in B_{\rho/4}(x)$ and
$R(n,\kappa)\lambda_j^{-1/2} <\rho/4,$ where $R(n,\kappa)$
is defined in \eqref{const:uniform_density_est}.
By Proposition \ref{prop:uniform_L_infty_est}, we have
$$
\dist_{E_0}(x_j)\le
R(n,\kappa)\lambda_j^{-1/2} <\frac{\rho}{4}.
$$
On the other hand, by construction, $\dist_{E_0}(x) \ge \frac{3\rho}{4},$ 
which leads to a contradiction. This yields 
$C\subseteq\overline{\Omega\cap \p E_0}$.
\end{proof}

\begin{proof}[\bf Proof of Theorem \ref{teo:lower_density_est}]
We repeat the same procedures of the proof of
Proposition \ref{prop:lower_density_est_weak}
with improved estimates for the volume term of
$\cA(\cdot,E_0,\lambda).$
Let $R:=R(n,\kappa),$ $\gamma:=\gamma(n,\kappa).$
Fix  $x\in \p^* E_\lambda,$ and choose
$r\in (0,\gamma\lambda^{-1/2})$  such that
$
\cH^n(\p B_r\cap \p E_\lambda) = 0.
$
From \eqref{eq:uniform_bound_dist} it follows
$$
\sup\limits_{(E_\lambda\setminus E_0)\cap B_r} \dist_{E_0} \le R\lambda^{-1/2}.
$$
Therefore, using the obvious inequality
\begin{align*}
\sup\limits_{(E_\lambda\cap E_0)\cap B_r} \dist_{E_0}
\le 2r+\sup\limits_{(E_0\setminus E_\lambda)\cap B_r}
\dist_{E_0} \le (2\gamma+R) \lambda^{-1/2},
\end{align*}
from \eqref{qwqwww} we  establish  that
\begin{align}\label{23w}
P(E_\lambda, B_r\cap\Omega) -\int_{B_r\cap \p\Omega }
\beta\chi_{E_\lambda}\,d\cH^n
 \le & \cH^n(E_\lambda\cap \p B_r) + (R + 2\gamma)
 \lambda^{1/2}|E_\lambda\cap B_r|.
\end{align}
Since $m(r):=|E_\lambda\cap B_r|\le \omega_{n+1}r^{n+1}$ and
$r\le \frac{\gamma}{\lambda^{1/2}},$
similarly to \eqref{eq:diff_eq_for_de} from \eqref{23w} we deduce
$$
\kappa(n+1)\omega_{n+1}^{\frac{1}{n+1}} m(r)^{\frac{n}{n+1}} \le 2m'(r) +
(R+2\gamma)\lambda^{1/2}r \omega_{n+1}^{\frac{1}{n+1}}
m(r)^{\frac{n}{n+1}},\,\, \text{for a.e. $r\in (0,\gamma\lambda^{1/2}).$}
$$
By the definition of $\gamma$ one has
$$
(R+2\gamma)\lambda^{1/2}r\le  (R + 2\gamma)\gamma=
\frac12\,\kappa(n+1).
$$
Thus,
$$
\frac{\kappa}{4}\,(n+1)\omega_{n+1}^{\frac{1}{n+1}} m(r)^{\frac{n}{n+1}}
\le m'(r)\quad
\text{for a.e. $r\in (0,\gamma\lambda^{-1/2}).$}
$$
Integrating this differential inequality we get the lower volume density estimate
in \eqref{eq:volume_est}.

Let us prove  the upper volume density estimate in \eqref{eq:volume_est}.
Due to  \eqref{trivial_ineq}
we may suppose that $x\in \Omega\cap\p^*E_\lambda.$
As above  one can estimate $\dist_{E_0}$ in
$(B_r\setminus E_\lambda)\cap \Omega$
as follows:
\begin{align}\label{kuchli_baho}
\sup\limits_{\Omega\cap ((B_r\setminus E_\lambda)\setminus E_0)}
\dist_{E_0} \le 2r + \sup\limits_{E_\lambda\Delta E_0} \dist_{E_0} \le (2\gamma+R) \lambda^{-1/2}.
\end{align}
Since $\sdist_{E_0}\le 0$ in $\Omega\cap ((B_r\setminus E_\lambda)\cap E_0),$
plugging \eqref{kuchli_baho}  in \eqref{ungtomon} and proceeding as above
we establish
$$
\frac{\kappa}{4}\,(n+1)\omega_{n+1}^{\frac{1}{n+1}}|(B_r\setminus
E_\lambda)\cap\Omega|^{\frac{n}{n+1}} \le \cH^n((\Omega\setminus E_\lambda)\cap B_r),
$$
from which the upper volume density estimates in \eqref{eq:volume_est} follows.

The proof of \eqref{eq:perimeter_est} is exactly the same as
the proof of perimeter density estimates in Proposition
\ref{prop:lower_density_est_weak}.
Finally, \eqref{eq:ess_boundary} is a standard consequence of a
covering argument.
\end{proof}

Let us prove the following $L^1$-estimate for the minimizers of
$\cA(\cdot,E_0,\lambda),$
the analog of \cite[Lemma 1.5]{LS:95} and
\cite[Proposition 3.2.3]{MSS:2016}. Notice carefully the exponent $-1/2$ of $\lambda$
in \eqref{eq:uniform_estimate}.

\begin{proposition}[\bf $L^1$-estimate]\label{prop:uniform_estimate}
Assume that $E_0$  and $\beta$ satisfy  \eqref{hyp:main} and the uniform
volume density estimates \eqref{eq:volume_est} holds for $E_0.$
Then for any minimizer $E_\lambda$ of
$\cA(\cdot, E_0,\lambda)$   the estimate
\begin{equation}\label{eq:uniform_estimate}
|E_\lambda\Delta E_0| \le C_{n,\kappa} P(E_0)\, \ell +
\frac{1}{\ell}\int_{E_\lambda\Delta E_0}\dist_{E_0}dx,
\quad \ell \in \left(0,
\frac{\gamma(n,\kappa)}{\lambda^{1/2}}\right)
\end{equation}
holds, where
\begin{equation}\label{3452}
C_{n,\kappa}:=\left(\frac{8}{\kappa}\right)^{n+1}  \omega_{n+1}^{\frac{1}{n+1}}\,\mathfrak{b}(n)\,c_{n+1}
\end{equation}
and
$\mathfrak{b}(n)$ is the constant in Besicovitch covering theorem.
\end{proposition}

\begin{proof}
Set
$$A:=\{x\in E_\lambda\Delta E_0:\,\dist_{E_0}(x) \ge \ell\},\quad
B:=\{x\in E_\lambda\Delta E_0:\,\dist_{E_0}(x) < \ell\}.$$
By Chebyshev inequality
$$
|A| \le \frac{1}{\ell}
 \int_{E_\lambda\Delta E_0}\dist_{E_0}dx.
$$
Let us estimate $|B|.$ Since $E_0$ is bounded, by Besicovitch's covering theorem
there exist at most countably many balls
$\{B_{\ell}(x_i)\},$ $x_i\in \p E_0$  such that
any point of  $\p E_0$ belongs to at most $\mathfrak{b}(n)$ balls,
$\p E_0\subset \bigcup\limits_i B_{\ell}(x_i)$  and  $B\subset \bigcup\limits_i B_{2\ell}(x_i).$
Since the balls $\{B_{2\ell}(x_i)\}$ cover $B,$
by the density estimates \eqref{eq:volume_est} and the
relative isoperimetric inequality   we get
\begin{align*}
|B_{2\ell}(x_i)| = &2^{n+1} \omega_{n+1} \ell^{n+1} \le 2^{n+1}\left(\frac{4}{\kappa}\right)^{n+1}
\min\{|B_\ell(x_i)\cap E_0|, |B_\ell(x_i)\setminus E_0|\} \\
\le & \left(\frac{8}{\kappa}\right)^{n+1} \omega_{n+1}^{\frac{1}{n+1}}\,\ell\,
\min\{|B_\ell(x_i)\cap E_0|, |B_\ell(x_i)\setminus E_0|\}^{\frac{n}{n+1}}\\
\le &  \left(\frac{8}{\kappa}\right)^{n+1}  \omega_{n+1}^{\frac{1}{n+1}}\,\ell\,c_{n+1} \,
P(E_0,B_\ell(x_i)).
\end{align*}
Therefore
\begin{align*}
|B| \le &  \left(\frac{8}{\kappa}\right)^{n+1}  \omega_{n+1}^{\frac{1}{n+1}}\,c_{n+1}\,
\ell \sum\limits_{i} P(E_0,B_\ell(x_i)) \le
 \left(\frac{8}{\kappa}\right)^{n+1}  \omega_{n+1}^{\frac{1}{n+1}}\,\mathfrak{b}(n)\,c_{n+1}\,P(E_0)\,\ell.
\end{align*}
Now   \eqref{eq:uniform_estimate} follows from the  estimates
for $|A|,|B|$ and from
$
|E_\lambda\Delta E_0| \le |A|+|B|.
$
\end{proof}

A specific choice of $\ell$ will be made in the proof of Theorem \ref{teo:existence_of_GMM}.
We conclude this section with a proposition about
the regularity of minimizers of $\cC(\cdot,\Omega).$

\begin{proposition}[\bf Density estimates for constrained
minimizers of $\cC$]\label{prop:dens_est_for_cap}
Assume that $E_0$ and $\beta$ satisfy \eqref{hyp:main} and
there exist  $c_1,c_2,\epsilon\in (0,1)$ such that for
every $x\in \p E_0$ and $r\in (0,\epsilon)$
the inequalities
\begin{align*}
c_1\le \frac{|B_r(x)\cap E_0|}{|B_r(x)|} \le c_2
\end{align*}
hold.
Let  $E^+$ be a constrained minimizer of $\cC(\cdot,\Omega)$ in $\cE(E_0).$
Then for every $x\in \p E^+$ and $r\in (0,\epsilon) $
\begin{equation}\label{eq:dens_est_cap_f}
\begin{gathered}
c_1 \left(\frac{\kappa}{8}\right)^{n+1}\le \frac{|B_r(x)\cap E^+|}{|B_r(x)|}\le
1 - \left(\frac{\kappa}{4}\right)^{n+1},\\
c_{n+1} c_1^{\frac{n}{n+1}} (\kappa/8)^n  \le
\frac{P(E^+,B(x,r))}{r^n}\le (n+1)\omega_{n+1} +\omega_n.
\end{gathered}
\end{equation}
In particular,
$\cH^n(\p E^+\setminus \p^*E^+)=0.$
\end{proposition}

\begin{proof}
Let $x\in \p E^+,$ and $r\in (0,\epsilon)$ be such that
$\cH^n(\p B_r\cap \p ^*E^+)=0,$ where $B_r:=B_r(x).$

We start with the upper volume density estimate  in \eqref{eq:dens_est_cap_f}.
We may suppose $x\in \Omega\cap \p^*E^+,$
since the case $x\in \p\Omega\cap\p^* E^+$ is trivial.
Using  $\cC(E^+,\Omega)\le \cC((E^+\cup B_r)\cap\Omega,\Omega),$
as in \eqref{ungtomon} we establish
\begin{equation}\label{taqqoslash_nat}
P(E^+,B_r) + \int_{\p\Omega} \beta\chi_{(B_r\setminus E^+)\cap \Omega}\,d\cH^n
\le \cH^n((\Omega\setminus E^+)\cap \p B_r).
\end{equation}
Adding $\cH^n(\p B_r\cap (\Omega\setminus E^+))$  to
both sides and proceeding as in \eqref{qaydantopay}
we get
$$
\kappa(n+1)\omega_{n+1}^{\frac{1}{n+1}}|(B_r\setminus E^+)\cap\Omega|^{\frac{n}{n+1}} \le
2\cH^n((\Omega\setminus E^+)\cap \p B_r)
$$
and hence as in the proof of Theorem \ref{teo:lower_density_est}
$$
|B_r\setminus E^+|\ge \left(\frac{\kappa}{4}\right)^{n+1}\,\omega_{n+1}r^{n+1}.
$$
This implies the upper volume density estimate in \eqref{eq:dens_est_cap_f}.

The lower volume density estimate is a little delicate, since in general we cannot
use the set $E=E^+\setminus B_r$ as a competitor since it need not belong to $\cE(E_0).$
If $d:=\dist_{E_0}(x)=0,$ then $x\in \p E_0$ and, hence, using $E_0\cap B_r\subset E^+\cap B_r$
and the lower volume density estimate for $E_0$ we establish
$$
\frac{|E^+\cap B_r|}{|B_r|} \ge \frac{|E_0\cap B_r|}{|B_r|} \ge c_1\ge
c_1 \left(\frac{\kappa}{8}\right)^{n+1}.
$$
If $d>0$ and $r\in (0,\min\{\epsilon,d\}),$ then
we may use comparison set $E^+\setminus B_r$ and
as in the proof of \eqref{eq:volume_est}
we obtain
\begin{equation}\label{nimadirf}
\frac{|E^+\cap B_r|}{|B_r|} \ge \left(\frac{\kappa}{4}\right)^{n+1} \ge
c_1 \left(\frac{\kappa}{8}\right)^{n+1}.
\end{equation}
Suppose $d<\epsilon.$
Since one can extend \eqref{nimadirf} to $(0,d]$ by continuity,
if $r\in (d,\min\{2d,\epsilon\}),$  then
$$
\frac{|E^+\cap B_r|}{|B_r|} \ge \frac{|E^+\cap B_d|}{|B_d|}\cdot \left(\frac{d}{r}\right)^{n+1} \ge
\left(\frac{\kappa}{8}\right)^{n+1}\ge c_1\left(\frac{\kappa}{8}\right)^{n+1}.
$$
Let $r\in [2d,\epsilon)$ and
 $x_0\in \overline{\Omega\cap \p E_0}$
be such that $d = |x-x_0|.$ Then using $B(x,r)\supset B(x_0,r-d),$
the lower density estimate for $E_0$ and $r-d\ge r/2,$ we obtain
$$
\frac{|E^+\cap B_r|}{|B_r|} \ge \frac{|E_0\cap B_{r-d}(x_0)|}{|B_{r-d}(x_0)|}
\cdot\left(\frac{r-d}{r}\right)^{n+1}
\ge c_1 \left(\frac{1}{2}\right)^{n+1} \ge c_1 \left(\frac{\kappa}{8}\right)^{n+1}.
$$

Now the lower perimeter estimate follows from the
volume density estimates and the relative isoperimetric inequality.
The upper perimeter estimate is obtained from \eqref{taqqoslash_nat}:
$$
P(E^+,B_r) \le \cH^n((\Omega\setminus E^+)\cap \p B_r) - \int_{\p\Omega}
\beta\chi_{(B_r\setminus E^+)\cap \Omega}\,d\cH^n \le ((n+1)\omega_{n+1} +\omega_n)r^n.
$$

Finally, the relation $\cH^n(\p E^+\setminus \p^*E^+)=0$ is a consequence of the density estimates together
with a covering argument.
\end{proof}

\section{Comparison principles}\label{sec:comparison_principles}

The main result of this section is the following comparison
between  minimizers of $\cA(\cdot,E_0,\lambda).$

\newcommand{\cBo}{\mathcal{B}_{\beta_1}}
\newcommand{\cBt}{\mathcal{B}_{\beta_2}}

\begin{theorem}[\textbf{Comparison  for minimizers
of $\cA$}]\label{teo:E_0_and_F_0}
Assume that   $E_0, F_0,$ $\beta_1, \beta_2$ satisfy
\eqref{hyp:main1}. Suppose that
$E_0\subseteq F_0$ and $\beta_1 \leq \beta_2.$
Then
\begin{itemize}
\item[a)] there exists  a minimizer ${F}^*_\lambda$
of  $\cAt(\cdot,F_0,\lambda)$  containing any minimizer of
$\cAo(\cdot,E_0,\lambda);$
\item[b)] there exists  a minimizer ${E_{\lambda}}_*$
of  $\cAo(\cdot,E_0,\lambda)$ contained in
any minimizer of $\cAt(\cdot,F_0,\lambda).$
\end{itemize}
If in addition
\begin{equation}\label{eq:strict_inclusion}
\distance(\Omega \cap \p E_0, \Omega \cap \p F_0)>0,
\end{equation}
then all minimizers $E_\lambda$ and $F_\lambda$ of
$\cAo(\cdot,E_0,\lambda)$ and
 $\cAt(\cdot,F_0,\lambda)$ respectively
 satisfy %
\begin{equation*} 
E_\lambda\subseteq F_\lambda.
\end{equation*}
\end{theorem}

\begin{remark}
We do not exclude the case that either $E_\lambda$ or $F_\lambda$ is empty.

\end{remark}

\begin{remark}\label{rem:exis_max_min}
For any $E_0,$ $\beta$ satisfying   \eqref{hyp:main1}, using Theorem \ref{teo:E_0_and_F_0}
with $\beta_1=\beta_2=\beta$ and $F_0=E_0,$
we establish the existence of unique
minimizers ${E_\lambda}_*$ and $E_\lambda^*$  of $\cA(\cdot,E_0,\lambda),$
such that any other minimizer $E_\lambda$ satisfies
${E_\lambda}_*\subseteq E_\lambda\subseteq E_\lambda^*.$
\end{remark}

\begin{definition}[\bf Maximal and minimal minimizers]\label{def:exis_max_min}
We call  $E_\lambda^*$ and ${E_\lambda}_*$ the maximal and minimal minimizer
of $\cA(\cdot,E_0,\lambda)$ respectively.
\end{definition}

Before proving Theorem \ref{teo:E_0_and_F_0}  we need the following observations. 
Given $\beta$ satisfying   \eqref{hyp:main1}, $C=\cyl{r}{h},$ $h,r>0$ and 
$v\in L_\loc^\infty(\Omega),$ $v\ge0$  a.e. in 
$\Omega\setminus C,$ 
define the convex functional $\cB(\cdot,v,C):BV(\Omega,[0,1])\to (-\infty,+\infty],$ 
a sort of level-set capillary
Almgren-Taylor-Wang-type functional, as
$$
\cB(u,v, C) = \tC(u,\Omega)+ \int_{\Omega} uv\,dx.
$$
Set 
\begin{equation}\label{eq:universal constant2}
\cR_1(C,v):= r+1+\max\Big\{8^{n^2+n+1}\,
\left(\frac{\cC(C,\Omega)+\|v\|_{L^\infty(C)}|C|}{\kappa}\right)^{\frac{n+1}{n}},
\,4\mu(\kappa,n)\Big\},
\end{equation}
where 
$
\mu(\kappa,n) = \left(1/\kappa +2\right)^{\frac{n+1}{n}}.
$
By Example \ref{exam:1} the functional 
$$
\cV:BV(\Omega,\{0,1\})\to (-\infty,+\infty],\qquad\cV(E):=\int_Evdx
$$
satisfies Hypothesis \ref{hyp:2}. Thus, by Theorem 
\ref{teo: unconstrained minimizer1} the functional 
$E\in BV(\Omega,\{0,1\})\mapsto \cB(\chi_E,v, C)\in\R$ has a minimizer,
and  every minimizer $E_v$ satisfies 
\begin{equation}\label{uni.bnom}
E_v\subseteq \cyl{\cR_1(C,v)}{h}. 
\end{equation}
Notice that by \eqref{coarea_trace} and \eqref{eq:coarea_J},
\begin{equation}\label{eq:coarea_for_F}
\cB(u,v,C) = \int_0^1 \cB(\chi_{\{u>t\}},v,C)\,dt 
\quad \forall u\in BV(\Omega,[0,1]),
\end{equation}
which yields that $\chi_{E_v}$ is a minimizer of $\cB(\cdot,v,C)$ in 
$BV(\Omega,[0,1]).$ 

The following remark is in the spirit of \cite[Section 1]{BPV:91}.

\begin{remark}[\bf Minimality of level sets]\label{rem:level_sets_of_minimizers}
From \eqref{eq:coarea_for_F}  it follows that
$u\in BV(\Omega,[0,1])$ is a minimizer of  $\cB(\cdot,v,C)$ in 
$BV(\Omega,[0,1])$ if and only if
$\chi_{\{u>t\}}$ is a minimizer of $\cB(\cdot,v,C)$ for a.e. 
$t\in [0,1].$ Indeed, let for some
$u\in BV(\Omega,[0,1])$ the function $\chi_{\{u>t\}}$ be a minimizer 
of $\cB(\cdot,v,C)$ for a.e. $t\in [0,1].$
Then for any $w\in BV(\Omega,[0,1])$ and for a.e. 
$t\in [0,1]$  one has $\cB(w,v,C)\ge \cB(\chi_{\{u>t\}},v,C),$
therefore,
$$
\cB(u,v,C) = \int_{0}^1 \cB(\chi_{\{u>t\}},v,C)dt \le \cB(w,v,C).
$$
Conversely, if $u\in BV(\Omega,[0,1])$ is a minimizer of 
$\cB(\cdot,v,C),$ then for a.e. $t\in [0,1]$ one has
$\cB(u,v,C)\le \cB(\chi_{\{u>t\}},v,C).$ Hence, 
from \eqref{eq:coarea_for_F} it follows that
$\cB(u,v,C) =\cB(\chi_{\{u>t\}},v,C)$ for a.e. $t\in [0,1].$
In particular, if $u\in BV(\Omega,[0,1])$ is a minimizer of $\cB(\cdot,v,C),$
then by \eqref{uni.bnom} 
$\{u>t\}\subseteq \cyl{\cR_1(C,v)}{h}$ for a.e. $t\in[0,1],$
i.e. $u=0$ a.e. in $\Omega\setminus \cyl{\cR_1(C,v)}{h}.$ Hence,
\begin{equation}\label{dfdfd}
\min\limits_{u\in BV(\Omega,[0,1])} \cB(u,v,C)=
\min\limits_{\begin{substack}
              u\in BV(\Omega,[0,1]),\\
              \text{$u=0$ a.e. in $\Omega\setminus \cyl{\cR_1(C,v)}{h}$}
             \end{substack}
} \cB(u,v,C). 
\end{equation}
\end{remark}

\begin{lemma}\label{lem:rel_min_of ATW_and_F}
Let $E_0,$ $\beta$ satisfy   \eqref{hyp:main1},  and $R_0$
be defined as in \eqref{eq:universal constant}. Then $E_\lambda$ is a minimizer of
$\cA(\cdot,E_0,\lambda)$ if and only if $\chi_{E_\lambda}$ is a minimizer of
$\cB(\cdot,v_{E_0}^\lambda,\cyl{R_0}{H}),$ where
 $v_{E_0}^\lambda = \lambda\chi_{\cyl{R_0}{H}}\sdist_{E_0}.$
\end{lemma}

\begin{proof}
By \eqref{representing_ATW} we have
\begin{equation}\label{eq:layer_cake_ATW}
\cA(E,E_0,\lambda) = \cB(\chi_E,v_{E_0}^\lambda,\cyl{R_0}{H}) - 
\lambda\int_{E_0}\sdist_{E_0}\,dx\qquad \forall E\in BV(\cyl{R_0}{H},\{0,1\}).
\end{equation}

Now if $E_\lambda$ minimizes $\cA(\cdot,E_0,\lambda),$ we have
$E_\lambda\subseteq \cyl{R_0}{H}$ (Theorem \ref{teo: unconstrained minimizer})
and thus, for any $u\in BV(\Omega,[0,1])$ with $u=0$ a.e. 
in $\Omega\setminus \cyl{R_0}{H}$ 
from \eqref{eq:coarea_for_F}-\eqref{eq:layer_cake_ATW}
we deduce
\begin{align*}
\cB(u,v_{E_0}^\lambda,\cyl{R_0}{H}) = & 
\int_0^1\cB(\chi_{\{u>t\}},v_{E_0}^\lambda,\cyl{R_0}{H})\,dt =
\int_0^1 \cA(\{u>t\},E_0,\lambda)\,dt + \lambda\int_{E_0}\sdist_{E_0}\,dx\\
\ge & \int_0^1 \cA(E_\lambda,E_0,\lambda)\,dt + \lambda\int_{E_0}\sdist_{E_0}\,dx
= \cB(\chi_{E_\lambda},v_{E_0}^\lambda,\cyl{R_0}{H}).
\end{align*} 
By \eqref{dfdfd} $\chi_{E_\lambda}$ is a minimizer of 
$\cB(\cdot,v_{E_0}^\lambda,\cyl{R_0}{H}).$

Conversely, assume that $\chi_{E_\lambda}$ is a minimizer
of $\cB(\cdot,v_{E_0}^\lambda,\cyl{R_0}{H}),$  then
by \eqref{eq:layer_cake_ATW}
$E_\lambda\subseteq \cyl{R_0}{H}$ is a minimizer of $\cA(\cdot,E_0,\lambda)$
in $BV(\cyl{R_0}{H},\{0,1\}).$
Hence, by Remark \ref{rem:unconst_min}
 $E_\lambda$ is a minimizer of $\cA(\cdot,E_0,\lambda).$
\end{proof}

\begin{proposition}[\bf Strong comparison for 
minimizers of $\cB$]\label{prop:strong_max_pr}
Assume that $v_1,v_2\in L_\loc^\infty(\Omega),$ $v_1 > v_2$ a.e. in $\Omega$
and $v_2\ge0$ a.e. in $\Omega\setminus C.$ 
Suppose also that  $\beta_1\le \beta_2$ satisfy
 \eqref{hyp:main1}. Let $u_1,u_2\in BV(\Omega,[0,1])$ be  minimizers of
$\cBo(\cdot,v_1,C)$ and $\cBt(\cdot,v_2,C)$ respectively. Then $u_1\le u_2$
a.e. in $\Omega.$
\end{proposition}

\begin{proof}
Adding the inequalities $\cBo(u_1,v_1,C)\le \cBo(u_1\wedge u_2,v_1,C)$ and
$\cBt(u_2,v_2,C)\le \cBt(u_1\vee u_2,v_2,C)$ and using
\begin{equation*} 
\int_{\Omega} |D(u_1\wedge u_2)| +\int_\Omega |D(u_1\vee u_2)| \le
\int_\Omega |Du_1| +\int_\Omega |Du_2|,
\end{equation*}
we establish
$$
\int_{\p \Omega\cap\{u_1>u_2\}} (\beta_2-\beta_1)(u_1-u_2)\,d\cH^n \le
\int_{\{u_1>u_2\}} (v_2-v_1)(u_1-u_2)\,dx.
$$
Since $v_1>v_2$   and $\beta_1\le \beta_2,$
this inequality holds if and only if
$|\{u_1>u_2\}| =0,$ i.e. $u_1\le u_2$ a.e. in $\Omega.$
\end{proof}

\begin{proposition}[\bf Comparison for minimizers of $\cB$]\label{prop:weak_comp_pr}
Assume that $v_1,v_2\in L_\loc^\infty(\Omega),$ $v_1\ge v_2$ a.e. in $\Omega$
and $v_2\ge0$ a.e. in $\Omega\setminus C.$ Suppose also that
$\beta_1\le \beta_2$ satisfy \eqref{hyp:main1}. Then:
\begin{itemize}
\item[a)]  there exists a minimizer ${u_1}_*$ of
$\cBo(\cdot,v_1,C)$ such that ${u_1}_*\le u_2$ for  any minimizer $u_2$ of
$\cBt(\cdot,v_2,C);$
\item[b)]  there exists a minimizer $u_2^*$ of
$\cBt(\cdot,v_2,C)$ such that $u_1\le u_2^*$ for any minimizer $u_1$ of
$\cBo(\cdot,v_1,C).$
\end{itemize}

\end{proposition}

\begin{proof}

a) Take $\epsilon\in (0,1).$ Since $v_1+\epsilon >v_2$ a.e. in $\Omega,$
by Proposition \ref{prop:strong_max_pr} any minimizer   $u_1^\epsilon,
u_2\in BV(\Omega,[0,1])$
of $\cBo(\cdot,v_1+\epsilon,C)$ and $\cBt(\cdot,v_2,C)$ respectively,
satisfies $u_1^\epsilon \le u_2.$ Let $\cR_1:=\max\{\cR_1(C,v_1),\cR_1(C,v_2)\}.$
By minimality, $\cBo(u_1^\epsilon,v_1+\epsilon,C) \le
\cBo(0,v_1+\epsilon,C)=0,$ and since by Remark \ref{rem:level_sets_of_minimizers} 
$u_1^\epsilon=0$ a.e. in $\Omega\setminus \cyl{\cR_1}{h},$ 
recalling \eqref{eq:bound_J-0} we get
$$
\kappa \int_\Omega|Du_1^\epsilon| \le 
(\|v_1\|_{L^\infty(\cyl{\cR_1}{h})}+1)|\cyl{\cR_1}{h}|<+\infty.
$$
By compactness, there exists ${u_1}_*\in BV(\Omega,[0,1])$ such that,
up to a (not relabelled) subsequence,
$u_1^\epsilon \to {u_1}_*$ in $L^1(\Omega)$ and a.e. in $\Omega$ as 
$\epsilon\to0^+.$ Then  any minimizer
$u_2$ of $\cBt(\cdot,v_2,C)$ satisfies ${u_1}_*\le u_2$ a.e. in $\Omega.$

It remains to show that ${u_1}_*$ is a minimizer of $\cBo(\cdot,v_1,C).$
By \eqref{dfdfd} we may consider only those $u\in BV(\Omega,[0,1])$ with 
$u=0$ a.e. in $\Omega\setminus \cyl{\cR_1}{h}$
as a competitor. In this case, the continuity
of $u\mapsto \int_{\cyl{\cR_1}{h}} uv\,dx,$ the minimality of $u_1^\epsilon$
 and the lower semicontinuity of $\tC(\cdot,\Omega)$ imply
\begin{align*}
\cBo(u,v_1,C) = &\lim\limits_{\epsilon\to0^+} \cBo(u,v_1+\epsilon,C) \ge
\liminf\limits_{\epsilon\to0^+} \cBo(u_1^\epsilon,v_1+\epsilon,C)\\
\ge & \liminf\limits_{\epsilon\to0^+} \mathtt{C}_{\beta_1}(u_1^\epsilon, \Omega) +
\lim\limits_{\epsilon\to0^+}
\int_{\cyl{\cR_1}{h}}u_1^\epsilon (v_1+\epsilon)\,dx\\
\ge& \mathtt{C}_{\beta_1}({u_1}_*,\Omega) +\int_{\cyl{\cR_1}{h}} 
{u_1}_*v_1\,dx = \cBo({u_1}_*,v_1,C).
\end{align*}

b) can be proven in a similar manner.
\end{proof}

\begin{proof}[Proof of Theorem \ref{teo:E_0_and_F_0}]
Let $R:=\max\{R(E_0),R(F_0)\},$ where $R(E_0)$ and $R(F_0)$
are defined as in \eqref{eq:universal constant}. Then by Theorem
\ref{teo: unconstrained minimizer}
any minimizer $E_\lambda$  (resp. $ F_\lambda$)
of $\cAo(\cdot,E_0,\lambda)$ (resp. $\cAt(\cdot,F_0,\lambda)$) is contained
in the cylinder $C:=\hat B_{R}\times(0,H),$
where
$$
H=1+\max\big\{\max\limits_{ (x',x_{n+1})\in \overline{E_0}} x_{n+1},\,
\max\limits_{ (x',x_{n+1})\in \overline{F_0}} x_{n+1}\big\}.
$$
Set $v_1:=v_1(\lambda,E_0)= \lambda \sdist_{E_0}$ and
$v_2:=v_2(\lambda,F_0)=\lambda \sdist_{F_0}.$ 
Since $E_0\subseteq F_0\subset\Omega,$ we have  $\sdist_{E_0}\ge\sdist_{F_0}.$
Moreover, by \eqref{hyp:main1} there exists a cylinder 
$C:=\cyl{D}{H}$ such that $v_2\ge0$ in $\Omega\setminus C.$

a) Since $v_1\ge v_2$ and $\beta_1\le \beta_2,$ by Proposition
\ref{prop:weak_comp_pr} b) there exists a minimizer $u_2^*:=u_2^*(\lambda,F_0)$
of $\cBt(\cdot,v_2,C)$
such that any  minimizer $u_1$ of $\cBo(\cdot,v_1,C)$ satisfies
\begin{equation}\label{compilloa}
u_1\le u_2^*.
\end{equation}
By Remark \ref{rem:level_sets_of_minimizers} there exists $t\in(0,1)$
such that $\chi_{\{u_2^*>t\}}$ is a minimizer of $\cBt(\cdot,v_2,C).$ Then, recalling the expression of
$v_2,$ by Lemma
\ref{lem:rel_min_of ATW_and_F}  $F_\lambda^*: = \{u_2^*>t\}$ is a minimizer of
$\cAt(\cdot,F_0,\lambda).$
Moreover, if $E_\lambda$ is a minimizer of $\cAo(\cdot,E_0,\lambda),$ then by Lemma
\ref{lem:rel_min_of ATW_and_F} $\chi_{E_\lambda}$ is a minimizer of $\cBo(\cdot,v_1,C),$ and
by \eqref{compilloa} $\chi_{E_\lambda}\le u_2^*.$ In particular,
$$
E_\lambda = \{\chi_{E_\lambda}>t\}\subseteq \{u_2^*>t\} =: F_\lambda^*.
$$

b) is analogous to a) using    Proposition
\ref{prop:weak_comp_pr} a).

The last assertion  follows with the same arguments from Lemma \ref{lem:rel_min_of ATW_and_F} and
Proposition \ref{prop:strong_max_pr}, since \eqref{eq:strict_inclusion}  implies that
$\sdist_{E_0}>\sdist_{F_0}.$
\end{proof}

One useful case  is when $E_0$ is a constrained
 minimizer of $\cC(\cdot,\Omega)$ in 
$ \cE(E_0)\! : $ 
in this case $E_0$ acts as a barrier for minimizers of $\cA(\cdot,E_0,\lambda).$

\begin{proposition}\label{prop: El_in_E0}
Assume that $E_0,\beta_1,\beta_2$ satisfy \eqref{hyp:main1}.
Let $\beta_1\le \beta_2,$
$E_0$ be a constrained minimizer
of $\mathcal{C}_{\beta_2}(\cdot,\Omega)$ in $\cE(E_0)$ and
$E_\lambda\in BV(\Omega,\{0,1\})$
 be a minimizer of $\mathcal{A}_{\beta_1}(\cdot,E_0,\lambda).$ Then
 $E_\lambda\subseteq \overline{E_0}.$
\end{proposition}

\begin{proof}[\bf Proof]
Comparing $E_\lambda$ with $E_0\cap E_\lambda$ we get
$$P(E_\lambda,\Omega) +\lambda \int_{E_\lambda\setminus E_0} \sdist_{E_0}\,dx
\le P(E_\lambda\cap E_0,\Omega) + \int_{\p \Omega} \beta_1\,
\chi_{E_\lambda\setminus E_0} \,d\cH^n.$$
From the constrained minimality of $E_0$ we
have $\mathcal{C}_{\beta_2}(E_0,\Omega)\le
\mathcal{C}_{\beta_2}(E_0\cup E_\lambda,\Omega),$ i.e.
$$P(E_0,\Omega)\le P(E_0\cup E_\lambda,\Omega)-\int_{\p \Omega} \beta_2\,
\chi_{E_\lambda\setminus E_0}\,d\cH^n.$$
Adding these inequalities we obtain
\begin{align*}
P(E_\lambda,\Omega)+P(E_0,\Omega) +
\lambda \int_{E_\lambda\setminus E_0}
\sdist_{E_0}\,dx
\le &P(E_\lambda\cup E_0,\Omega)  +
P(E_\lambda\cap E_0,\Omega)\\
&+ \int_{\p \Omega} (\beta_1 - \beta_2)
\chi_{E_\lambda\setminus E_0}\,d\cH^n.
\end{align*}
Then the condition  $\beta_1\le \beta_2$ and \eqref{famfor} yield that
$$\lambda \int_{E_\lambda\setminus E_0}
\sdist_{E_0}\,dx\le0.$$
Since $\sdist_{E_0}>0$ outside $\overline{E_0},$ the last inequality is possible only
if $|E_\lambda\setminus \overline{E_0}|=0,$ i.e. $E_\lambda\subseteq \overline{E_0}.$
\end{proof}

Proposition \ref{prop: El_in_E0}
gives the following  monotonicity principle.

\begin{proposition}[\bf Monotonicity]\label{prop: El_monotone}
Assume that $E_0,\beta$  satisfy \eqref{hyp:main1}, $E_0$
is a constrained minimizer of $\cC(\cdot,\Omega)$ in
$\cE(E_0)$ such that $|\overline{E_0}\setminus E_0|=0$ and
$E_\alpha\in BV(\Omega,\{0,1\})$
is a minimizer of $\cA(\cdot,E_0,\alpha)$ for   $\alpha\ge1.$
Then $E_\lambda\subseteq E_\mu$ for any $1\le\lambda<\mu.$
Moreover, every $E_\alpha,$ $\alpha\ge1$ is also a constrained minimizer of
$\cC(\cdot,\Omega)$ in $\cE(E_\alpha).$
\end{proposition}

\begin{proof}
Comparison between $E_\lambda$ and $E_\lambda\cap E_\mu$ gives
$$P(E_\lambda,\Omega) +\lambda \int_{E_\lambda\setminus E_\mu} \sdist_{E_0}\,dx
\le P(E_\lambda\cap E_\mu,\Omega) + \int_{\p \Omega} \beta\,
\chi_{E_\lambda\setminus E_\mu} \,d\cH^n.$$
Similarly, for $E_\mu$ and $E_\lambda\cup E_\mu$ we have
$$
P(E_\mu,\Omega) \le P(E_\lambda\cup E_\mu,\Omega)
+\mu \int_{E_\lambda\setminus E_\mu} \sdist_{E_0}\,dx - \int_{\p \Omega} \beta\,
\chi_{E_\lambda\setminus E_\mu} \,d\cH^n. 
$$
Adding the above inequalities and using \eqref{famfor} we obtain
\begin{equation}\label{a2}
 (\lambda-\mu) \int_{E_\lambda\setminus E_\mu}
\sdist_{E_0}\,dx \le 0.
\end{equation}
By hypothesis $|\overline{E_0}\setminus E_0|=0,$
according to Proposition \ref{prop: El_in_E0},
$E_\lambda, E_\mu\subseteq E_0,$ Thus $\sdist_{E_0}\le0$ in
$E_\lambda\setminus E_\mu.$ But since $\lambda<\mu,$ \eqref{a2} is possible
only if $|E_\lambda\setminus E_\mu|=0,$ i.e. $E_\lambda\subseteq E_\mu.$

To prove the final assertion  take any set 
$E\in \cE(E_\alpha).$ Then using $\cA(E_\alpha,E_0,\alpha)\le 
\cA(E_\alpha\cap E_0,E_0,\alpha),$ 
$\alpha\int_{(E_0\cap E)\setminus E_\alpha} \dist_{E_0}dx\ge 0,$ 
and $E_\alpha \subseteq E_0\cap E,$ we get 
$$
\cC(E_\alpha,\Omega) \le \cC(E_\alpha,\Omega) + \alpha\int_{(E_0\cap E)\setminus
E_\alpha} \dist_{E_0}dx \le 
\cC(E\cap E_0,\Omega).
$$
Moreover, since $\cC(E_0,\Omega)\le \cC(E\cup E_0,\Omega),$  from \eqref{famfor} 
we obtain 
$$
\cC(E_\alpha,\Omega) +\cC(E_0,\Omega) \le \cC(E_0\cap E,\Omega) + 
\cC(E_0\cup E,\Omega)\le \cC(E,\Omega) +\cC(E_0,\Omega),
$$
i.e. $\cC(E_\alpha,\Omega)\le \cC(E,\Omega).$
\end{proof}

\begin{proposition}[\bf Comparison between minimizers
of $\cC$ and $\cA$]\label{prop:boundedness_of_minimizers}
Suppose that $E_0$ and $\beta$ satisfy \eqref{hyp:main1}.
\begin{itemize}
\item[a)] Let $E^+\in BV(\Omega,\{0,1\})$  be a constrained minimizer
of $\cC(\cdot,\Omega)$ in $\cE(E_0).$
Then every minimizer $E_\lambda$  of $\cA(\cdot,E_0,\lambda)$ satisfies
$E_\lambda\subseteq  \overline{E^+}.$
\item[b)] Let $E^+\in BV(\Omega,\{0,1\})$ be a bounded constrained minimizer
of $\cC(\cdot,\Omega)$ in $\cE(E^+).$
Then for every $E_0\subseteq E^+$ and
for every minimizer $E_\lambda$
of $\cA(\cdot,E_0,\lambda)$ one has $E_\lambda\subseteq \overline{E^+}.$
Moreover, $E^+$ can be chosen such that $|\overline{E^+}\setminus E^+|=0.$
\end{itemize}
\end{proposition}

\begin{proof}
a) By  Proposition \ref{prop:minimizer_of_F_0}
$E^+$ is a constrained minimizer of $\cC(\cdot,\Omega) $ in $\cE(E^+).$
Let $E_\lambda^+$ be the maximal minimizer of $\cA(\cdot,E^+,\lambda)$
(Definition \ref{def:exis_max_min}).
By Proposition \ref{prop: El_in_E0} we have $E_\lambda^+\subseteq \overline{E^+}.$
Take any minimizer $E_\lambda$  of $\cA(\cdot,E_0,\lambda).$
Since $E_0\subseteq E^+,$  by Theorem \ref{teo:E_0_and_F_0} a) we have
$$
E_\lambda\subseteq E_\lambda^+\subseteq \overline{E^+}.
$$

b) The proof of the first part is exactly the same as the proof of a).
To prove   the second part,
we take any $E_0'\in BV(\Omega,\{0,1\})$ satisfying the hypotheses
of Proposition \ref{prop:dens_est_for_cap} and
containing  $E_0.$ By Theorem \ref{prop:minimizer_of_F_0} there exists a
constrained  minimizer
$E^+$ of $\cC(\cdot,\Omega)$ in $\cE(E_0').$ In particular, $E^+$ is bounded,
and by Proposition \ref{prop:dens_est_for_cap} $\cH^n(\p E^+)=P(E^+) <+\infty.$
Since $\overline{E^+}\setminus E^+ \subseteq \p E^+,$ we have
$|\overline{E^+}\setminus E^+|=0.$
\end{proof}

\section{Existence of a generalized minimizing movement}\label{sec:existence_of_GMM}

Consider the functional $\widehat\cA:BV(\Omega,\{0,1\})
\times BV(\Omega,\{0,1\}) \times [1,+\infty)\times  \Z\to [-\infty,+\infty]$
given by
$$
\widehat\cA(F,G,\lambda,k): =
\begin{cases}
\cA(F,G,\lambda)& \text{if}\,\,k>0,\\
|F\Delta G|& \text{if}\,\, k\le0.
\end{cases}
$$

For any   $k\in \N$ we build the family of sets $E_\lambda(k)$
iteratively as follows: $E_\lambda(0):=E_0$ and $E_\lambda(k),$ $k\ge1,$ is a
minimizer of
$\widehat \cA(\cdot,E_\lambda(k-1),\lambda,k)$ in $BV(\Omega,\{0,1\});$
notice that existence of minimizers follows from Theorem
\ref{teo: unconstrained minimizer}.

From now on, we omit the dependence on $k$ of $\widehat\cA,$ and
we use the notation $\widehat\cA(F,G,\lambda).$

\begin{theorem}[\bf Existence]\label{teo:existence_of_GMM}
Let $E_0$ and $\beta $ satisfy   \eqref{hyp:main}.
Then  $GMM(E_0)$ is nonempty, i.e.
there exist a map $t\in [0,+\infty)\mapsto E(t)\in BV(\Omega,\{0,1\})$ and
a diverging  sequence  $\{\lambda_j\}\subset [1,+\infty)$
such that
\begin{equation}\label{L1_conv_GMM}
\lim\limits_{j\to+\infty} |E_{\lambda_j}([\lambda_jt])\Delta E(t)| = 0,
\qquad t\in [0,+\infty).
\end{equation}
Moreover, every GMM $t\in[0,+\infty)\mapsto E(t)$ starting from $E_0$
is contained in  a  bounded set depending only on $E_0$ and $\beta,$
and belongs to $C_\loc^{1/2}((0,+\infty),L^1(\Omega)),$
in the sense that
\begin{equation}\label{Holder_cont}
|E(t)\Delta E(t')| \le \theta(n,\kappa)P(E_0)  |t-t'|^{1/2}\qquad
\text{for all $t,t'>0,$ $|t-t'|<1,$}
\end{equation}
where
$
\theta(n,\kappa) =  \frac{C_{n,\kappa}}{\kappa}+1
$
and $C_{n,\kappa}$ is defined in \eqref{3452}. 
If in addition $|\overline{E_0}\setminus E|=0,$
then  \eqref{Holder_cont} holds for any $t,t'\ge0$ with $|t-t'|<1.$
Finally,
\begin{equation}\label{eq:wcms}
\nu_{E_{\lambda_j}([\lambda_jt])}\cH^n \res \p^*E_{\lambda_j} 
([\lambda_jt]) \overset{w^*}{\rightharpoonup } \nu_{E(t)}\cH^n\res \p^*E(t) \quad
{\text{for all $t\ge0$  as $\lambda_j\to+\infty.$}}
\end{equation}
\end{theorem}

\begin{proof}[\bf Proof]
Given $k\ge0$  set $\dist_k(\cdot): = \distance(\cdot,\Omega
\cap \p E_{\lambda}(k)).$ Then for $k\ge1$ the minimality of $E_\lambda(k)$
entails
$$\cA(E_\lambda(k),E_\lambda(k-1),\lambda)\le
\cA(E_\lambda(k-1),E_\lambda(k-1),\lambda),$$ i.e.
\begin{equation}\label{monoton_capil}
\cC(E_\lambda(k),\Omega) +\lambda \int_{E_\lambda(k)\Delta E_\lambda(k-1)}
\dist_{k-1}dx \le \cC(E_\lambda(k-1),\Omega).
\end{equation}
In particular, the sequence $k\in\N\cup\{0\}\mapsto \cC(E_\lambda(k),\Omega)$ is
nonincreasing and
\begin{equation}\label{gfgfgf}
\cC(E_\lambda(k),\Omega)
\le \cC(E_\lambda(0),\Omega)=\cC(E_0,\Omega)\le P(E_0).
\end{equation}
Let $t>0$ and set $k=[\lambda t].$ Then
\eqref{eq:bound_F-0} yields
\begin{equation}\label{1234}
\kappa P(E_\lambda([\lambda t])) \le \cC(E_\lambda([\lambda t]),\Omega) \le P(E_0).
\end{equation}

Take $t_1,t_2>0,$ $t_1<t_2$ and let $\lambda\ge1$ be large enough
that for some $k_0,N\in\N,$ $N\ge3$
$$k_0= [\lambda  t_1],\quad k_0+N-1=[\lambda t_2],$$
i.e.
$$\frac{k_0}{\lambda } \le t_1 < \frac{k_0+1}{\lambda } < \ldots <
\frac{k_0+N-1}{\lambda } \le t_2 <\frac{k_0+N}{\lambda }.$$
 Then
\begin{equation}\label{eq:diff_t_1_t_2}
\frac{N-2}{\lambda }= \frac{k_0+N-1 - (k_0+1)}{\lambda } \le t_2-t_1.
\end{equation}

Since all $E_{\lambda}(s),$ $s\ge1$ satisfy uniform density estimates
\eqref{eq:volume_est}-\eqref{eq:perimeter_est} (Theorem \ref{teo:lower_density_est}),
by Proposition \ref{prop:uniform_estimate} we have\footnote{Notice that at this point we
use $t_1>0;$ since a priori we do not know whether $E_0$ satisfies
the density estimates,  we cannot start summing from $s=0=k_0.$}
\begin{equation}\label{ikkita_sum}
\begin{aligned}
|E_{\lambda}([\lambda t_2])\Delta E_{\lambda }([\lambda t_1])|& =
|E_{\lambda }(k_0+N-1) \Delta E_{\lambda }(k_0)|
\le \sum\limits_{s=k_0}^{k_0+N-2} |E_{\lambda}(s) \Delta E_{\lambda}(s+1)|\\
\le & C_{n,\kappa}\ell \sum\limits_{s=k_0}^{k_0+N-2} P(E_{\lambda}(s)) +
\frac{1}{\ell} \sum\limits_{s=k_0}^{k_0+N-2} \int_{E_{\lambda}(s+1)
\Delta E_{\lambda}(s)}
\dist_{E_{\lambda}(s)}\,dx
\end{aligned}
\end{equation}
for any $\ell \in (0,\frac{\gamma(n,\kappa)}{\lambda^{1/2}}).$
The first sum can be estimated using  \eqref{1234}:
\begin{equation}\label{birinchi_sum}
\sum\limits_{s=k_0}^{k_0+N-2} P(E_{\lambda}(s)) \le
\frac{P(E_0)}{\kappa}\,(N-1).
\end{equation}
Moreover, for any $s\in\N,$ by \eqref{monoton_capil}
$$
\int_{E_{\lambda_j}(s+1) \Delta E_{\lambda}(s)}
\dist_{E_{\lambda}(s)}\,dx \le \frac{1}{\lambda}\Big(\cC(E_{\lambda}(s),\Omega)
 - \cC(E_{\lambda}(s+1),\Omega)\Big),
$$
and thus
\begin{align*}
\sum\limits_{s=k_0}^{k_0+N-2} \int_{E_{\lambda}(s+1)
\Delta E_{\lambda}(s)}
\dist_{E_{\lambda}(s)}\,dx \le &   \frac{1}{\lambda}
\sum\limits_{s=k_0}^{k_0+N-2} \Big(\cC(E_{\lambda}(s),\Omega)
 - \cC(E_{\lambda}(s+1),\Omega)\Big)\\
= & \frac{1}{\lambda}  \Big(\cC(E_{\lambda}(k_0),\Omega)
 - \cC(E_{\lambda}(k_0+N-1),\Omega)\Big).
\end{align*}
Using  \eqref{gfgfgf} and the nonnegativity of $\cC(\cdot,\Omega)$ we get
\begin{equation}\label{ikkinchi_sum}
\sum\limits_{s=k_0}^{k_0+N-2} \int_{E_{\lambda}(s+1)
\Delta E_{\lambda}(s)}
\dist_{E_{\lambda}(s)}\,dx \le \frac{P(E_0)}{ \lambda}.
\end{equation}
Thus, from \eqref{ikkita_sum}, \eqref{birinchi_sum} and \eqref{ikkinchi_sum}
\begin{equation}\label{ghfgb}
|E_{\lambda}([\lambda t_1])\Delta E_{\lambda}([\lambda t_2]) |  \le
\frac{C_{n,\kappa}P(E_0)}{\kappa}\,(N-1)\ell + \frac{P(E_0)}{ \lambda \ell}.
\end{equation}

Now take $\lambda$ so large that
$$
t_2-t_1> \frac{1}{\gamma(n,\kappa)^2\,\lambda},
$$
so that Proposition \ref{prop:uniform_estimate} holds for
$\ell = \frac{1}{\lambda|t_2-t_1|^{1/2}}.$
From \eqref{ghfgb} and \eqref{eq:diff_t_1_t_2} we obtain
\begin{equation}\label{xyzta}
\begin{aligned}
\big|E_{\lambda}([\lambda t_1])\Delta E_{\lambda }([\lambda t_2]) \big|  \le&
\frac{C_{n,\kappa}P(E_0)}{\kappa}\,\frac{N-2}{\lambda |t_2-t_1|^{1/2}} +
\frac{1}{\lambda }\,\frac{C_{n,\kappa}P(E_0)}{\kappa|t_2-t_1|^{1/2}}+
 P(E_0) \, |t_2-t_1|^{1/2}\\
\le & \theta(n,\kappa)P(E_0)\,|t_2-t_1|^{1/2}+
\frac{1}{\lambda}\,\frac{C_{n,\kappa}P(E_0)}{\kappa|t_2-t_1|^{1/2}}.
\end{aligned}
\end{equation}

By Proposition \ref{prop:boundedness_of_minimizers} b) there exists
a constrained minimizer $E^+\supseteq E_0$ of $\cC(\cdot,\Omega)$ in $\cE(E^+)$ such that
$|\overline{E^+}\setminus E^+|=0$ and $E_\lambda(1)\subseteq E^+.$
By induction,
we can show that  $E_\lambda(k)\subseteq  E^+$ for all $k\ge1.$
Consider now an arbitrary diverging sequence $\{\lambda_j\}.$
Compactness  and a diagonal process yield the
existence of a subsequence (still denoted by $\{\lambda_j\}$) such that
$E_{\lambda_j}([\lambda_jt])$ converges in $L^1(\Omega)$
to a set $E(t)$ for any rational $t\ge 0$ as $j\to+\infty.$

If $t_1,t_2\in\Q\cap(0,+\infty),$ with $0<|t_1-t_2|<1,$ letting $\lambda_j\to+\infty$
in \eqref{xyzta} we get
\begin{equation}\label{gelder_condit}
|E(t_1)\Delta E(t_2)| \le \theta(n,\kappa)P(E_0) |t_2-t_1|^{1/2}.
\end{equation}
By completeness of $L^1(\Omega)$ we can uniquely extend
$\{E(t):\,t\in\Q\cap (0,+\infty)\}$ to a family
$\{E(t):\,t\in (0,+\infty)\}$ preserving the H\"older continuity \eqref{gelder_condit}
in $(0,+\infty).$
Now we show \eqref{L1_conv_GMM}.
If $t=0,$ $E_0=E_{\lambda_j}(0)\to E(0)$ in $L^1(\Omega)$ as $j\to+\infty.$
If $t>0,$ take any  $\epsilon\in (0,1)$ and $t_\epsilon\in\Q\cap(0,+\infty)$ such that
$|t-t_\epsilon|<\epsilon.$ By the choice of $\{\lambda_j\},$
\eqref{L1_conv_GMM} holds for $t_\epsilon$  and thus,
using \eqref{xyzta}-\eqref{gelder_condit}
we get
\begin{align*}
\limsup\limits_{j\to+\infty} |E_{\lambda_j}([\lambda_jt])\Delta E(t)| \le &
\limsup\limits_{j\to+\infty} |E_{\lambda_j}([\lambda_jt])\Delta E_{\lambda_j}([\lambda_jt_\epsilon])|\\
& + \limsup\limits_{j\to+\infty} |E_{\lambda_j}([\lambda_jt_\epsilon])\Delta E(t_\epsilon)| +
|E(t_\epsilon)\Delta E(t)|\\
\le & 2 \theta(n,\kappa)P(E_0) |t-t_\epsilon|^{1/2}< 2 \theta(n,\kappa)P(E_0)\sqrt\epsilon.
\end{align*}
Therefore, letting $\epsilon\to0^+$ we get \eqref{L1_conv_GMM}.

When $|\overline{E_0}\setminus E_0|=0,$ for any $t\in(0,1),$ choosing $\lambda$
sufficiently large, from \eqref{xyzta} we obtain
\begin{equation}\label{trtr}
\begin{aligned}
|E_\lambda([\lambda t])\Delta E(0)| \le &
|E_\lambda([\lambda t])\Delta E_\lambda(1)| +  |E_\lambda(1)\Delta E_0|\\
\le &\theta(n,\kappa)P(E_0)\Big|t-\frac1\lambda\Big|^{1/2} +\frac{1}{\lambda}
\frac{C_{n,\kappa}P(E_0)}{\kappa|t-\frac{1}{\lambda}|^{1/2}}+
|E_\lambda(1)\Delta E_0|.
\end{aligned}
\end{equation}
By Lemma \ref{lem:behav_big_lambda} a) the last term on the right hand side converges to $0$
as $\lambda\to+\infty.$ Hence letting $\lambda\to+\infty$ in \eqref{trtr}
we get the $(1/2)$-H\"older continuity of $t\mapsto E(t)$ in $[0,+\infty).$

Now let us prove \eqref{eq:wcms}. We need to show that 
for any $t\in [0,+\infty)$ 
\begin{equation*}
\lim\limits_{j\to+\infty} \int_{\p^*E_{\lambda_j}([\lambda_jt])} 
\phi\cdot \nu_{E_{\lambda_j}([\lambda_jt])}\,d\cH^n=
\int_{\p^*E(t)} \phi\cdot\nu_{E(t)}\,d\cH^n \quad\forall \phi\in 
C_c(\R^{n+1},\R^{n+1}). 
\end{equation*}
If $\phi\in C_c^1(\R^{n+1},\R^{n+1}),$
by the generalized divergence formula \eqref{integ_by_parts}
and by \eqref{L1_conv_GMM} we have 
\begin{equation}\label{int_parts}
\begin{aligned}
\lim\limits_{j\to+\infty} \int_{\p^*E_{\lambda_j}([\lambda_jt])} 
\phi\cdot \nu_{E_{\lambda_j}([\lambda_jt])}\,d\cH^n&= 
\lim\limits_{j\to+\infty} \int_{E_{\lambda_j}([\lambda_jt])} 
\div \phi\,d\cH^n \\
&=  \int_{E(t)} \div \phi\,d\cH^n=
\int_{\p^*E(t)} \phi\cdot\nu_{E(t)}\,d\cH^n. 
\end{aligned} 
\end{equation}
In  general, we approximate $\phi\in C_c(\R^{n+1},\R^{n+1})$ uniformly 
with $\phi_k\in C_c^1(\R^{n+1},\R^{n+1}),$ $k\ge1$ and use the previous result.

Finally, if $\{E(t)\}_{t\ge0} \in GMM( E_0),$
then by construction and  Proposition
\ref{prop:boundedness_of_minimizers} b) one has $E_{\lambda_j}([\lambda_j t])\subseteq E^+,$
where $E^+:=E^+(E_0,\beta)$ is a bounded minimizer of $\cC(\cdot,\Omega)$ in $\cE(E^+);$
therefore $E(t)\subseteq E^+$ for all $t\ge0.$
\end{proof}

\begin{definition}[\bf Maximal and minimal GMM]\label{def:max_min_GMM}
Let $E_0,\beta$  satisfy \eqref{hyp:main}, and
$\{\lambda_j\}$ be a diverging sequence such that
$$
E^*(t):=\lim\limits_{j\to+\infty} E_{\lambda_j}([\lambda_jt])^*\qquad\forall t\ge0
$$
exist in $L^1(\Omega),$ where $E_{\lambda_j}([\lambda_jt])^*$
is the maximal   minimizer
of  $\cA(\cdot,E_{\lambda_j}([\lambda_jt]-1)^*,\lambda_j)$
with $(E_0)^*:=E_0$ (Definition \ref{def:exis_max_min}).
We call $E^*(t)$  the {\it maximal}   GMM
associated to the sequence $\{\lambda_j\}.$
Analogously,
$$
E_*(t):=\lim\limits_{j\to+\infty} E_{\lambda_j}([\lambda_jt])_*\qquad\forall t\ge0,
$$
obtained using the  minimal minimizers $E_{\lambda_j}([\lambda_jt])_*$
of $\widehat\cA(\cdot,E_{\lambda_j}([\lambda_jt]-1)_*,\lambda_j)$  with
$(E_0)_*:=E_0,$ is called the minimal GMM associated to the sequence $\{\lambda_j\}.$
\end{definition}

Observe that if   $t\mapsto E(t)$ is any GMM obtained by the sequence
$\{\lambda_j\},$ then according to the proof of Theorem \ref{teo:existence_of_GMM}
(possibly passing to nonrelabelled subsequences) there exist
the maximal  GMM $t\mapsto E^*(t)$ and the minimal
GMM $t\mapsto E_*(t)$ associated to  $\{\lambda_j\}.$
Now by Remark \ref{rem:exis_max_min} one has
$E_*(t)\subseteq E(t)\subseteq E^*(t)$ for all $t\ge0.$

\begin{theorem}[\bf Comparison principle for maximal and minimal GMM]\label{teo:comp_princ_for_GMM}
Let $E_0,F_0,\beta_1,\beta_2$ satisfy   \eqref{hyp:main} with $E_0\subseteq F_0$
and $\beta_1\le\beta_2.$ If $E_*(t)$ and $F_*(t)$ are minimal GMMs associated to a
sequence $\{\lambda_j\},$ then $E_*(t)\subseteq F_*(t)$ for all $t\ge0.$
Analogously, if $E^*(t)$ and $F^*(t)$ are maximal GMMs associated to $\{\lambda_j'\},$  then
$E^*(t)\subseteq F^*(t)$ for all $t\ge0.$
\end{theorem}

\begin{proof}
Since $E_0\subseteq F_0,$ and $\beta_1\le \beta_2,$ by definition of
${E_\lambda(k)}^*$ and ${F_\lambda(k)}^*$ (resp. ${E_\lambda(k)}_*$ and ${F_\lambda(k)}_*$)
and by Theorem \ref{teo:E_0_and_F_0},
we have ${E_\lambda}_*(k) \subseteq {F_\lambda}_*(k)$ (resp. $E_{\lambda}^*(k)
\subseteq F_{\lambda}^*(k)$) which implies  $E_*(t)\subseteq F_*(t)$
(resp. $E^*(t)\subseteq F^*(t)$) for all $t\ge0.$
\end{proof}

From the proof of Theorem \ref{teo:existence_of_GMM} and
Propositions \ref{prop: El_in_E0} -\ref{prop: El_monotone} 
we get the following result (compare with \cite{BChN:05}),
that could be applied, for instance, to $E_0$ as in Example 
\ref{exa:decreas}.

\begin{theorem}\label{teo:homot_shrink}
Let $E_0$ be a constrained minimizer of $\cC(\cdot,\Omega)$ in $\cE(E_0)$
such that $|\overline{E_0}\setminus E_0|=0.$ Then every maximal (minimal) 
GMM $t\mapsto E(t)$
starting from $E_0$ satisfies $E(t)\subseteq E(t')$ provided $t>t'\ge0.$ 
\end{theorem}

\begin{proof}
Applying Propositions \ref{prop: El_in_E0} and \ref{prop: El_monotone} 
inductively to maximal minimizers $E_\lambda(k)^*$ of
$\widehat\cA(\cdot, E_\lambda(k-1)^*,\lambda)$
we  get $E_\lambda(k)^*\subseteq E_\lambda(k-1)^*$
for all $k\ge1$ and 
$\lambda\ge1.$ Hence, if $t>t'\ge0$ then  
$E_\lambda([\lambda t])^*\subseteq E_\lambda([\lambda t'])^*.$ 
Now the assertion of the theorem follows from \eqref{L1_conv_GMM}.
The arguments for minimal minimizers are the same.
\end{proof}

\section{GMM as a distributional solution}\label{sec:weak_curvature}

The aim of this section is to prove that under suitable assumptions
 GMM is in fact a distributional solution
of   \eqref{mcf}-\eqref{contact_angle}.
Let us start with the following

\begin{definition}[\bf Admissible variation]
A vector field $X=(X',X_{n+1}) \in C_c^1(\overline{\Omega},\R^{n+1})$  is called
{\it admissible } if
$X\cdot e_{n+1}=0$  on $\p\Omega.$
\end{definition}

Observe that if
$X\in C_c^1(\overline{\Omega},\R^{n+1})$ is admissible,
then for any $s\in (-\epsilon,\epsilon)$ with $\epsilon>0$ sufficiently small,
the vector field  $ f_s =  \Id+ s X$
is a $C^1$-diffeomorphism that
satisfies $f_s(\Omega) = \Omega,$ $f_s(\overline{\Omega}) =
\overline{\Omega}.$

\begin{proposition}[\bf First variation of $\cA$]\label{prop:first_var}
Suppose that $E_0,$ $\beta$ satisfy assumptions \eqref{hyp:main}  and let
$E\in BV(\Omega,\{0,1\})$ be bounded   with $\Trace(E)\in BV(\R^n,\{0,1\}).$
Then
\begin{equation}\label{first_variatsion}
\begin{aligned}
\dfrac{d}{ds}\Big|_{s=0}\,\cA(f_s(E),E_0,\lambda) =&
\int_{\Omega\cap \p^*E} (\div X - \nu_E \cdot (\nabla X)\nu_E)\,d\cH^n\\
&+
\lambda \int_{\Omega\cap \p^*E} \sdist_{E_0}\,X\cdot \nu_E\,d\cH^n
- \int_{\p^* \Trace(E)} \beta\,X'\cdot \nu_{\Trace (E)}'d\cH^{n-1},
\end{aligned}
\end{equation}
where $\p^*\Trace(E)$ is the essential boundary of
$\Trace(E)$ on $\p\Omega$ and $\nu_{\Trace(E)}'$ is the outer
unit normal to $\Trace(E)\subset \R^n.$
\end{proposition}

\begin{proof}
From \cite[Theorem 17.5]{Mag12}
$$
\dfrac{d}{ds} \Big|_{s=0} P(f_s(E),\Omega) =
\int_{\Omega\cap \p^* E} (\div X - \nu_E \cdot (\nabla X)\nu_E)\,d\cH^n.
$$
Moreover, \cite[Theorem 17.8]{Mag12} and the admissibility of $X$  imply that
$$
\dfrac{d}{ds} \Big|_{s=0} \int_{f_s(E)}\sdist_{E_0}\,dx =
\int_{\p^* E} \sdist_{E_0}\, X\cdot \nu_E\,d\cH^n=
\int_{\Omega\cap \p^* E} \sdist_{E_0}\, X\cdot \nu_E\,d\cH^n.
$$
Finally, since $\Trace(E)$ is a set of finite perimeter in $\p\Omega\equiv \R^n,$
again using  \cite[Theorem 17.8]{Mag12} we get
$$
\frac{d}{ds}\Big|_{s=0}\int_{\p\Omega} \beta \,\chi_{f_s(E)}d\cH^n =
\int_{\p^*\Trace(E)} \beta\, X'\cdot \nu_{\Trace(E)}'\,d\cH^{n-1}.
$$
\end{proof}

\begin{remark}\label{rem:w_curvature}
Under assumptions \eqref{hyp:main} and $\beta\in \Lip(\p\Omega),$
if $E_\lambda$ is a minimizer  of $\cA(\cdot,E_0,\lambda),$
and if  $\Omega\cap\p E_\lambda$ is a $C^2$-manifold
with $\cH^{n-1}$- rectifiable boundary, then the mean curvature
$H_{E_\lambda}$ of $\Omega \cap \p E_\lambda$
is equal to $-\lambda \sdist_{E_0}.$
Indeed, using the tangential divergence formula for manifolds with boundary
we have
$$
\int_{\Omega\cap \p E_\lambda} (\div X - \nu_{E_\lambda} \cdot
(\nabla X)\nu_{E_\lambda})\,d\cH^n =
\int_{\Omega\cap \p E_\lambda} H_{E_\lambda}\,X\cdot \nu_{E_\lambda}\,d\cH^n +
\int_{\p^*\Trace(E_\lambda)} X'\cdot {\conormal^\lambda}'\,d\cH^{n-1},
$$
where $\conormal^{\lambda}=({\conormal^{\lambda}}',
\conormal^\lambda_{n+1})$
is the outer unit conormal to $\overline{\Omega\cap \p E_\lambda}$
at $\overline{\Omega\cap \p {E_\lambda}} \cap\p\Omega.$
By minimality of $E_\lambda,$ we have
$\frac{d}{ds}\cA(f_s(E_\lambda),E_0,\lambda) =0,$
i.e.
$$
\int_{\Omega\cap \p E_\lambda} (H_{E_\lambda} +
\lambda \sdist_{E_0})\,X\cdot \nu_{E_\lambda}\,d\cH^n
+
\int_{\p^*\Trace(E_\lambda)} X'\cdot ({\conormal^\lambda}' - \beta
\nu_{\Trace({E_\lambda})}') \,d\cH^{n-1}=0.
$$
This implies $H_{E_\lambda}=- \lambda \sdist_{E_0}$ and
${\conormal^\lambda}' = \beta
\nu_{\Trace(E_\lambda)}'.$ Notice that from the latter in particular,
we get
$$
\beta = \conormal^\lambda \cdot (\nu_{\Trace(E_\lambda)}',0)=
\nu_{E_\lambda}\cdot e_{n+1},
$$
accordingly for instance with Theorem  \ref{teo: regularity}.
\end{remark}

Remark \ref{rem:w_curvature} motivates the following definition
\cite{BGM:10,Mag12}.

\begin{definition}[\bf Distributional mean curvature]\label{def:dist_curv}
Let $E \in BV(\Omega,\{0,1\}).$
The function $H_E\in L^1(\Omega\cap \p^*E;
\cH^n\res(\Omega\cap \p^*E))$  is called
{\it distributional mean curvature of $\Omega\cap\p^*E$} if for every
$X\in C_c^1(\Omega,\R^{n+1})$
the generalized tangential divergence formula holds:
\begin{equation}\label{dist_curvat}
\int_{ \Omega\cap\p^*E} \left( \div X - \nu_E\cdot (\nabla X)\nu_E\right)\,d\cH^n
=\int_{\Omega\cap \p^*E } H_E\,X\cdot \nu_E\,d\cH^n.
\end{equation}
\end{definition}

\bigskip

Given  $x\in \R^{n+1}$ and $t>0$  set
$$
v_\lambda(t,x): =
\begin{cases}
-\lambda \sdist_{E_\lambda([\lambda t]-1)}(x)&
\text{if $t\ge \frac1\lambda,$}\\
0&\text{if $t\in[0,\frac1\lambda).$}
\end{cases}
$$

\begin{remark}\label{rem:traces}
By Theorem \ref{teo: regularity}, $\Trace(E_\lambda([\lambda t]))\in BV(\R^n,\{0,1\}).$
\end{remark}

The next result relates GMM with distributional
solutions of \eqref{mcf}-\eqref{contact_angle}. 

\begin{theorem}[\bf GMM is a distributional solution]\label{teo:distr_sol}
Let $E_0,\beta$ satisfy \eqref{hyp:main}, $|\overline{E_0}\setminus E_0|=0,$
$\{E(t)\}_{t\ge0}$ be a GMM starting from $E_0$ obtained along the
diverging  sequence
$\{\lambda_j\}$. Suppose that
\begin{equation}\label{uniform_per_est}
\cH^n\res(\Omega\cap \p^*E_{\lambda_j}([\lambda_jt]))
\overset{w^*}{\rightharpoonup} \cH^n\res(\Omega\cap\p^* E(t))
\quad \text{ as $j\to+\infty$ for a.e. $t\ge0.$}
\end{equation}
Then there exist
a function  $v:[0,+\infty)\times\Omega\to \R$ with
\begin{equation}\label{l2_bound_velic}
\int_0^{+\infty} \int_{\Omega\cap \p^*E(t)} (v)^2\,d\cH^n\,dt \le
\alpha(n,\kappa)\,P(E_0),
\end{equation}
and a  (not relabelled)
subsequence such that
\begin{equation}\label{weak_conv_velocities}
\lim\limits_{j\to+\infty}  \int_0^{+\infty} \int_{\Omega\cap \p^*E_{\lambda_j}([\lambda_jt])}
\phi v_{\lambda_j}\,d\cH^ndt
=\int_0^{+\infty} \int_{\Omega\cap \p^*E(t)} \phi v \,d\cH^ndt,
\end{equation}
\begin{equation}\label{weak_conv_some}
\lim\limits_{j\to+\infty} \int_0^{+\infty}\int_{\Omega\cap \p^*E_{\lambda_j}([\lambda_jt])}
v_{\lambda_j}\,\nu_{E_{\lambda_j}([\lambda_jt])}\cdot \Psi\, d\cH^ndt
=\int_0^{+\infty}\int_{\Omega\cap \p^*E(t)} v\, \nu_{E(t)}\cdot \Psi  \,d\cH^n  dt
\end{equation}
for any $\phi\in C_c(\Omega),$ $\Psi\in C_c([0,+\infty)\times\Omega,\R^{n+1}),$
where $\alpha(n,\kappa):=
\frac{75[(n+1)\omega_{n+1}+\omega_n]\mathfrak{b}(n)}{(\kappa/2)^{n+1}\omega_{n+1}}.$
Moreover,
$\{E(t)\}_{t\ge0}$ solves \eqref{mcf}-\eqref{contact_angle} with initial
datum $E_0$ in the following sense:
\begin{itemize}
\item[(i)]  for a.e. $t\ge0$ the set
$\Omega\cap \p^*E(t)$ has  distributional
mean curvature $H_{E(t)}=v$  and if $1\le n\le 6,$ for every $\phi \in
C_c^1([0,+\infty)\times \Omega):$
\begin{equation}\label{normal_velocity}
\int_0^{+\infty} \int_{E(t)} \p_t \phi\,dxdt +\int_{E(0)} \phi(0,x)\,dx
=\int_0^{+\infty} \int_{ \Omega\cap\p^*E(t) } \phi H_{E(t)}\,d\cH^ndt;
\end{equation}
\item[(ii)] if $\beta\in \Lip(\p\Omega)$ and there exists $h\in L_\loc^1([0,+\infty))$
such that 
\begin{equation}\label{un.bund.pe}
P(\Trace(E_{\lambda_j}([\lambda_jt])))\le h(t) \qquad 
\text{for all $j\ge1$ and a.e. $t\ge0,$} 
\end{equation}
then $\Trace(E(t))\in BV(\R^n,\{0,1\})$ for a.e. $t>0$ and
\begin{equation}\label{contactanglejon}
\begin{aligned}
\int_{\Omega\cap \p^*E(t)} &
\Big(\div X - \nu_{E(t)} \cdot (\nabla X)
\nu_{E(t)}\Big)\,d\cH^n\\
& = \int_{\Omega\cap \p^*E(t)}
H_{E(t)}\,X\cdot \nu_{E(t)}\,d\cH^n
+ \int_{\p^* \Trace(E(t))}
\beta\,X'\cdot \nu_{\Trace (E(t))}'\,d\cH^{n-1}
\end{aligned}
\end{equation}
for every admissible $X \in C_c^1(\overline{\Omega},\R^{n+1}).$
\end{itemize}
\end{theorem}

The need for assumption \eqref{uniform_per_est} is not surprising; see
\cite{LS:95,MSS:2016} for conditional results obtained in other contexts
in a similar spirit.
We postpone the proof after several auxiliary results.

\begin{proposition}\label{prop:dist_m_c}
Assume that $E_0$ and $\beta$ satisfy \eqref{hyp:main}. Then
for any $\lambda\ge1$ and a.e. $t\ge 1/\lambda$ the function $v_\lambda(t,\cdot)$
is the distributional  mean curvature of $E_\lambda([\lambda t]).$
\end{proposition}

\begin{proof}
Set $E:=E_{\lambda}([\lambda t]).$ Remark \ref{rem:traces}
and \eqref{first_variatsion} imply that
$$
\int_{ \Omega\cap\p^*E} \left( \div X - \nu_E\cdot (\nabla X)\nu_E\right)\,d\cH^n
=\int_{\Omega\cap \p^*E } v_\lambda \,X\cdot \nu_E\,d\cH^n.
$$
Hence, it suffices to prove $v_\lambda(t,\cdot)\in
L^1(\Omega\cap \p^*E ; \cH^n\res {\Omega\cap \p^*E})$
for a.e. $t\in[1/\lambda,+\infty)$
and since $P(E(t),\Omega)<+\infty,$ this follows from Lemma \ref{lem:l2bound_velocity} below.
\end{proof}

\begin{remark}
From Definition \ref{def:dist_curv}, Proposition \ref{prop:dist_m_c} and Lemma
\ref{lem:l2bound_velocity} it follows that
$$v_\lambda(t,x) = H_{E_{\lambda}([\lambda t])}(t,x)\quad
\text{for a.e. $t\ge1/\lambda$ and $\cH^n$-a.e. $x\in \Omega\cap\p E_{\lambda}([\lambda t]).$}$$
This is a discretized version of equation \eqref{mcf}.
\end{remark}

\begin{lemma}[\bf Uniform $L^2$-bound of the approximate velocities]\label{lem:l2bound_velocity}
Under assumptions \eqref{hyp:main} the inequality
$$
\int_0^{+\infty}\int_{\Omega\cap \p E_{\lambda}([\lambda t])}
(v_\lambda)^2\,d\cH^ndt \le \alpha(n,\kappa)P(E_0)
$$
holds.
\end{lemma}

\begin{proof}
The proof is analogous to the proof of \cite[Lemma 3.6]{MSS:2016}.
Given $\epsilon>0$ and $E\in BV(\Omega,\{0,1\})$ let
$$
(\p E)_\epsilon^+:=
\{x\in \R^{n+1}:\,\, \distance(x,\Omega\cap \p E)\le \epsilon\}.
$$

 For $t\in [\frac1\lambda,+\infty)$ and $\ell\in\Z$ such that
 $\ell \le 1+[\log_2(R(n,\kappa) \lambda^{1/2})],$
 where $R(n,\kappa)$ is given by \eqref{const:uniform_density_est},
define
 $$
 K(\ell) = \Big\{x\in \Big(\p E_\lambda([\lambda t]-1)\Big)_{R(n,\kappa)
 \lambda^{-1/2}}^+:\,\,2^{\ell}<|v_\lambda(x,t)|\le 2^{\ell+1}\Big\}.
 $$

 By Proposition \ref{prop:uniform_L_infty_est}
 $
 E_\lambda([\lambda t])\Delta E_\lambda([\lambda t]-1)\subseteq \cup_\ell K(\ell).
 $
 Take $x\in  K(\ell)\cap\Omega\cap  \p E_\lambda([\lambda t]).$
 Then   $B_{\frac{2^{\ell-1}}{\lambda}}(x)\cap E_{\lambda}([\lambda t]-1)=\emptyset$
and hence, by Remark \ref{rem:uncons_dens_est}  the following density estimates hold:
 \begin{equation}\label{dens_estimate}
 \begin{aligned}
 |E_\lambda([\lambda t])\cap  B_{\frac{2^{\ell-1}}{\lambda}}(x)|\ge
\left(\frac{\kappa}{2}\right)^{n+1}
 \omega_{n+1}\left(\frac{2^{\ell-1}}{\lambda}\right)^{n+1},\\
 \cH^n(B_{\frac{2^{\ell-1}}{\lambda}}(x)\cap\Omega\cap \p E_\lambda([\lambda t]))
 \le \big[(n+1)\omega_{n+1}+\omega_n\Big] \left(\frac{2^{\ell-1}}{\lambda}\right)^n.
 \end{aligned}
 \end{equation}
Using $2^{\ell-1} \le |v_\lambda(y,t)|\le 5\cdot2^{\ell-1}$ for any
$y\in B_{\frac{2^{\ell-1}}{\lambda}}(x),$
from \eqref{dens_estimate} we deduce
\begin{align*}
\int_{B_{\frac{2^{\ell-1}}{\lambda}}(x)\cap\Omega \cap
\p E_\lambda([\lambda t])} (v_\lambda)^2 \,d\cH^n  \le&
25[(n+1)\omega_{n+1}+\omega_n]
(2^{\ell-1})^2\left(\frac{2^{\ell-1}}{\lambda}\right)^n\\
\le &\frac{25 [(n+1)\omega_{n+1}+\omega_n]}{(\kappa/2)^{n+1}\omega_{n+1}}\,
\lambda\int_{B_{\frac{2^{\ell-1}}{\lambda}}(x)\cap (E_\lambda([\lambda t])\Delta
E_\lambda([\lambda t]-1))}|v_\lambda|\,dx.
\end{align*}
Application of the Besicovitch covering theorem  to the collection of
balls $\{B_{\frac{2^{\ell-1}}{\lambda}}(x):\,x\in K(\ell)\cap
\p  E_\lambda([\lambda t])\}$ gives
$$
\int_{K(\ell) \cap\Omega \cap \p E_\lambda([\lambda t])} (v_\lambda)^2 \,d\cH^n \le
\frac{25[(n+1)\omega_{n+1}+\omega_n]\mathfrak{b}(n)}{(\kappa/2)^{n+1}\omega_{n+1}}\,
\lambda \int_{\{2^{\ell-1}\le |v_\lambda| \le 2^{\ell+2}\} \cap
(E_\lambda([\lambda t])\Delta E_\lambda([\lambda t]-1))}|v_\lambda|\,dx.
$$
Now summing up these inequalities over $\ell\in \Z$ with
$\ell\le 1+[\log_2(R(n,\kappa)\lambda^{1/2})],$
and using the properties of $K(\ell)$ and the definition of $\alpha(n,\kappa)$ we get
$$
\int_{\Omega\cap \p E_\lambda([\lambda t])} (v_\lambda)^2 \,d\cH^n \le
\alpha(n,\kappa)\,\lambda\int_{E_\lambda([\lambda t])\Delta
E_\lambda([\lambda t]-1)}|v_\lambda|\,dx.
$$
Observe that by \eqref{monoton_capil} for any $t\ge 1/\lambda$ one has
$$\int_{E_\lambda([\lambda t])\Delta E_\lambda([\lambda t]-1)}|v_\lambda|\,dx
\le\cC(E_\lambda([\lambda t]-1),\Omega) -  \cC(E_\lambda([\lambda t]),\Omega).$$
Thus
$$
\int_{\Omega\cap \p E_\lambda([\lambda t])} (v_\lambda)^2 \,d\cH^n \le
\alpha(n,\kappa)\,\lambda \Big(\cC(E_\lambda([\lambda t]-1),\Omega) -
\cC(E_\lambda([\lambda t]),\Omega)\Big).
$$
Fixing $T>0$ and integrating this inequality in $t\in [0,T]$ we get
\begin{align*}
\int_0^T \int_{\Omega\cap \p E_\lambda([\lambda t])} (v_\lambda)^2 \,d\cH^ndt \le&
\alpha(n,\kappa)\,
\sum\limits_{k=1}^{[T\lambda]+1}\Big(\cC(E_\lambda(k-1),\Omega) -
\cC(E_\lambda(k),\Omega)\Big)\\
\le  &\alpha(n,\kappa)\, \cC(E_0,\Omega) \le
\alpha(n,\kappa)\,P(E_0),
\end{align*}
where  we used \eqref{eq:bound_F-0}.
Now letting $T\to+\infty$ completes the proof.
\end{proof}

\begin{proposition}\label{prop:strange_est}
Let $E_0,\beta$ satisfy \eqref{hyp:main}, $\lambda\ge1$ and $E^+$ be as in Proposition 
\ref{prop:boundedness_of_minimizers}.
Then 
\begin{equation}\label{L_1boundq}
\lambda \int_{1/\lambda}^T |E_\lambda([\lambda t])\Delta E_\lambda([\lambda t]-1)|\, dt\le 
|E^+| + \frac{P(E_0)}{\gamma(n,\kappa)} + 
\frac{2^{n+1}\omega_{n+1}\gamma(n,\kappa)\mathfrak{b}(n)}{\kappa c(n,\kappa)}\,P(E_0)\,T 
\end{equation}
for any  $T>\frac1\lambda.$ Here $\mathfrak{b}(n),\gamma(n,\kappa),c(n,\kappa)$ are defined in
 Section \ref{sec:density_estimates}.
\end{proposition}

\begin{proof}
Let $[\lambda T]=N.$ Clearly, 
$$
\lambda \int_{1/\lambda}^T |E_\lambda([\lambda t])\Delta E_\lambda([\lambda t]-1)|\, dt
=\sum\limits_{k=1}^N |E_\lambda(k)\Delta E_\lambda(k-1)|.
$$
We recall that $E_\lambda(k)\subset E^+$ for all $\lambda\ge1$ and $k\ge0,$ 
by Proposition \ref{prop:boundedness_of_minimizers}. 

If $k=1,$ then
\begin{equation}\label{k_equal_0}
|E_\lambda(1)\Delta E_\lambda(0)| \le |E^+|. 
\end{equation}
Now if $k\ge2,$ we write $E_\lambda(k)\Delta E_\lambda(k-1)$ as a union of 
$A_k$ and $B_k,$ where 
$$
A_k=\Big\{x\in E_\lambda(k)\Delta E_\lambda(k-1):\,\, \dist_{E_{\lambda}(k-1)}(x) >\ell\Big\},
$$
$$
B_k=\Big\{x\in E_\lambda(k)\Delta E_\lambda(k-1):\,\, \dist_{E_{\lambda}(k-1)}(x) \le \ell\Big\}.
$$
where  $\ell:=\frac{\gamma(n,\kappa)}{\lambda}.$
By Chebyshev inequality $|A_k|$ can be estimated using \eqref{monoton_capil} as 
\begin{equation*}
\begin{aligned}
|A_k| \le & \frac{\lambda}{\gamma(n,\kappa)} \int_{E_\lambda(k)\Delta E_\lambda(k-1)} \dist_{E_{\lambda}(k-1)}\,dx
\le  \frac{1}{\gamma(n,\kappa)}\Big(\cC(E_\lambda(k),\Omega) - \cC(E_\lambda(k),\Omega)\Big).
\end{aligned} 
\end{equation*}
Hence, by \eqref{1234}
\begin{equation*}
\sum\limits_{k=2}^N |A_k| \le   \frac{1}{\gamma(n,\kappa)} \sum\limits_{k=2}^N
\Big(\cC(E_\lambda(k),\Omega) - \cC(E_\lambda(k),\Omega)\Big)\le \frac{P(E_0)}{\gamma(n,\kappa)}.
\end{equation*}

Moreover, by definition  $B_k$ can be covered by the family of balls  
$\{B_{2\ell}(x),\,x\in\p E_\lambda(k-1)\}.$ 
Thus, by Besicovitch covering theorem we can find at most countably many  
balls $\{B_\ell(x_j),\,x_j\in\p E_\lambda(k-1)\}$
covering $\Omega\cap \p E_\lambda(k-1).$
Hence, the lower density estimate \eqref{eq:perimeter_est} 
for $E_\lambda(k-1)$ used with $\ell$  implies
\begin{align*}
|B_{2\ell}(x_j)\cap B_k| \le & 
(2^{n+1}\omega_{n+1} \ell) \ell^n \le \frac{2^{n+1}\omega_{n+1}}{c(n,\kappa)} \,\ell\, P(E_\lambda(k-1),B_\ell(x_j)),
\end{align*}
from which it follows that 
\begin{equation*}
\begin{aligned}
\sum\limits_{k=2}^N|B_k| \le & \sum\limits_{k=2}^N \sum\limits_{j\ge 1} 
|B_{2\ell}(x_j)\cap B_k|
\le \frac{2^{n+1}\omega_{n+1}}{c(n,\kappa)} \,\ell\,\sum\limits_{k=2}^N\sum\limits_{j\ge 1} 
P(E_\lambda(k-1),B_\ell(x_j))\\
\le& \frac{2^{n+1}\mathfrak{b}(n)\omega_{n+1}}{c(n,\kappa)} \,\ell\,\sum\limits_{k=2}^N P(E_\lambda(k-1),\Omega).
\end{aligned}
\end{equation*}
Therefore, using \eqref{1234} and $N\le\lambda T,$ we get 
\begin{equation}\label{curv_small_k_katta_2}
\sum\limits_{k=2}^N|B_k| \le \frac{2^{n+1}\mathfrak{b}(n)
\omega_{n+1}\gamma(n,\kappa)}{\kappa c(n,\kappa)} 
\,P(E_0)\,T. 
\end{equation}
Finally, \eqref{L_1boundq} follows from \eqref{k_equal_0}-\eqref{curv_small_k_katta_2}.
\end{proof}

The following error estimate is similar to error estimates shown in
\cite{LS:95,MSS:2016}.

\begin{proposition}[\bf Error estimate]\label{prop:error_estimate}
Let $1\le n\le 6.$ Under assumption \eqref{hyp:main1},
for every $\phi\in C_c^1([0,+\infty)\times \Omega)$ the following 
error-estimate holds:
\begin{equation}\label{buldiye}
\lim\limits_{j\to+\infty} \int_{1/\lambda_j}^{+\infty} \lambda_j
\Bigg(\int_{\Omega} (\chi_{E_{\lambda_j}([\lambda_jt])} -
\chi_{ E_{\lambda_j}([\lambda_jt]-1)})\phi\,dx
-\int_{\Omega\cap\p E_{\lambda_j}([\lambda_jt])}
\sdist_{E_{\lambda_j}([\lambda_jt]-1)}\, \phi \,d\cH^n \Bigg)dt \to0.
\end{equation}

\end{proposition}

\begin{proof}
Let us assume that $\supp \phi\strictlyincluded [0,T)\times\Omega^\epsilon,$
$\Omega^\epsilon:=\R^n\times(\epsilon,+\infty)\subset\Omega$
for some $\epsilon,T>0.$
Let us take $j$ so large that 
\begin{equation}\label{j_large}
4R\lambda_j^{-1/2}<\epsilon, 
\end{equation}
where $R:=R(n,\kappa)$ is defined in 
\eqref{eq:universal constant}.

Given an integer $k\ge1$ set 
\begin{align*}
\Delta_k(j):= & \int_{k/\lambda_j}^{(k+1)/\lambda_j} \lambda_j
\Bigg(\int_{\Omega}(\chi_{E_{\lambda_j}(k)} -
\chi_{ E_{\lambda_j}(k-1})\,\phi \,dx 
-\int_{\Omega\cap \p E_{\lambda_j}(k)}
\sdist_{E_{\lambda_j}(k-1)}\,\phi \, d\cH^n \Bigg)dt.
\end{align*}
We need to estimate
\begin{align*}
\int_{1/\lambda_j}^{T} \lambda_j
\Bigg(\int_{\Omega^\epsilon}(\chi_{E_{\lambda_j}([\lambda_jt])} -
\chi_{ E_{\lambda_j}([\lambda_jt]-1)})\,\phi \,dx 
-\int_{\Omega^\epsilon\cap\p E_{\lambda_j}([\lambda_jt])}
\sdist_{E_{\lambda_j}([\lambda_jt]-1)}\,\phi \, d\cH^n \Bigg)dt
=\sum\limits_{k=1}^N \Delta_k(j), 
\end{align*}
where $N=[\lambda_j T].$

First consider $\Delta_1(j).$
By virtue of Proposition \ref{prop:uniform_L_infty_est} and \eqref{1234}
we get
\begin{align*}
|\Delta_1(j)|= & \Bigg|\int_{1/\lambda_j}^{2/\lambda_j} \lambda_j
\Bigg(\int_{\Omega^\epsilon}(\chi_{E_{\lambda_j}(1)} -
\chi_{ E_{\lambda_j}(0)})\,\phi \,dx 
-\int_{\Omega^\epsilon\cap\p E_{\lambda_j}(1)}
\sdist_{E_{\lambda_j}(0)}\,\phi \, d\cH^n \Bigg)dt\Bigg| \\
\le &\|\phi\|_\infty \Bigg(|E_{\lambda_j}(1)\Delta E_0|+ 
\frac{R(n,\kappa)}{\sqrt{\lambda_j}}\, P(E_{\lambda_j}(1),\Omega)\Bigg)\\
\le & \|\phi\|_\infty \Bigg(|E_{\lambda_j}(1)\Delta E_0|+ 
\frac{R(n,\kappa)}{\kappa \sqrt{\lambda_j}}\, P(E_0)\Bigg).
\end{align*}
By Lemma \ref{lem:behav_big_lambda}, $\Delta_1(j)\to 0$ as $j\to+\infty.$

Now we estimate  $\Delta_k(j),$  $k\ge2.$ We use a 
covering argument with balls
centered at the boundaries of $E_{\lambda_j}(k)$
and $E_{\lambda_j}(k-1).$

First we deal with the regions of low curvature.
We claim that if  $1/2<\sigma_1<\sigma_2<1,$ and 
$x\in \Omega^\epsilon\cap\p E_{\lambda_j}(k)$ is such that  
\begin{equation}\label{dadajon}
\dist_{E_{\lambda_j}(k-1)}(y)\le \lambda_j^{-\sigma_2},\qquad \forall y\in 
B_{R\lambda_j^{-1/2}}(x)\cap (E_{\lambda_j}(k)\Delta E_{\lambda_j}(k-1)), 
\end{equation}
then there exists $\nu:=\nu(x) \in \S^n$ and a continuous increasing 
function $w\in C([0,\infty))$ with $w(0)=0$
such that
\begin{equation}\label{graph_prop}
\begin{gathered}
|\nu_{E_{\lambda_j}(k)}(y) - \nu|\le w(1/\lambda_j) 
\qquad \forall  y\in B_{\lambda_j^{-\sigma_1}}(x)\cap \p E_{\lambda_j}(k),\\
|\nu_{E_{\lambda_j}(k-1)}(y) - \nu|\le w(1/\lambda_j) \qquad 
\forall  y\in B_{\lambda_j^{-\sigma_1}}(x)\cap \p E_{\lambda_j}(k-1).
\end{gathered} 
\end{equation}
Indeed, fix $r\in (0,  \lambda_j^{\sigma_1-1/2}).$
By choice \eqref{j_large} of $j,$ the ball $B_{r\lambda_j^{-\sigma_1}}(x)$ does 
not intersect $\p\Omega,$ and hence
$$
P(E_{\lambda_j}(k),B_{r\lambda_j^{-\sigma_1}}(x)) \le 
P(F,B_{r\lambda_j^{-\sigma_1}}(x)) 
+\lambda_j \int_{F\Delta E_{\lambda_j}(k)} \dist_{E_{\lambda_j}(k-1)}dy 
$$
for every $F\in BV(\Omega,\{0,1\})$ with 
$F\Delta E_{\lambda_j}(k)\strictlyincluded B_{r\lambda_j^{-\sigma_1}}(x).$
Since $\dist_{E_{\lambda_j}(k-1)}(\cdot)$ is $1$-Lipschitz,
by virtue of Proposition \ref{prop:uniform_L_infty_est},
$$
 \dist_{E_{\lambda_j}(k-1)}(y) \le 
 \dist_{E_{\lambda_j}(k-1)}(x)+ |x-y|\le 
  R\lambda_j^{-1/2} +r\lambda_j^{-\sigma_1} 
  \le \frac{R+1}{\lambda_j^{1/2}},\qquad
 y\in  F\Delta E_{\lambda_j}(k),
$$
whence 
$$
\lambda_j \int_{F\Delta E_{\lambda_j}(k)} \dist_{E_{\lambda_j}(k-1)}dy \le 
(R+1)\lambda_j^{1/2} |F\Delta E_{\lambda_j}(k)| 
$$
and
\begin{equation}\label{eq:minimaljon}
P(E_{\lambda_j}(k),B_{r\lambda_j^{-\sigma_1}}(x)) \le 
P(F,B_{r\lambda_j^{-\sigma_1}}(x)) 
+ (R+1)\lambda_j^{-1/2} |F\Delta E_{\lambda_j}(k)|. 
\end{equation}
As $k\ge2,$ 
$$
P(E_{\lambda_j}(k-1),B_{r\lambda_j^{-\sigma_1}}(x)) \le 
P(F,B_{r\lambda_j^{-\sigma_1}}(x)) 
+ (R+1)\lambda_j^{-1/2} |F\Delta E_{\lambda_j}(k-1)|  
$$
whenever $F\in BV(\Omega,\{0,1\})$ satisfies 
$F\Delta E_{\lambda_j}(k-1)\strictlyincluded B_{r\lambda_j^{-\sigma_1}}(x).$
Set 
$$
E_{\lambda_j}^{\sigma_1}(k):= \frac{E_{\lambda_j}(k)-x}{\lambda_j^{\sigma_1}},
\qquad 
E_{\lambda_j}^{\sigma_1}(k-1):=\frac{E_{\lambda_j}(k-1)-x}{\lambda_j^{\sigma_1}}.
$$
By virtue of \eqref{eq:minimaljon} these sets satisfy
$$
P(E_{\lambda_j}^{\sigma_1}(s), B_r) \le P(F, B_r) + 
\lambda_j^{\frac12-\sigma_1} (R+1)
|E_{\lambda_j}^{\sigma_1}(s)\Delta F|
$$
for any $r\in (0,\lambda_j^{\sigma_1 - 1/2})$ and 
$F\Delta E_{\lambda_j}^{\sigma_1}(s)\strictlyincluded 
B_r(0),$ $s=k,k-1.$
Hence, $E_{\lambda_j}^{\sigma_1}(s),$ $s=k,k-1$ is 
an $((R+1)\lambda_j^{1/2-\sigma_1}, \lambda_j^{\sigma_1 - 1/2})$-minimizer 
of the perimeter (see \cite[ Section 23]{Mag12}). Since 
$\sigma_1>1/2,$ $\lambda_j^{\frac12-\sigma_1} (R+1)\to0$ 
as $j\to+\infty$ and therefore, by the compactness 
\cite[Proposition 23.13]{Mag12}, 
up to a (not relabelled) subsequence, 
$$
E_{\lambda_j}^{\sigma_1}(s) \to E_s^{\sigma_1}\quad 
\text{in $L_\loc^1(\R^{n+1})$ as $j\to+\infty,$ $s=k,k-1.$}
$$
where $E_s^{\sigma_1},$ $s=k,k-1,$ is a local minimizer  of the perimeter.
Since  $n\le6,$ by virtue of  \cite[Theorem 17.3]{Gius84} 
$E_k^{\sigma_1}$ and $E_{k-1}^{\sigma_1}$   are
half-spaces. Moreover, by  hypothesis
$$
\dist_{E_{\lambda_j}^{\sigma_1}(k-1)} (z)\le 
\lambda_j^{\sigma_1-\sigma_2}\quad \forall 
z\in B_{\lambda_j^{\sigma_1-1/2}}(0)\cap 
\big(E_{\lambda_j}^{\sigma_1}(k)\Delta E_{\lambda_j}^{\sigma_1}(k-1)\big),
$$
and, therefore, $E_k^{\sigma_1}=E_{k-1}^{\sigma_1},$ i.e. 
there exists $\nu\in \S^{n}$ such that 
$$
E_k^{\sigma_1}=E_{k-1}^{\sigma_1} = \{z\in \R^{n+1}:\,\, z\cdot \nu <0\}.
$$
By \cite[Theorem 26.6]{Mag12} $\nu_{E_{\lambda_j}^{\sigma_1}}(s)\to \nu$ 
uniformly in $B_1(0)$
implying \eqref{graph_prop}, which concludes the proof of the claim.

Recall that 
if  $\sigma_1,\sigma_2$ are as above and 
$x\in \Omega^\epsilon\cap\p E_{\lambda_j}(k)$ satisfies \eqref{dadajon},
then by virtue of \cite[Corollary 4.2.2]{MSS:2016}  
there exists $C(n)>0$ such that
\begin{equation}\label{gishmatqul}
\begin{aligned}
\Bigg|\int_{\cyl{\rho}{\rho}(x,\nu)}
(\chi_{E_{\lambda_j}(k)}-\chi_{E_{\lambda_j}(k-1)})dx - &
\int_{\cyl{\rho}{\rho}(x,\nu) \cap \p E_{\lambda_j}(k)} 
\sdist_{E_{\lambda_j}(k-1)}\,d\cH^n\Bigg|\\
\le & C(n)w(1/\lambda_j)\, \int_{\cyl{\rho}{\rho}(x,\nu)}
|\chi_{E_{\lambda_j}(k)}-\chi_{E_{\lambda_j}(k-1)}|dx,
\end{aligned} 
\end{equation}
where $\rho:=\lambda_j^{-\sigma_1}/2,$ and 
$$
\cyl{\rho}{\rho}(x,\nu): = \big\{y\in \R^{n+1}:\,\, |(y-x)\cdot\nu|<\rho,\,
\sqrt{|y-x|^2-|(y-x)\cdot\nu|^2}<\rho\big\}.
$$

Now we estimate $\Delta_k(j),$ $k\ge2.$ 
Set $\sigma_2:= \frac{2n+5}{4(n+2)},$ $\sigma_1\in (1/2,\sigma_2)$
and
$$
\mathcal{O}:=\Big\{\cyl{\rho}{\rho}(x,\nu(x)):\,\,\text{$x$ satisfies 
\eqref{dadajon}}\Big\}.
$$
Then by virtue of \eqref{gishmatqul},
\begin{equation}\label{disjoint_partisso}
\begin{aligned}
\Bigg|\int_{\bigcup\limits_{\mathcal O} \cyl{\rho}{\rho} } &
(\chi_{ E_{\lambda_j}(k) } - 
\chi_{ E_{\lambda_j}(k-1) }) \phi\,dy -
\int_{\bigcup\limits_{\mathcal O} \cyl{\rho}{\rho} \cap  
\p E_{\lambda_j}(k)} 
\sdist_{ E_{\lambda_j}(k-1) }\phi\, d\cH^n\Bigg| \\
 \le& C(n)w(1/\lambda_j) \Big(|E_{\lambda_j}(k)\Delta E_{\lambda_j}(k)| 
 + \lambda_j^{-\sigma_2}P(E_{\lambda_j}(k))\Big).  
\end{aligned} 
\end{equation}

It remains to estimate the error in $\Omega^\epsilon\setminus 
\bigcup\limits_{\mathcal O} \cyl{\rho}{\rho}.$
By virtue of Proposition \ref{prop:uniform_L_infty_est}
$\Omega^\epsilon\setminus 
\bigcup\limits_{\mathcal O} \cyl{\rho}{\rho} \cap (E_{\lambda_j}(k)\Delta 
E_{\lambda_j}(k-1))$ can be covered 
by the family of balls
$$
\mathtt{B}:=\{B_r(x):\,x\in \Omega^\epsilon\cap 
\p E_{\lambda_j}(k-1)\},\qquad r:=R\lambda_j^{-1/2}. 
$$ 
By the standard Besicovitch
theorem we can extract a finite collection 
$\{B_r(x_i)\}\subset \mathtt{B}$
such that each point of 
$\Omega^\epsilon\cap(E_{\lambda_j}(k)\Delta E_{\lambda_j}(k-1))$
is covered with at most $\mathfrak{b}(n)$ elements 
of $\{B_r(x_i)\}.$ 
First we handle the error in each $B_r(x).$
By the definition, there exists $y\in O_\rho(x_i)\cap 
(E_{\lambda_j}(k) \Delta E_{\lambda_j}(k-1))$ 
such that $\dist_{E_{\lambda_j}(k-1)}(y)\ge \lambda_j^{-\sigma_2}.$ 
Then clearly, 
$$
\dist_{E_{\lambda_j}(k-1)}(z)\ge \lambda_j^{-\sigma_2}/2,
\quad \forall z\in B_{\lambda_j^{-\sigma_2}/2}(y)\subseteq
 E_{\lambda_j}(k) \Delta E_{\lambda_j}(k-1).
$$
Applying density estimates in Remark \ref{rem:uncons_dens_est} to 
$B_{\lambda_j^{-\sigma_2}/2}(y)$ we get
$$
\int_{B_{\lambda_j^{-\sigma_2}/2}(y)\cap (E_{\lambda_j}(k)\Delta 
E_{\lambda_j}(k-1))} 
\dist_{E_{\lambda_j}(k-1)}\,dz \ge 
\left(\frac{\kappa}{2}\right)^{n+1} \,\omega_{n+1}
\left(\frac{1}{2\lambda_j^{\sigma_2}}\right)^{n+2}.
$$
Therefore, by choice of $\sigma_2,$
\begin{align*}
\int_{B_r(x_i)} |\chi_{E_{\lambda_j}(k)} - \chi_{E_{\lambda_j}(k-1)}|\,dx 
 \le & \omega_{n+1} \left(R\lambda_j^{-1/2}\right)^{n+1}\\
 \le & \left(\frac{4R}{\kappa}\right)^{n+1} 2\lambda_j^{3/4} 
 \int_{B_{\lambda_j^{-\sigma_2}/2}(y)\cap (E_{\lambda_j}(k)\Delta E_{\lambda_j}(k-1))}
 \dist_{E_{\lambda_j}(k-1)}\,dz\\
 \le &\left(\frac{4R}{\kappa}\right)^{n+1} 2\lambda_j^{3/4} 
 \int_{ B_{r} (x_i)\cap (E_{\lambda_j}(k)\Delta E_{\lambda_j}(k-1))}
 \dist_{E_{\lambda_j}(k-1)}\,dz.
\end{align*}
Moreover, comparing $E_{\lambda_j}(k)$ to 
$E_{\lambda_j}(k)\setminus B_r(x_i)$ (and using $B_r(x_i)\strictlyincluded \Omega$)
we find
$$
P(E_{\lambda_j}(k), B_r(x_i)) \le C(n, \diam(E^+)) (R\lambda_j^{-1/2} )^n,
$$
whence, by Proposition \ref{prop:uniform_L_infty_est}),
\begin{align*}
&\int_{B_r(x_i)\cap \p E_{\lambda_j}(k)} 
\dist_{E_{\lambda_j}(k-1)} \,d\cH^n \le 
 C(n,\diam(E^+)) (R\lambda_j^{-1/2} )^{n+1} \\
& \le \frac{C(n,\diam(E^+))}{\omega_{n+1}}
 \left(\frac{4R}{\kappa}\right)^{n+1} 2\lambda_j^{3/4}   
 \int_{ B_r(x_i)\cap (E_{\lambda_j}(k)\Delta E_{\lambda_j}(k-1))}
 \dist_{E_{\lambda_j}(k-1)}\,dz.
\end{align*}
Thus
\begin{equation*}
\begin{aligned}
\Bigg|\int_{ B_r(x_i) } &(\chi_{ E_{\lambda_j}(k) } - 
\chi_{ E_{\lambda_j}(k-1) })\phi\,dy -
\int_{ B_r(x_i)\cap \p E_{\lambda_j}(k)} 
\sdist_{ E_{\lambda_j}(k-1) }\phi\, d\cH^n\Bigg| \\
\le &\|\phi\|_\infty C(n,\kappa,\diam(E^+)) \lambda_j^{3/4} 
\int_{B_r(x_i)\cap (E_{\lambda_j}(k)\Delta E_{\lambda_j}(k-1))}
 \dist_{E_{\lambda_j}(k-1)}\,dz,
\end{aligned}
\end{equation*}
and therefore,
\begin{equation*} 
\begin{aligned}
\Bigg|\int_{\Omega^\epsilon\setminus \bigcup\limits_{\mathcal O} \cyl{\rho}{\rho} } &
(\chi_{ E_{\lambda_j}(k) } - 
\chi_{ E_{\lambda_j}(k-1) }) \phi\,dy -
\int_{\Omega^\epsilon\setminus\bigcup\limits_{\mathcal O} \cyl{\rho}{\rho} \cap  
\p E_{\lambda_j}(k)} 
\sdist_{ E_{\lambda_j}(k-1) }\eta_l\phi\, d\cH^n\Bigg| \\
 \le& \mathfrak{b}(n)\|\phi\|_\infty C(n,\kappa,\diam(E^+)) \lambda_j^{3/4} 
\int_{E_{\lambda_j}(k)\Delta E_{\lambda_j}(k-1)}
 \dist_{E_{\lambda_j}(k-1)}\,dz.
\end{aligned}
\end{equation*}
By \eqref{monoton_capil} 
$$
\lambda_j^{3/4} 
\int_{E_{\lambda_j}(k)\Delta E_{\lambda_j}(k-1)}
 \dist_{E_{\lambda_j}(k-1)}\,dz\le \lambda^{-1/4}
 \Big(\cC(E_{\lambda_j}(1),\Omega) - \cC(E_{\lambda_j}(N),\Omega)\Big),
$$
and thus we have also
\begin{equation}\label{qolgan_joyi}
\begin{aligned}
\Bigg|\int_{\Omega^\epsilon\setminus \bigcup\limits_{\mathcal O} \cyl{\rho}{\rho} } &
(\chi_{ E_{\lambda_j}(k) } - 
\chi_{ E_{\lambda_j}(k-1) }) \phi\,dy -
\int_{\Omega^\epsilon\setminus\bigcup\limits_{\mathcal O} \cyl{\rho}{\rho} \cap  
\p E_{\lambda_j}(k)} 
\sdist_{ E_{\lambda_j}(k-1) }\eta_l\phi\, d\cH^n\Bigg| \\
 \le& \mathfrak{b}(n)\|\phi\|_\infty C(n,\kappa,\diam(E^+)) \lambda_j^{-1/4} 
\Big(\cC(E_{\lambda_j}(1),\Omega) - \cC(E_{\lambda_j}(N),\Omega)\Big).
\end{aligned}
\end{equation}

Combining \eqref{disjoint_partisso} and \eqref{qolgan_joyi}, we obtain
\begin{equation*}
\begin{aligned}
|\Delta_k(j)| \le  &  
C(n)w(1/\lambda_j) \Big(|E_{\lambda_j}(k)\Delta E_{\lambda_j}(k)| 
 + \lambda_j^{-\sigma_2}P(E_{\lambda_j}(k))\Big)\\
 &+
\mathfrak{b}(n)\|\phi\|_\infty C(n,\kappa,\diam(E^+)) \lambda_j^{-1/4} 
\Big(\cC(E_{\lambda_j}(1),\Omega) - \cC(E_{\lambda_j}(N),\Omega)\Big).
\end{aligned}
\end{equation*}
Therefore, and  we get 
\begin{align*}
\sum\limits_{k=2}^N  |R_k(j)| \le & 
\mathfrak{b}(n) \Big(C(n)w(1/\lambda_j) \|\phi\|_\infty  + 
\|\nabla \phi\|_\infty \lambda_j^{-\sigma_1}\Big)
\sum\limits_{k=2}^N  |E_{\lambda_j}(k)\Delta E_{\lambda_j}(k-1)| \\
&+ \mathfrak{b}(n)\|\nabla \phi\|_\infty \lambda_j^{-\sigma_1-\sigma_2}\, 
\frac{(N-1) P(E_0)}{\kappa}\\
&+ \mathfrak{b}(n)\|\phi\|_\infty C(n,\kappa,\diam(E^+)) \lambda_j^{-1/4} 
.
\end{align*}
Now applying Proposition \ref{prop:strange_est}, \eqref{1234},
\eqref{monoton_capil}  and the 
relation $N\le T\lambda_j$  we obtain
\begin{align*}
\sum\limits_{k=2}^N  |\Delta_k(j)| \le &  
C(n,\kappa,P(E_0), T)\,\|\phi\|_\infty \,w(1/\lambda_j)   
+ C(n)\|\phi\|_\infty \lambda_j^{-\sigma_1}(N-1)P(E_0)\\
 &+ \mathfrak{b}(n)\|\nabla \phi\|_\infty 
 \lambda_j^{1-\sigma_1-\sigma_2}\, \frac{P(E_0)}{\kappa}\,T 
 + \mathfrak{b}(n)\|\phi\|_\infty 
 C(n,\kappa,\diam(E^+)) \lambda_j^{-1/4} P(E_0), 
\end{align*}
This estimate, the assumption $\sigma_1+\sigma_2>1,$ 
and the fact that $R_1(j)\to0$ as $j\to+\infty$ imply 
 \eqref{buldiye}.
\end{proof}

\begin{proof}[Proof of Theorem \ref{teo:distr_sol}]
Lemma \ref{lem:l2bound_velocity},
\eqref{uniform_per_est} and \cite[Theorem 4.4.2]{Hutch:86}  imply that
there exist a (not relabelled) subsequence and a function
$v:[0,+\infty)\times \Omega\to \R$ satisfying
\eqref{l2_bound_velic}-\eqref{weak_conv_some}.
In particular, from \eqref{l2_bound_velic} it follows that
$H_{E(t)}:=v(t,\cdot)\big|_{\Omega\cap\p^*E(t)}\in L^2(\Omega\cap\p^*E(t),
\cH^n\res (\Omega\cap\p^*E(t)))$ for a.e. $t>0.$
Let us prove that $H_{E(t)}$ is the distributional mean curvature of $E(t)$
for a.e. $t\ge0.$
Fixing $t\ge0,$ by the divergence formula \eqref{integ_by_parts}
for any $\phi\in C_c^1(\R^{n+1},\R^{n+1})$ one has
\begin{equation*} 
\int_{E_{\lambda_j}([\lambda_jt])} \div \phi dx -
\int_{\Omega\cap \p^*E_{\lambda_j}([\lambda_jt])}
\phi\cdot\nu_{E_{\lambda_j}([\lambda_jt])}\,d\cH^n =
\int_{\p\Omega\cap \p^* E_{\lambda_j}([\lambda_jt])}\phi_{n+1}d\cH^n.
\end{equation*}
Hence, from \eqref{L1_conv_GMM} and \eqref{eq:wcms}
we get
\begin{equation}\label{DG_div_formula_limit}
\int_{E(t)} \div \phi dx - \int_{\Omega\cap \p^*E(t)}
\phi\cdot\nu_{E(t)}\,d\cH^n =
\lim\limits_{j\to+\infty} \int_{\Trace(E_{\lambda_j}([\lambda_jt]))}\phi_{n+1}d\cH^n.
\end{equation}
The left-hand-side of \eqref{DG_div_formula_limit} is
$
\int_{\Trace(E(t))}\phi_{n+1}d\cH^n,
$
therefore, 
\begin{equation}\label{con_weak_m}
\cH^n\res\Trace(E_{\lambda_j}([\lambda_jt]))\overset{w^*}{\rightharpoonup}
\cH^n\res\Trace(E(t)) \qquad \text{as $j\to+\infty.$}
\end{equation}
Combining this with \eqref{uniform_per_est} we get
\begin{equation*}
\cH^n\res \p^*E_{\lambda_j}([\lambda_jt])
\overset{w^*}{\rightharpoonup} \cH^n\res \p^* E(t)
\quad \text{ as $j\to+\infty$ for a.e. $t\ge0.$} 
\end{equation*}

Take   $\eta\in C_c^1([0,+\infty))$ and  an admissible
$X\in C_c^1(\overline \Omega,\R^{n+1}).$
By \eqref{uniform_per_est} and \cite[formula (4.2)]{MSS:2016}  for a.e. $t\ge0$  and
for every $F\in C_c(\R^{n+1} \times \R^{n+1})$ one has
\begin{equation}\label{varifold_conv}
\lim\limits_{j\to+\infty} \int_{\Omega\cap \p^*E_{\lambda_j}([\lambda_jt])}
F(x,\nu_{E_{\lambda_j}([\lambda_jt])}(x))\,d\cH^n =
\int_{\Omega\cap \p^*E(t)}F(x,\nu_{E(t)}(x))\,d\cH^n.
\end{equation}
In particular, taking  $F\in C_c(\overline{\Omega} \times \R^{n+1})$ such that
$F(x,\xi)= \div X(x) - \xi\cdot  \nabla X(x) \xi$ in
$\Omega\times\{|\xi|\le2\},$
by the dominated convergence theorem,
\eqref{dist_curvat} and \eqref{weak_conv_some},
for $\Psi(t,x)=\eta(t)X(x)$ we establish
\begin{equation*} 
\begin{aligned}
\int_0^{+\infty} &\eta(t) \int_{\Omega\cap \p^*E(t)}  F(x,\nu_{E(t)}(x))d\cH^ndt
=
\lim\limits_{j\to+\infty} \int_0^{+\infty} \int_{\Omega\cap \p^*E_{\lambda_j}([\lambda_jt])} \eta(t)
F(x,\nu_{E_{\lambda_j}([\lambda_jt])})d\cH^ndt\\
=&  \lim\limits_{j\to+\infty} \int_0^{+\infty}
\int_{\Omega\cap \p^*E_{\lambda_j}([\lambda_jt])}
v_{\lambda_j}\nu_{E_{\lambda_j}([\lambda_jt])}\cdot
\Psi(t,x)d\cH^ndt\\
= & \int_0^{+\infty} \int_{\Omega\cap \p^*E(t) }
v \nu_{E(t)}\cdot \Psi(t,x)d\cH^ndt
= \int_0^{+\infty}\eta(t) \int_{\Omega\cap \p^*E(t) }
H_{E(t)} \nu_{E(t)}\cdot X\,d\cH^ndt.
\end{aligned}
\end{equation*}
Since $\eta\in C_c^1([0,+\infty))$ is arbitrary, for a.e. $t\ge0$
we get
$$
\int_{\Omega\cap \p^*E(t)}  ( \div X - \nu_{E(t)}\cdot (\nabla X)\nu_{E(t)} )d\cH^n=
\int_{\Omega\cap \p^*E(t) }
H_{E(t)} \nu_{E(t)}\cdot X\,d\cH^n,
$$
hence $H_{E(t)}$ is the generalized mean curvature of $\Omega\cap \p^*E(t).$

Let us show \eqref{normal_velocity}. Take
$\phi\in C_c^1([0,+\infty)\times\Omega).$ By a change of variables we have
\begin{align*}
\int_{1/\lambda_j}^{+\infty} \Big[\int_{ E_{\lambda_j}([\lambda_jt])}  \phi dx - &
\int_{E_{\lambda_j}([\lambda_jt]-1)} \phi dx \Big]dt \\
=& \int_{1/\lambda_j}^{+\infty} \int_{E_{\lambda_j}([\lambda_jt])} (\phi(t,x) - \phi(t+1/\lambda_j,x))dxdt
-\frac{1}{\lambda_j}\int_{E(0)} \phi(x,0)\,dx.
\end{align*}
Since $E(0)=E_0$, from \eqref{gelder_condit} we get
$$
\lim\limits_{j\to+\infty} \int_{1/\lambda_j}^{+\infty} \lambda_j
\Big[\int_{ E_{\lambda_j}([\lambda_jt])}  \phi dx -
\int_{E_{\lambda_j}([\lambda_jt]-1)} \phi dx \Big]dt =
-\int_0^{+\infty} \int_{E(t)} \frac{\p \phi}{\p t}\,(t,x)\,dxdt - \int_{E_0} \phi(x,0)dx.
$$
Therefore,  \eqref{buldiye}, \eqref{weak_conv_velocities} and the definition of $H_{E(t)}$
imply
\begin{align*}
\int_0^{+\infty} \int_{E(t)} \p_t \phi\,dxdt + \int_{E_0} \phi(x,0)dx
=& \lim\limits_{j\to+\infty} \int_{0}^{+\infty}
\int_{\Omega\cap\p E_{\lambda_j}([\lambda_jt])}v_{\lambda_j}  \phi \,d\cH^n dt\\
= &\int_{0}^{+\infty} \int_{\Omega\cap\p^* E(t)} H_{E(t)}  \phi \,d\cH^n dt.
\end{align*}

(ii)  Take an admissible $X\in C_c^1(\overline{\Omega},\R^{n+1})$ and $\eta\in
C_c^1([0,+\infty)).$ From   \eqref{first_variatsion}
\begin{equation}\label{birinchi-var}
\begin{gathered}
\int_0^{+\infty} \eta(t) \int_{\Omega\cap \p^*E_{\lambda_j}([\lambda_jt])}
\Big(\div X - \nu_{E_{\lambda_j}([\lambda_jt])} \cdot (\nabla X)
\nu_{E_{\lambda_j}([\lambda_jt])}\Big)\,d\cH^ndt\\
- \int_0^{+\infty} \eta(t)\int_{\Omega\cap \p^*E_{\lambda_j}([\lambda_jt])}
v_{\lambda_j}\,X\cdot
\nu_{E_{\lambda_j}([\lambda_jt])}\,d\cH^ndt\\
=   \int_0^{+\infty}\eta(t)\int_{\p^* \Trace(E_{\lambda_j}([\lambda_jt]))}
\beta\,X'\cdot \nu_{\Trace (E_{\lambda_j}([\lambda_jt]))}'\,d\cH^{n-1}.
\end{gathered}
\end{equation}
Let $\{\lambda_{j_l}\}_{l\ge1}$ be any subsequence of $\{\lambda_j\}.$
By the uniform bound \eqref{un.bund.pe} on the perimeters  and by compactness
there exists a further subsequence $\{\lambda_{j_{l_k}}\}_{k\ge1}$
of $\{\lambda_{j_l}\}_{l\ge1}$ and a set $\hat F\in BV(\R^n,\{0,1\})$ such that
$\Trace(E_{j_{l_k}}([j_{l_k}t])) \to \hat F$ in $L^1(\R^n)$
and\footnote{Arguing, for example, as in \eqref{int_parts}.}
$$
\nu_{\Trace (E_{\lambda_{j_{l_k}}}([\lambda_{j_{l_k}}t]) )}'\,
\cH^{n-1}\res \p^*\Trace (E_{\lambda_{j_{l_k}}}([\lambda_{j_{l_k}}t]))
\overset{w^*}{\rightharpoonup}
\nu_{\hat F}'\,\cH^{n-1}\res \p^* \hat F\quad\text{as $k\to+\infty$}
$$
for a.e. $t\ge0.$ By \eqref{con_weak_m} for every $\phi\in C_c(\R^n)$ we have 
$$
\int_{ \Trace(E(t)) } \phi\,d\cH^n =\lim\limits_{k\to+\infty}
\int_{ \Trace(E_{\lambda_{j_{l_k}}}([\lambda_{j_{l_k}}t])) } \phi\,d\cH^n=
\int_{ \hat F } \phi\,d\cH^n.
$$
Whence, $\hat F = \Trace(E(t)).$   Therefore 
$$
\nu_{\Trace (E_{\lambda_j}([\lambda_jt]) )}'\,
\cH^{n-1}\res \p^*\Trace (E_{\lambda_j}([\lambda_jt]))
\overset{w^*}{\rightharpoonup}
\nu_{\Trace (E(t))}'\,\cH^{n-1}\res \p^* \Trace E(t)\quad\text{as $j\to+\infty.$}
$$
Now taking  limit in \eqref{birinchi-var}, using  
\eqref{varifold_conv},\eqref{weak_conv_some}  and applying
the dominated convergence theorem on the right-hand-side
we get \eqref{contactanglejon}.
\end{proof}



\appendix

\section{Existence of minimizers for some functionals}\label{sec:error_est}

\renewcommand{\appendixname}{}

In this section we prove an existence result  for  minimum problems
of type
\begin{equation}\label{eq.min.problem}
\inf\limits_{E\in BV(\Omega,\{0,1\})} \cG(E),\qquad \cG(E):=\cC(E,\Omega)+\cV(E), 
\end{equation}
where  $\cV:BV(\Omega,\{0,1\})\to(-\infty,+\infty].$
Since $\cC(\cdot,\Omega)$ is finite in $BV(\Omega,\{0,1\}),$ the functional 
$\cG $ is well-defined in $BV(\Omega,\{0,1\}).$ We study 
\eqref{eq.min.problem} under the following hypotheses on $\cV:$

\begin{hypothesis}\label{hyp:2}
\begin{itemize}
\item[(a)] $\cV $ is bounded from below in 
$BV(\Omega,\{0,1\})$ and there  exists a cylinder
$\cyl{r}{\height}\subset\Omega,$ $\height>1$ 
such that $\cV(\cyl{r}{\height})<+\infty;$ 
\item[(b)] $\cV(E)\ge \cV(E\cap   \cyl{\rho}{l} )$ 
for any $E\in BV(\Omega,\{0,1\}),$ $\rho\in (r,+\infty],$ 
and $l\in (\height-1, \height+1);$
\item[(c)] $\cV(E)\ge \cV(E\setminus (\cyl{\rho_1}{\height}\setminus
\overline{ \cyl{\rho_2}{\height} }))$ 
for any $E\in BV(\Omega,\{0,1\})$ and $r<\rho_2<\rho_1<+\infty;$
\item[(d)] $\cV $ is $L^1(\Omega)$-lower semicontinuous in $BV(\Omega,\{0,1\}).$
\end{itemize} 
\end{hypothesis}

\begin{example}\label{exam:1}
Besides \eqref{example_for_V} the following functionals 
$\cV:BV(\Omega,\{0,1\})\to(-\infty,+\infty]$ 
satisfy Hypothesis \ref{hyp:2}:
\begin{itemize}
\item[1)] given   $f\in L_\loc^1(\Omega)$ with
$f\ge0$ a.e. in $\Omega\setminus \cyl{r}{l}$  
for some $r,l>0,$  
$$\cV(E)=\int_Efdx.$$
In particular, we may take 
$f=\lambda \sdist_{E_0}$ with $\emptyset\ne E_0\in BV(\Omega,\{0,1\})$ and 
$E_0\subset \cyl{r}{h}$ so that by \eqref{representing_ATW}
$\cG$ coincides with $\cA(\cdot,E_0,\lambda) + \int_{E_0}\sdist_{E_0}dx.$

\item[2)] Given a bounded set $E_0\in BV(\Omega,\{0,1\}),$  
$\cV(E)=|E\Delta E_0|^p,$ $p>0.$
\end{itemize}
\end{example}

Given $\cV$ satisfying Hypothesis \ref{hyp:2} set
$$
\mathfrak{a}:= \kappa^{-1} \left(\sup\limits_{R>r}
\inf\limits_{E\in BV(\cyl{R}{\height},\{0,1\})} \cG(E) - \inf\cV\right).
$$ 
Clearly, $\kappa \mathfrak{a} \le \cG(\cyl{r}{\height})-\inf\cV,$
hence $\inf \cG<+\infty.$

In view of the previous observation, once we prove 
the next theorem, the proof of Theorem 
\ref{teo: unconstrained minimizer} follows.

\begin{theorem}[\textbf{Existence of minimizers and uniform bound}]
\label{teo: unconstrained minimizer1}
Suppose that Hypothesis  \ref{hyp:2} holds. Suppose
also $\beta\in L^\infty(\dOm)$  
and there exists $\kappa \in (0,\frac{1}{2}]$ 
such that $-1\le \beta \le 1-2\kappa$ $\cH^n$-a.e 
on $\p\Omega.$  Then the minimum
problem
\begin{equation*}
\inf_{E \in BV(\Omega,\{0,1\})} \cG(E)
\end{equation*}
has a solution. Moreover, any minimizer is contained in $\cyl{\cR_0}{\height},$
where\footnote{One could refine the expression of $\cR_0$ 
using the isoperimetric
inequality \cite{DG:58-1}, but we do not need this here.}
\begin{equation}\label{eq:universal constant1}
\cR_0:= r+1+\max\Big\{8^{n^2+n+1}\,
\mathfrak{a}^{\frac{n+1}{n}},
\,4\mu(\kappa,n)\Big\} 
\end{equation}
and $\mu(\kappa,n)$ is defined in Section 
\ref{subsec:exist_min_ATW}.
\end{theorem}
 
\begin{remark}\label{rem:muhim_coef}
In case of Example \ref{exam:1} 1) with  $f=\lambda\sdist_{E_0}$ for some 
$\cyl{r}{\height} \supseteq E_0,$ 
$$
\kappa \mathfrak{a} \le \kappa \sup\limits_{R>r}
\inf\limits_{E\in BV(\cyl{R}{\height},\{0,1\})}
\cA(E,E_0,\lambda) \le\kappa \cA(E_0,E_0,\lambda)=\kappa 
\cC(E_0,\Omega)\le \kappa P(E_0).
$$
Hence,  $\cR_0\le R_0,$  where $R_0$ is defined in 
\eqref{eq:universal constant}. The same is true if $\cV$ 
is as in \eqref{example_for_V}.
\end{remark}

The assumption on $\beta$ and  the $L^1(\Omega)$-lower 
semicontinuity of $\cC(\cdot,\Omega)$
(Lemma \ref{lem: lsc_of_F0}) imply the $L^1(\Omega)$-lower  
semicontinuity of $\cG.$ 
Moreover,  the  coercivity \eqref{eq:bound_F-0} of 
$\cC(\cdot,\Omega),$  Hypothesis  \ref{hyp:2} (a)
and \eqref{eq:bound_F-0_zero}
 imply the coercivity of $\cG:$
\begin{equation}\label{new_coer}
\cG(E)\ge \kappa P(E) + \inf \cV \qquad \forall E\in BV(\Omega,\{0,1\}).  
\end{equation}

The main problem in the proof of existence of
minimizers of $\cG $ is the lack of compactness due to the 
unboundedness of $\Omega.$ 
However, for every $R>0$ inequality \eqref{new_coer}, 
the compactness theorem in
$BV(\cyl{R}{\height},\{0,1\})$ (see for instance \cite[Theorems 3.23 and 3.39]{AFP:00})
and the lower semicontinuity of $\cG $
imply that  there exists a solution $E^R\in BV(\cyl{R}{\height},\{0,1\})$ of
$$
\inf\limits_{E\in BV(\cyl{R}{\height},\{0,1\})} \cG(E).
$$

To prove Theorem \ref{teo: unconstrained minimizer1}
we mainly follow  \cite[Section 4]{LCAM07}, where the existence of
volume preserving minimizers of $\cC(\cdot,\Omega)$ has been shown.
We need two preliminary lemmas.
As in \cite[Section 3]{LCAM07}
first we show that one can choose a minimizing sequence
consisting of bounded sets.

\begin{lemma}[\bf Truncations with horizontal hyperplanes and
vertical cylinders ]\label{lem: minmizers in cylindr}
Suppose that Hypothesis  \ref{hyp:2} holds. Then
\begin{equation}\label{eq:min_R1}
\inf\limits_{E\in BV(\Omega,\{0,1\})} \cG(E) =
\inf\limits_{R>0} \inf\limits_{E\in BV(\cyl{R}{\height},\{0,1\})} \cG(E).
\end{equation}
\end{lemma}

\begin{proof}

We need two intermediate steps. The first step
concerns truncations with horizontal hyperplanes.

{\bf Step 1.} We have
\begin{equation}\label{eq231}
\inf\limits_{E\in BV(\Omega,\{0,1\})} \cG(E) =
 \inf\limits_{E\in BV(\Omega_{\height},\{0,1\})} \cG(E).
\end{equation}

Indeed, it suffices to show that if $E\setminus \Omega_{\height-\frac14} \ne\emptyset,$ then
\begin{equation*} 
\cG(E) \ge \cG\big(E\cap\overline{\Omega_{\height-\frac12}}\big).
\end{equation*}
Clearly, $E$ and $E\cap \overline{\Omega_{\height-\frac12}}$ have the same trace on $\p\Omega$
and thus
$$
\int_{\p\Omega} [1+\beta]\, \chi_E \,d\cH^n =
\int_{\p\Omega} [1+\beta]\, \chi_{ E\cap\overline{\Omega_{\height-\frac12}} } \,d\cH^n.
$$
From the comparison theorem of \cite[page 216]{LAln}
we have
$$
P(E) > P\big(E\cap\overline{\Omega_{\height-\frac12}} \big).
$$
By Hypothesis \ref{hyp:2} (b) we have also
$$
\cV(E)\ge \cV(E\cap \overline{ \Omega_{\height-\frac12}}),
$$
therefore from the definition of $\cG$ we get even the strict inequality 
\begin{align}\label{strict_ineq1}
\cG(E) > \cG(E\cap\overline{\Omega_{\height-\frac12}}).
\end{align}

The second step is more delicate and concerns truncations with
the lateral boundary of vertical cylinders.

{\bf Step 2.} For any $\epsilon\in (0,1)$
there exists $R_\epsilon>r$ and $E_\epsilon\in 
BV(\cyl{R_\epsilon}{\height},\{0,1\})$ such that
$$
\cG(E_\epsilon) \le \inf\limits_{E\in BV(\Omega_{\height},\{0,1\})}
 \cG(E) +\epsilon.
$$

Indeed, according to Step 1 and Hypothesis  \ref{hyp:2} (a), given $\epsilon>0$
there exists $F_\epsilon\in BV(\Omega_{\height},\{0,1\})$ with $F_\epsilon\subset
\overline{\Omega_{\height-\frac14}}$ such that
$$
\cG(F_\epsilon) < \inf\limits_{E\in BV(\Omega,\{0,1\})} \cG(E)
 +\frac{\epsilon}{2} <+\infty.$$
Since  $|F_\epsilon|<+\infty,$  for sufficiently large $R>r$ one has
$$
|F_\epsilon\cap (\cyl{R+1}{\height}\setminus \cyl{R}{\height})| = \int_{R}^{R+1}
\cH^n(F_\epsilon\cap \p \cyl{\rho}{\height})\,d\rho < \dfrac{\epsilon}{2}.
$$
Hence  there exists $R_\epsilon\in (R,R+1)$ such that
$$\cH^n\big(F_\epsilon\cap    \p \cyl{R_\epsilon}{\height}\big)\le\frac{\epsilon}{2},\qquad
\cH^n\big(\Omega\cap \p^*F_\epsilon\cap \p \cyl{R_\epsilon}{\height}\big)=0.$$


Now, let $E_\epsilon:= F_\epsilon\cap \cyl{R_\epsilon}{\height}.$ Since
$\cH^n\big(\Omega\cap \p^*F_\epsilon\cap \p \cyl{R_\epsilon}{\height}\big)=0,$
we have
\begin{equation}\label{per_kesish}
\begin{aligned}
P(E_\epsilon,\Omega) = & P(E_\epsilon,\Omega_\height)
=P(F_\epsilon,\Omega_\height) + \cH^n\big(F_\epsilon\cap  \p \cyl{R_\epsilon}{\height}\big) -
P\big(F_\epsilon,\Omega_\height\setminus \overline{\cyl{R_\epsilon}{\height}}\big)\\
=& P(F_\epsilon,\Omega) + \cH^n\big(F_\epsilon \cap \p \cyl{R_\epsilon}{\height}\big) -
P\big(F_\epsilon,\Omega_\height\setminus \overline{\cyl{R_\epsilon}{\height}}\big).
\end{aligned} 
\end{equation}
By Hypothesis \ref{hyp:2} (a), $\cV(F_\epsilon)\ge\cV(E_\epsilon),$ thus
employing \eqref{per_kesish} we get
\begin{align*}
\cG(F_\epsilon) \ge & \cG(E_\epsilon) -
\cH^n(F_\epsilon\cap \p \cyl{R_\epsilon}{\height}) + 
 P(F_\epsilon, \Omega_\height\setminus \overline{\cyl{R_\epsilon}{\height}}) -
\int_{\p\Omega} \beta\,\chi_{F_\epsilon\setminus \cyl{R_\epsilon}{\height}}\,d\cH^n.
\end{align*}

By Lemma \ref{lem:Lower_bound_of_F0} applied with  $E=F_\epsilon$ and
$A= \Omega_\height\setminus \overline{\cyl{R_\epsilon}{\height}},$ we have
$$
P(F_\epsilon, \Omega_\height\setminus \overline{\cyl{R_\epsilon}{\height}}) -
\int_{\p\Omega} \beta\,\chi_{F_\epsilon\setminus \cyl{R_\epsilon}{\height}}\,d\cH^n\ge0.
$$
Consequently, from the choice of $F_\epsilon$ and $R_\epsilon$ we get
\begin{align*}
\cG(E_\epsilon) \le  & \cG(F_\epsilon) +
\cH^n(F_\epsilon \cap \p \cyl{R_\epsilon}{\height}) <
\inf\limits_{E\in BV(\Omega,\{0,1\})} \cG(E) +\epsilon.
\end{align*}
This concludes the proof of Step 2.

Now, observe that
$$
\inf\limits_{E\in BV(\Omega,\{0,1\})} \cG(E) \le
\inf\limits_{R>0} \inf\limits_{E\in BV(\cyl{R}{\height},\{0,1\})} \cG(E).
$$
On the other hand, since the mapping
$$
R\in(0,+\infty)\mapsto \inf\limits_{E\in BV(\cyl{R}{\height},\{0,1\})} \cG(E)
$$
is nonincreasing, Step 2 implies
$$
\inf\limits_{E\in BV(\Omega,\{0,1\})} \cG(E) \ge
\inf\limits_{R>0} \inf\limits_{E\in BV(\cyl{R}{\height},\{0,1\})} \cG(E),
$$
therefore \eqref{eq:min_R1} follows.
\end{proof}

As in \cite[Lemma 3]{LCAM07} the following lemma holds.

\begin{lemma}[\bf Good choice of a radius]\label{lem:key_lemma1}
Suppose that $\beta$ satisfies \eqref{hyp:main1} and 
Hypothesis \ref{hyp:2} holds. Let
$E^R$ be a minimizer of
$\cG $ in  $BV(\cyl{R}{\height},\{0,1\}).$ Then
for any $R>\cR_0$ there exists $t_R\in [r+1,\cR_0]$ such that
\begin{equation*} 
\cH^n(E^R\cap  \p \cyl{t_R}{\height}) = 0.
\end{equation*}
Hence
\begin{gather}\label{muhim_tenglik1}
P(E^R,\Omega) = P\big(E^R\setminus \overline{\cyl{t_R}{\height}},\Omega\big)+
P\big(E^R\cap \cyl{t_R}{\height},\Omega\big).
\end{gather}
\end{lemma}

\begin{proof}
The idea of the proof is to cut the $E^R$ with vertical cylinders,
similarly  to \cite[Lemma 5]{LCAM07} where
cuts with horizontal hyperplanes are performed.

For  $R>\cR_0$ by the isoperimetric-type inequality \cite[Theorem VI]{DG:54-1}, 
\eqref{new_coer},
the minimality of $E^R$  and by the definition
of $\mathfrak{a}$ we have
\begin{align*}
|E^R|^{\frac{n}{n+1}} \le & P(E^R)\le
\frac{\cG(E^R)- \inf \cV}{\kappa} =\frac{1}{\kappa}
\left(\inf\limits_{E\in BV(\cyl{R}{\height},\{0,1\})} \cG(E)+\inf\cV\right) 
\le \mathfrak{a}.
\end{align*}
Thus, for any $0<a<b$ one has
\begin{equation}\label{uniform_est1}
|E^R\cap (\cyl{b}{\height}\setminus \cyl{a}{\height})|\le
\mathfrak{a}^{\frac{n+1}{n}}.
\end{equation}

Take $r+1<r_1<r_2<r_3<\cR_0$  such that
$$
\cH^n(\Omega\cap \p^* E^R \cap \p \cyl{r_i}{\height})=0,\quad i=1,2,3,
$$
and set
$$v_1=|E^R\cap (\cyl{r_2}{\height}\setminus \cyl{r_1}{\height} )|,\qquad
v_2=|E^R\cap (\cyl{r_3}{\height}\setminus \cyl{r_2}{\height} )|,$$
$$m=\max\limits_{i=1,2,3} \cH^n(E^R \cap \p \cyl{r_i}{\height}).$$

{\bf Step 1.} We claim that
\begin{equation}\label{small v1 v21}
\min\{v_1,v_2\} \le \mu  m^{\frac{n+1}{n}},
\end{equation}
where $\mu:=\mu(\kappa,n)>0.$

It suffices to prove that
$$
v_1^{\frac{n}{n+1}}+v_2^{\frac{n}{n+1}}\le \,
 2\mu^{\frac{n}{n+1}} m.
$$
We have
\begin{align*}
v_1^{\frac{n}{n+1}}\le &   P\big(E^R\cap(\cyl{r_2}{\height}\setminus \overline{ \cyl{r_1}{\height}} )\big)
\le P(E^R,\cyl{r_2}{\height}\setminus \overline{\cyl{r_1}{\height}}) +
\cH^n(E^R\cap  \p \cyl{r_1}{\height})\\
& + \cH^n(E^R\cap  \p \cyl{r_2}{\height})+
\int_{\p \Omega} \chi_{E^R\cap (\cyl{r_2}{\height}\setminus \overline{\cyl{r_1}{\height}} )}\,d\cH^n\\
\le &  P(E^R,\cyl{r_2}{\height}\setminus \overline{\cyl{r_1}{\height}}) +
\int_{\p \Omega} \chi_{E^R\cap (\cyl{r_2}{\height}\setminus \overline{\cyl{r_1}{\height}} )} \,d\cH^n+2m.
\end{align*}
Similarly,
\begin{align*}
v_2^{\frac{n}{n+1}}
\le&  P(E^R,\cyl{r_3}{\height}\setminus \overline{\cyl{r_2}{\height}}) +
\int_{\p \Omega} \chi_{E^R\cap (\cyl{r_3}{\height}\setminus \overline{\cyl{r_2}{\height}} )} \,d\cH^n+2m.
\end{align*}
Hence
\begin{equation}\label{comp v1 v21}
v_1^{\frac{n}{n+1}}+v_2^{\frac{n}{n+1}}\le
 P(E^R,\cyl{r_3}{\height}\setminus \overline{\cyl{r_1}{\height}}) +
\int_{\p \Omega} \chi_{E^R\cap (\cyl{r_3}{\height}\setminus
\overline{\cyl{r_1}{\height}} )} \,d\cH^n +4m.
\end{equation}
Comparing $E^R\setminus (\cyl{r_3}{\height}\setminus 
\overline{\cyl{r_1}{\height}}))$ 
with $E^R,$ we get  
   $\cG(E^R)\le\cG(E^R\setminus (\cyl{r_3}{\height}\setminus 
\overline{\cyl{r_1}{\height}})),$ therefore 
from Hypothesis \ref{hyp:2} (c) we obtain
\begin{equation}\label{zgsd1}
\begin{aligned}
P(E^R) \le &
P\big(E^R\setminus (\cyl{r_3}{\height}\setminus
\overline{\cyl{r_1}{\height}})\big)
+\int\limits_{\p \Omega} [1+\beta]\,
\chi_{E^R\cap (\cyl{r_3}{\height}\setminus
\overline{\cyl{r_1}{\height}} )} \,d\cH^n.
\end{aligned}
\end{equation}
Inserting in \eqref{zgsd1} the identity
\begin{align*}
P(E^R\setminus (\cyl{r_3}{\height}\setminus \overline{\cyl{r_1}{\height}}))
=&P(E^R) +  \cH^n(E^R\cap   \p \cyl{r_1}{\height}) +
\cH^n(E^R\cap   \p \cyl{r_3}{\height})\\
&-P(E^R, \cyl{r_3}{\height}\setminus \overline{\cyl{r_1}{\height}}) -
\int_{\p \Omega} \chi_{E^R\cap (\cyl{r_3}{\height}\setminus \overline{\cyl{r_1}{\height}})} \,d\cH^n,
\end{align*}
we get
\begin{equation}\label{dwqs1}
P(E^R,\cyl{r_3}{\height}\setminus \overline{\cyl{r_1}{\height}}) -
\int_{\p \Omega} \beta\, \chi_{E^R\cap (\cyl{r_3}{\height}\setminus
\overline{\cyl{r_1}{\height}})} \,d\cH^n
\le 2m.
\end{equation}
By Lemma \ref{lem:Lower_bound_of_F0} applied with
 $A=\cyl{r_3}{\height}\setminus\overline{\cyl{r_1}{\height}}$
and $E=E^R,$ the left-hand-side of \eqref{dwqs1}
is not less than
\begin{gather*}
\kappa P(E^R,\cyl{r_3}{\height}\setminus\overline{\cyl{r_1}{\height}})+\kappa
\int_{\p \Omega} \chi_{E^R\cap (\cyl{r_3}{\height}\setminus \overline{\cyl{r_1}{\height}})} \,d\cH^n,
\end{gather*}
hence
$$P\big(E^R,\cyl{r_3}{\height}\setminus \overline{\cyl{r_1}{\height}}\big)+
\int_{\p \Omega} \chi_{E^R\cap (\cyl{r_3}{\height}\setminus\overline{\cyl{r_1}{\height}})} \,d\cH^n
 \le \frac{2m}{\kappa}.$$
Then from \eqref{comp v1 v21} it follows that
$$
v_1^{\frac{n}{n+1}}+v_2^{\frac{n}{n+1}}\le \,
\left(\frac{2m}{\kappa} +4m\right)= 2\mu^{\frac{n}{n+1}} m.
$$
This finishes the proof of Step 1.

Before going to Step 2 we need some preliminaries.
Choose any $R\ge \cR_0.$
Let $a_0=r+1,$ $b_0=\cR_0.$ Given $r+1\le a_k\le b_k\le \cR_0,$ $k\in\N,$ define
\begin{equation*} 
v_k=|E^R\cap (\cyl{b_k}{\height}\setminus {\cyl{a_k}{\height}})|.
\end{equation*}
By \eqref{strict_ineq1} $E^R\setminus \Omega_{\height-\frac14} = \emptyset,$
hence
\begin{equation*} 
|E^R\cap (\cyl{b}{\height}\setminus \cyl{a}{\height})| =
\int\limits_a^b \cH^n(E^R\cap  \p \cyl{\rho}{\height})\,d\rho,\quad 0\le a<b.
\end{equation*}
Therefore,  for $h_k=\frac{b_k-a_k}{4}$ it is possible to find $r_{k,1}\in(a_k,a_k+h_k),$
$r_{k,2}\in (\frac{a_k+b_k}{2}-\frac{h_k}{2},\frac{a_k+b_k}{2}+\frac{h_k}{2})$ and
$r_{k,3}\in(b_k-h_k,b_k)$ such that
\begin{equation}\label{hhjhjh1}
\cH^n(E^R\cap \p \cyl{r_{k,i}}{\height})\le \frac{v_k}{h_k},\quad
\cH^n(\Omega\cap \p^*E^R\cap \p \cyl{r_{k,i}}{\height})=0
\quad\text{for $i=1,2,3$.}
\end{equation}
We choose
$$(a_{k+1},b_{k+1}) =
\begin{cases}
(r_{k,1},r_{k,2}) & ~{\rm if}~
|E^R\cap (\cyl{r_{k,1}}{\height}
\setminus {\cyl{r_{k,2}}{\height}})|
 \le |E^R\cap (\cyl{r_{k,2}}{\height} \setminus {\cyl{r_{k,3}}{\height}})|,\\
(r_{k,2},r_{k,3}) & ~{\rm if}~
|E^R\cap (\cyl{r_{k,1}}{\height}
\setminus {\cyl{r_{k,2}}{\height}})|
> |E^R\cap (\cyl{r_{k,2}}{\height} \setminus {\cyl{r_{k,3}}{\height}})|.
\end{cases}
$$

Let $$m_k = \max\limits_{i=1,2,3} \cH^n(E^R\cap  \p \cyl{r_{k,i}}{\height}).$$

{\bf Step 2.}  Using the definition of $\cR_0$ we show that
\begin{equation}\label{krutoy_hisob1}
m_k\le \left(\dfrac12\right)^{(\frac{n+1}{n})^k}.
\end{equation}
Indeed, according to \eqref{small v1 v21}, \eqref{hhjhjh1} and the definition of
$(a_k,b_k)$ one has
$$v_{k+1}\le \mu  m_k^{\frac{n+1}{n}},\qquad\quad  m_k \le \frac{v_k}{h_k}.$$
By construction, $b_{k+1}-a_{k+1}\ge \frac{b_k-a_k}{8},$ i.e. $h_{k+1}\ge \frac{h_k}{8}.$
By induction one can check that
\begin{equation}\label{eq.ind1}
m_k\le \left(8^{\sum\limits_{j=1}^k j\alpha^j }
\left(\frac{\mu}{h_0}\right)^{\sum\limits_{j=1}^k \alpha^j}\,\frac{v_0}{h_0}
\right)^{1/\alpha^k},
\end{equation}
where  $\alpha:=\frac{n}{n+1}.$
Note that
$$\sum\limits_{j=1}^k j\alpha^j \le \alpha \sum\limits_{j\ge 1}j\alpha^{j-1} =
\frac{\alpha}{(1-\alpha)^2} = n(n+1).$$
Since $h_0=\frac{\cR_0-r-1}{4}$ and
$v_0\le \mathfrak{a}^{\frac{n+1}{n}}$ by  \eqref{uniform_est1},
the choice of $\cR_0$ in \eqref{eq:universal constant1} implies
$8^{n(n+1)}\,v_0/h_0  \le 1/2.$
Moreover $\left(\frac{\mu}{h_0}\right)^{\sum\limits_{j=1}^k \alpha^j}\le1,$
since $\frac{\mu}{h_0} = \frac{4\mu}{R_0-r-1}\le1.$
Now \eqref{krutoy_hisob1} follows from these estimates and \eqref{eq.ind1}.

{\bf Step 3.}
Let $i_k\in\{1,2,3\}$ be such that $m_k = \cH^n(E^R\cap   \p \cyl{r_{k,i_k}}{\height}).$
Since  $a_k\le r_{k,i_k}\le b_k,$ $\{a_k\}$ is nondecreasing and $\{b_k\}$ is nonincreasing,
there exists $t_R\in [r+1,\cR_0]$ such that $r_{k,i_k}\to t_R$ (possibly up to a
subsequence). Then, by Step 2,
$$\cH^n(E^R\cap  \p \cyl{t_R}{\height}) = \lim\limits_{k\to+\infty} m_k=0,$$
which concludes the proof of the lemma. 
\end{proof}

\begin{proof}[\bf Proof of Theorem \ref{teo: unconstrained minimizer1}]
Let us prove the existence of a minimizer of $\cG .$
For $R>\cR_0$ let $t_R\in [r+1,\cR_0] $ be as in Lemma \ref{lem:key_lemma1}.
Then from \eqref{muhim_tenglik1} and $\cV(E^R)\ge \cV(E^R\cap \cyl{t_R}{\height})$ we get
\begin{equation}\label{eq:min_cyl_some1}
\cG(E^R) \ge \cG(E^R\cap \cyl{t_R}{\height}) +
P\big(E^R\setminus \overline{\cyl{t_R}{\height}},\Omega\big) -
\int_{\p\Omega}\beta\chi_{E^R\setminus \overline{\cyl{t_R}{\height}}}d\cH^n.
\end{equation}
By \eqref{eq:bound_F-0} and the isoperimetric-type inequality 
\begin{equation}\label{oooooo1}
 P\big(E^R\setminus \overline{\cyl{t_R}{\height}},\Omega\big) - \int_{\p \Omega}
\beta\, \chi_{E^R\setminus \overline{\cyl{t_R}{\height}}} \,d\cH^n\ge\kappa
P\big(E^R\setminus \overline{\cyl{t_R}{\height}}\big)\ge \kappa
\big|E^R\setminus \overline{\cyl{t_R}{\height}}\big|^{\frac{n}{n+1}}.
\end{equation}
Thus from \eqref{eq:min_cyl_some1}
$$\cG(E^R)\ge \cG(E^R\cap \cyl{t_R}{\height}).$$
Hence, $F^R:=E^R\cap \cyl{t_R}{\height} \subseteq 
\cyl{\cR_0}{\height}$ satisfies
$$
\min\limits_{E\in BV(\cyl{R}{\height},\{0,1\})} 
\cG(E) = \cG(F^R).
$$
From \eqref{eq:bound_F-0} and the minimality of $F^R$ we get
$$
\kappa P(F^R) \le \cC(F^R,\Omega) \le \cG(F^R)-\inf\cV\le\kappa \mathfrak{a},
$$
and thus, by compactness  there exists $E \in 
BV(\cyl{\cR_0}{\height},\{0,1\})$
such that (up to a subsequence)
$F^R\to E$ in $L^1(\Omega)$ as $R\to+\infty.$ From the $L^1(\Omega)$-lower semicontinuity
of $\cG $ and  from \eqref{eq:min_R1} we conclude that
$E$ is a minimizer of $\cG.$

Now we prove that any minimizer  $E$ of $\cG$
satisfies $E\subseteq \cyl{\cR_0}{\height}.$ Arguing as in the proof of
\eqref{strict_ineq1}  one can show that $E\subseteq
\overline{\Omega_{\height-\frac14}}.$

{\bf Claim.} There exists $R>r+1$ (possibly depending on $\cV$ and $r$) such that
$E\subseteq \cyl{R}{\height}.$

For any $\rho>1$ such that
$\cH^n(\Omega\cap\p^* E\cap \p \cyl{\rho}{\height})=0,$ by the minimality of $E$ we have
$\cG(E) \le \cG(E\cap \cyl{\rho}{\height}),$ i.e.
\begin{align}\label{eq:dif_eq1}
P(E,\Omega_\height\setminus\overline{\cyl{\rho}{\height}}) - \int_{\p\Omega}
\beta\chi_{E\setminus \cyl{\rho}{\height}}\,d\cH^n \le \cH^n(E\cap \p \cyl{\rho}{\height}).
\end{align}
By Lemma \ref{lem:Lower_bound_of_F0}
\begin{equation}\label{qwqqqq}
P(E,\Omega_\height\setminus\overline{\cyl{\rho}{\height}}) - \int_{\p\Omega}
\beta\chi_{E\setminus \cyl{\rho}{\height}}\,d\cH^n\ge\kappa \Big(
P(E,\Omega_\height\setminus\overline{\cyl{\rho}{\height}})  +  \int_{\p\Omega}
\chi_{E\setminus \cyl{\rho}{\height}}\,d\cH^n\Big). 
\end{equation}

Moreover, by the isoperimetric-type inequality,
$$
 |E\setminus \cyl{\rho}{\height}|^{\frac{n}{n+1}} \le
 P(E,\Omega_\height\setminus\overline{\cyl{\rho}{\height}}) +
 \cH^n(E\cap   \p \cyl{\rho}{\height}) + \int_{\p\Omega}
\chi_{E\setminus \cyl{\rho}{\height}E}\,d\cH^n.
$$
therefore, \eqref{eq:dif_eq1} and \eqref{qwqqqq} imply
\begin{equation}\label{eq:final_diff_eq1}
 |E\setminus \cyl{\rho}{\height}|^{\frac{n}{n+1}} \le \frac{\kappa+1}{\kappa} \,
 \cH^n(E\cap  \p \cyl{\rho}{\height} ).
\end{equation}
Set $m(\rho)=|E\setminus \cyl{\rho}{\height}|.$ Clearly,
$m:(1,+\infty)\to [0,|E|].$  Moreover,
$m$ is absolutely continuous, nonincreasing,
$\lim\limits_{\rho\to+\infty} m(\rho) = 0$
and  $\cH^n(E \cap   \p \cyl{\rho}{\height}) = -m'(\rho)$ for a.e. $\rho>r+1.$
By \eqref{eq:final_diff_eq1}
$-m'(\rho)\ge \frac{\kappa+1}{\kappa}\,(n+1) m(\rho)^{\frac{n}{n+1}}.$
If $E$ is unbounded, then $m(\rho)>0$
for any $\rho>r+1,$ and thus, for any $\rho_1,\rho_2>r+1,$ $\rho_1<\rho_2$
we have
$$
m(\rho_1)^{\frac{1}{n+1}} - m(\rho_2)^{\frac{1}{n+1}} \ge \frac{\kappa+1}{\kappa}\,(\rho_2-\rho_1).
$$
Now letting $\rho_2\to+\infty$ we obtain $m(\rho_1)=+\infty,$ a contradiction.
Consequently, there exists $R>r+1$ such that $m(R)=0,$ i.e. $E\subseteq \cyl{R}{\height}.$

From the claim it follows that $E$ is a minimizer of
$\cG$ also  in $BV(\cyl{R}{\height},\{0,1\}).$ By Lemma
\ref{lem:key_lemma1} we can find
$t_R\in [r+1,\cR_0]$ such that $\cH^n(E\cap  \p \cyl{t_R}{\height})=0.$
Then using $\cV(E)\ge \cV(E\cap \cyl{t_R}{\height}),$ the relations
\eqref{eq:min_cyl_some1} - \eqref{oooooo1}
applied with $E$ in place of $E^R$ imply
\begin{equation*} 
\cG(E) \ge \cG(E\cap \cyl{t_R}{\height})
+\kappa\big|E\setminus \overline{\cyl{t_R}{\height}}\big|^{\frac{n}{n+1}}.
\end{equation*}
Therefore,  the minimality of $E$ yields
$\big|E\setminus \overline{\cyl{t_R}{\height}}\big|=0,$
i.e. $E\subseteq \cyl{t_R}{\height}.$ Since
$t_R\le \cR_0,$ the conclusion follows.
\end{proof}

\renewcommand{\appendixname}{Appendix }

\section{Local well-posedness}\label{sec:short_time_existence}

\renewcommand{\appendixname}{}

In this appendix we sketch the proof of short time existence and 
uniqueness
of  smooth hypersurfaces
moving with normal velocity equal to their
mean curvature in $\Omega$  and
meeting the boundary $\p\Omega$ at a prescribed (not necessarily constant)
angle. 
The following theorem is a generalization of \cite[Theorem 1]{KKR:95}, where
short time existence and uniqueness have been 
proven for constant $\beta.$

\begin{theorem}[\textbf{Short time existence and 
uniqueness}]\label{teo:short_time_existence}
Let $\beta\in C^{1+\alpha}(\p\Omega),$
$\|\beta\|_\infty\le 1-2\kappa,$ $\kappa\in (0,\frac12]$ and
$E_0\subset\Omega$ be a bounded open set such that
$\Gamma_0=\overline{\Omega\cap \p E_0}$ is a bounded 
$C^{3+\alpha}$-hypersurface, $\alpha\in (0,1).$
Assume that $\cU \subset \R^n$ is a bounded open 
set with $C^{3+\alpha}$-boundary,
$p^0\in C^{3+\alpha}(\overline{\cU}, \R^{n+1})$ is a parametrization 
of $\Gamma_0$  such that $p_{n+1}^0>0$ in $\cU,$  $p^0_{n+1} = 0$ on $\p\cU,$
and
\begin{equation}\label{eq:neumann_boundary_c}
-e_{n+1} + \beta(p^0)\nu_0 = Dp^0[n^0]
\quad \text{on $\p\cU,$}
\end{equation}
where
$n^0 =  (n_1^0,\ldots,n_n^0)$ is the outward unit normal to
$\p\cU,$  $\nu_0 = \nu(p^0)$ is the outward unit normal of $\Gamma_0$ at 
$p^0$ and $Dp^0[n^0] = \sum\limits_{j=1}^n n_j^0p_{\sigma_j}^0.$ 
Then there exists $T_0=T_0(\|\beta\|_{C^{1+\alpha}},
\|p^0\|_{C^{3+\alpha}})>0,$ 
a unique family of bounded open sets 
$\{E(t)\subset \Omega:\,\, t\in [0,T_0]\}$
with a parametrization $p\in C^{1+\alpha/2,2+\alpha}([0,T_0]
\times\overline{\cU},\R^{n+1})$
of $\Gamma(t)=\overline{\Omega\cap \p E(t)}$ solving the parabolic system
\begin{equation}\label{eq:par_system}
p_t =\mathrm{trace}((Dp\cdot(Dp)^T)^{-1}D^2p)\,\,\text{ in $(0,T_0)\times \cU,$}
\end{equation}
where $(Dp\cdot(Dp)^T)_{ij} = p_{\sigma_i}\cdot 
p_{\sigma_j}$ and $(D^2p)_{ij} = p_{\sigma_i\sigma_j}$,
coupled with the initial condition $p(0,\cdot) = p^0,$
the boundary conditions
\begin{equation}\label{eq:Nboundary_cond_to_par_system}
\begin{cases}
p_{n+1}(t,\cdot) = 0 & \text{on $\p\cU$ for any $t\in [0,T_0],$ }\\
e_{n+1}\cdot \nu(p(t,\cdot)) = \beta(p(t,\cdot))
&  \text{on $\p\cU$ for any $t\in [0,T_0],$ } 
\end{cases}
\end{equation}
and  the orthogonality conditions
\begin{equation}\label{eq:boundary_cond_to_par_system}
\begin{aligned}
&Dp^0[n^0]\cdot {\tau_0}_i = 0\,\,
\text{on $[0,T_0]\times\p\cU$ for every $i=1,\ldots,n-1,$}
\end{aligned}
\end{equation}
where $\nu(p(t,\cdot))$ is the outward unit normal 
to $\Gamma(t)$ at $p(t,\cdot)$ and
${\tau_0}_1,\ldots,{\tau_0}_{n-1}\in\R^n\times\{0\}$ is a basis for 
the tangent space of $\Gamma_0\cap\p\Omega$ at $p^0.$
\end{theorem}

\begin{remark}
Assumption \eqref{eq:neumann_boundary_c} on $p^0$ is not restrictive.
Indeed, if $q:\p\cU\to \Gamma_0\cap\p\Omega$ is a $C^{3+\alpha}$
parametrization of the contact set, we may extend it
to a sufficiently small tubular neighborhood 
$S:=\{x\in \cU:\,\,\distance(x,\p\cU)<\epsilon\}$ 
of $\p\cU$ in $\cU$  with the properties that
$q$ is a $C^{3+\alpha}$ diffeomorphism, $q(S)\subset\Gamma_0$
and
$$
q(\sigma) = q(\varsigma) + |\sigma-\varsigma| (e_{n+1}- \beta(q(\varsigma)) 
\nu_0(q(\varsigma))) +O(|\sigma-\varsigma|^2),
$$
where $\varsigma$ is the projection of $\sigma\in S$ on $\p\cU.$
Since $\sigma= \varsigma - |\sigma-\varsigma|n^0(\varsigma),$ it follows
$$\nabla q(\varsigma)\,n^0(\varsigma) =- e_{n+1} + \beta(q(\varsigma)) \nu_0(q(\varsigma)),$$
which is \eqref{eq:neumann_boundary_c}.
Now we may arbitrarily extend $q$ to a $C^{3+\alpha}$ diffeomorphism in $\overline{\cU}$
such that $q(\overline{\cU})=\Gamma_0.$
\end{remark}

\begin{remark}\label{rem:normal_to_surface}
The unit normal to $\Gamma(t)$
at the point  $p(t,\sigma_1,\ldots,\sigma_n)\in \Gamma(t)$ can be written 
with a (standard) abuse of notation
$\nu=\nu(p(t,\sigma_1,\ldots,\sigma_n)) = \frac{\tilde \nu}{|\tilde \nu|},$ where
$$
\tilde \nu:=\tilde \nu(p_\sigma)=
\det
\begin{bmatrix}
e_1 & e_2 &\ldots & e_{n} & e_{n+1}\\
&     & p_{\sigma_1} &           & \\
&     & p_{\sigma_2} &           & \\
&     & \vdots &           & \\
&     & p_{\sigma_n} &           &
\end{bmatrix}.
$$
\end{remark}

\begin{proof}[ Proof of Theorem \ref{teo:short_time_existence}]
The idea of the proof is standard: first we linearize the 
equation around the boundary
conditions, then prove the existence result for the linearized 
system and finally we use a
fixed point argument.

{\bf Step 1.} Let us  linearize system 
\eqref{eq:par_system} fixing some $t_0>0.$ 
Let $\Class(t_0)\subset C^{1+\alpha/2,2+\alpha}
([0,t_0]\times \overline{\cU},\R^{n+1})$
be  the nonempty convex set 
consisting of all functions $w\in C^{1+\alpha/2,2+\alpha}([0,t_0]\times 
\overline{\cU},\R^{n+1})$
such that
\begin{itemize}
 \item[1)] $w(0,\cdot )=p^0,$
 \item[2)] $w_{n+1}(t,\cdot) = 0$ on $\p\cU$ for any $t\in[0,t_0],$
 \item[3)] $\sum\limits_{j=1}^{n} n_j^0 w_{\sigma_j} \cdot {\tau_0}_i = 0$
on $[0,t_0]\times\p\cU$ for every $i=1,\ldots,n-1.$
\end{itemize}
For $w\in \Class(t_0)$ set 
$f(t,w):=\mathrm{trace}\big[\big((Dw\cdot(Dw)^T)^{-1} - 
(Dp^0\cdot(Dp^0)^T)^{-1}\big)D^2w\big].$
Then \eqref{eq:par_system} is equivalent to 
$$
w_t = \mathrm{trace}\big[(Dp^0\cdot(Dp^0)^T)^{-1} D^2w\big] +f(t,w).
$$
Notice that 
$$
|f(t,w)|\le c(\|p^0\|_{C^1(\overline{\cU})})\|w\|_{C^{0,2}([0,t_0]\times 
\overline{\cU})} \|w-p_0\|_{C^{0,1}([0,t_0]\times 
\overline{\cU})},
$$
where $c(\|p^0\|_{C^1(\overline{\cU})})>0.$
Now we linearize the contact angle condition.
Since we have   $e_{n+1}\cdot \nu(p^0) = \beta(p^0),$ 
from Remark \ref{rem:normal_to_surface}  it follows that
\begin{equation}\label{ffff}
e_{n+1}\cdot \big(\tilde\nu(w_\sigma)-\tilde\nu(p_\sigma^0)\big) = 
\beta(w)|\tilde\nu(w_\sigma)| - \beta(p^0)|\tilde\nu(p_\sigma^0)|. 
\end{equation}
Let
$ H_1(t,w): = \tilde \nu(w_\sigma) - \tilde \nu(p_\sigma^0) - D\tilde \nu(p_\sigma^0)[w_\sigma-p_\sigma^0]$,
where
$$
D\tilde \nu =
\begin{bmatrix}
D_{p_{\sigma_1}}\tilde \nu^1 & D_{p_{\sigma_2}}\tilde \nu^1 & \ldots & D_{p_{\sigma_n}}\tilde \nu^1\\
D_{p_{\sigma_1}}\tilde \nu^2 & D_{p_{\sigma_2}}\tilde \nu^2 & \ldots & D_{p_{\sigma_n}}\tilde \nu^2\\
\vdots & \vdots & \ldots & \vdots\\
D_{p_{\sigma_1}}\tilde \nu^{n+1} & D_{p_{\sigma_2}}\tilde \nu^{n+1} & \ldots & D_{p_{\sigma_n}}\tilde \nu^{n+1}
\end{bmatrix},
\qquad
q_\sigma=
\begin{bmatrix}
q_{\sigma_1}\\
q_{\sigma_2}\\
\vdots\\
q_{\sigma_n}\\
\end{bmatrix}
=
\begin{bmatrix}
(q_1)_{\sigma_1} &\ldots & (q_{n+1})_{\sigma_1} \\
(q_1)_{\sigma_2} &\ldots & (q_{n+1})_{\sigma_2} \\
\ldots & \vdots & \ldots \\
(q_1)_{\sigma_n} &\ldots & (q_{n+1})_{\sigma_n}
\end{bmatrix}
$$
and
$$
D\tilde \nu [q_\sigma] =
\begin{bmatrix}
\sum\limits_{i=1}^n D_{p_{\sigma_i}}\tilde \nu^1 \cdot q_{\sigma_i}\\
\sum\limits_{i=1}^n D_{p_{\sigma_i}}\tilde \nu^2 \cdot q_{\sigma_i}\\
\vdots \\
\sum\limits_{i=1}^n D_{p_{\sigma_i}}\tilde \nu^{n+1} \cdot q_{\sigma_i}\\
\end{bmatrix}=
\begin{bmatrix}
\sum\limits_{i=1}^n  \sum\limits_{j=1}^{n+1} D_{(p_j)_{\sigma_i}}\tilde \nu^1 \cdot(q_j)_{\sigma_i}\\
\sum\limits_{i=1}^n  \sum\limits_{j=1}^{n+1} D_{(p_j)_{\sigma_i}}\tilde \nu^2 \cdot(q_j)_{\sigma_i}\\
\vdots \\
\sum\limits_{i=1}^n  \sum\limits_{j=1}^{n+1} D_{(p_j)_{\sigma_i}}\tilde \nu^{n+1} \cdot(q_j)_{\sigma_i}\\
\end{bmatrix}
$$
Clearly, $|H_1(t,w)| = 
O\Big(\|w-p^0\|_{C^{0,1}([0,t_0]\times 
\overline{\cU})}^2\Big).$ 
Moreover, 
$$
|\tilde\nu(w_\sigma)| = |\tilde\nu(p_\sigma^0)| + \nu(p^0) \cdot 
D\tilde\nu(p_\sigma^0)[w_\sigma - p_\sigma^0] +H_2(t,w) 
$$
with $|H_2(t,w)| = 
O\Big(\|w-p^0\|_{C^{0,1}([0,t_0]\times 
\overline{\cU})}^2\Big).$ Finally,
since $\beta\in C^{1+\alpha}(\p\Omega)$  we have
\begin{align*}
\beta(w)|\tilde\nu(w_\sigma)| - \beta(p^0)|\tilde\nu(p_\sigma^0)| = 
\beta(p^0) \nu(p^0) \cdot D\tilde\nu(p_\sigma^0)[w_\sigma - p_\sigma^0] +H_3(t,w),
\end{align*}
where $
H_3(t,w) = 
O\Big(\|w-p^0\|_{_{C^{0,1}([0,t_0]\times 
\overline{\cU})}}^2\Big). 
$
Thus, \eqref{ffff} is equivalent to 
\begin{equation*}
(e_{n+1} - \beta(p^0) \nu(p^0) )\cdot 
D\tilde\nu(p_\sigma^0)[w_\sigma] = (e_{n+1} - \beta(p^0) \nu(p^0) )\cdot 
D\tilde\nu(p_\sigma^0)[p_\sigma^0] + H_4(t,w), 
\end{equation*}
where $H_4(t,w) =O\Big(\|w-p^0\|_{_{C^{0,1}([0,t_0]\times 
\overline{\cU})}}^2\Big).$

Thus we have the following linear parabolic system of equations
\begin{gather*}
\cL(\sigma,\p_t,\p_\sigma) w= f \,\,\text{in $(0,t_0)\times\cU$}
\end{gather*}
subject to the boundary conditions
$\cB(\varsigma,\p_\sigma)w = F(t,\varsigma)$ on $[0,t_0]\times\p\cU,$
where
$$F(t,\varsigma)=
\left[
0, \,\, (e_{n+1} - \beta(p^0) \nu(p^0) )\cdot 
D\tilde\nu(p_\sigma^0)[p_\sigma^0] + H_4(t,w), \,\,
\underbrace{0, \,\,\ldots \,\,,0}_{(n-1)-\text{times}}
\right]^T 
$$
and, under the notation $\{g_0\}^{ij} = \{p_{\sigma_i}^0\cdot 
p_{\sigma_j}^0\}^{-1},$ $\tilde \nu_0 = \tilde \nu(p_\sigma^0), $ $\beta_0 = \beta(p^0)$
the $(n+1)\times(n+1)$-matrices $\cL(\sigma,t,\xi,\zeta)$ and 
$\cB(\varsigma,\xi),$ $\xi\in \R^{n},$ $\zeta\in\C$
are defined as follows:
$$
\cL(\sigma,\zeta,\xi):= {\rm diag}\left(
\zeta - \sum\limits_{i,j=1}^n g_0^{ij} \xi_i\xi_j,
\zeta - \sum\limits_{i,j=1}^n g_0^{ij} \xi_i\xi_j,
\ldots, \zeta - \sum\limits_{i,j=1}^n g_0^{ij} \xi_i\xi_j\right),
$$
$$
\cB(\varsigma,\xi):=
\begin{bmatrix}
0 &  \ldots &  1\\
\sum\limits_{k=1}^{n+1} \sum\limits_{i=1}^n (-\delta_{k,n+1} - \beta_0
 \nu_0^k) D_{(p_1)_{\sigma_i}}
\tilde \nu_0^k \xi_i &
 \ldots &  \sum\limits_{k=1}^{n+1} \sum\limits_{i=1}^n (-\delta_{k,n+1} -
\beta_0 \nu_0^k) D_{(p_{n+1})_{\sigma_i}}
\tilde \nu_0^k \xi_i \\
{\tau_0}_1^1\sum\limits_{i=1}^n n_i^0 \xi_i &  \ldots &
 {\tau_0}_1^{n+1}\sum\limits_{i=1}^n n_i^0 \xi_i\\
\vdots & \vdots &   \vdots \\
{\tau_0}_{n-1}^1\sum\limits_{i=1}^n n_i^0 \xi_i  & \ldots &
  {\tau_0}_{n-1}^{n+1}\sum\limits_{i=1}^n n_i^0 \xi_i\\
\end{bmatrix},
$$
where the first row must be intended as $[0,\ldots,0,1].$

{\bf Step 2.}
Now we check the compatibility conditions \cite{So:1965}.
Take any $\varsigma\in \p\cU$ and let $\theta$ be in the tangent space
of $\p\cU$ at $\varsigma.$
Let $\lambda_0:= \lambda_0(\varsigma, \zeta,\theta)$ be a solution of the quadratic equation
$$
h(\lambda;\varsigma, \zeta,\theta):=\zeta + \sum\limits_{i,j=1}^n g_0^{ij}\theta_i\theta_j 
- 2\lambda \sum\limits_{i,j=1}^n g_0^{ij}\theta_i n_j^0
+ \lambda^2 \sum\limits_{i,j=1}^ng_0^{ij}n_i^0n_j^0=0
$$
in $\lambda\in\C $ with  positive imaginary part.
Notice that $\det\cL = (h(\lambda;\varsigma,\zeta,\theta))^{n+1}$ and
$$
\hat\cL = (\det\cL) \cL^{-1} = 
\diag((h(\lambda;\varsigma, \zeta,\theta))^n,\ldots,(h(\lambda;\varsigma, \zeta,\theta))^n).
$$
In order to prove the compatibility conditions we should prove  that the rows of matrix
$$\cB(\varsigma,i(\theta- \lambda n^0))\hat\cL(x,\zeta, i(\theta - \lambda n^0))$$
are linearly independent modulo
the polynomial $(\lambda-\lambda_0)^{n+1}$ whenever $\Re(\zeta)\ge0,$ $|\zeta|>0.$
According to the  definitions of $\cL$ and $\cB$ one checks \cite{KKR:95} that
the compatibility conditions are equivalent to the conditions
$$
c_1e_{n+1} + c_2 \tilde \nu(p^0) + \sum\limits_{i=1}^{n-1} c_{i+2}{\tau_0}_i =0 \,\,
 \Longleftrightarrow \,\, c_1=c_2=\ldots=c_{n+1}=0.
$$
Since  a basis of the tangent space
$\{{\tau_0}_i\}_{i=1}^{n-1}$ of $\Gamma_0\cap \p\Omega$
belongs to the horizontal subspace of $\R^{n+1}$
and $\tilde \nu(p^0)$ is normal   to $\Gamma_0\cap \p\Omega$
at $p^0$ we have $c_3=\ldots=c_{n+1}=0.$
Moreover, since $|\beta|\le 1-2\kappa,$ and $\Gamma_0$ satisfies 
the contact angle condition,
 $e_{n+1}$ and $\tilde \nu(p^0)$ are linearly independent,
i.e. $c_1=c_2=0.$

{\bf Step 3. } By \cite[Theorem 4.9]{So:1965} since 
$\p\cU \in C^{3+\alpha},$ $\beta \in C^{1+\alpha}(\p\Omega)$
and the  compatibility conditions hold,
for any $\tilde f,\tilde F\in C^{0,\alpha}([0,t_0]\times\overline{\cU}),$ 
$p^0\in C^{3+\alpha}(\overline{\cU})$
there exists a unique solution 
$w\in C^{1+\alpha/2,2+\alpha}([0,t_0]\times\overline{\cU})$ 
such that
\begin{align*}
& w_t = \tr((Dp^0\cdot(Dp^0)^t)^{-1} D^2w) + \tilde f,\\
&w(0,\cdot) = p^0,\\
& w_{n+1}(t,\cdot)= 0\qquad\text{on $\p\cU$ for any $t\in [0,t_0],$}\\
& (e_{n+1} - \beta(p^0) \nu(p^0)) \cdot D\tilde \nu(p^0)[w_\sigma] = 
(e_{n+1} - \beta(p^0) \nu(p^0)) \cdot D\tilde \nu(p^0)[p_\sigma^0] + \tilde F(t,x)
\quad \text{on $[0,t_0]\times \p\cU$},\\
& \left(\sum\limits_{j=1}^n n_j^0 w_{\sigma_j}\right)\cdot {\tau_0}_i = 0\qquad 
\text{on $[0,t_0]\times \p\cU$ and $i=1,\ldots,n-1.$}
\end{align*}

{\bf Step 4.} Finally, mimicking \cite{ES:92} we can
prove the existence of and uniqueness of solution
\eqref{eq:par_system}-\eqref{eq:boundary_cond_to_par_system}
in time interval $[0,T_0]$ for some sufficiently small $T_0>0$
depending on $\|\beta\|_{C^{1+\alpha}}$ and $\|p^0\|_{C^{3+\alpha}}.$
\end{proof}

We call $E(t)$ the smooth flow starting from $E_0.$

\begin{proposition}[\bf Comparison for strong solutions]
Let $\beta_i\in (-1,1),$ $E_0^{(i)}\subset\Omega$ be bounded sets 
such that $\overline{\Omega\cap \p E_0^{(i)} }$ are
$C^{3+\alpha}$ hypersurfaces,  and the smooth flows $E^{(i)}(t)$ 
starting from $E_0^{(i)}$ exist in $[0,T_0],$ $i=1,2.$ 
If $\beta_1\le \beta_2$ and
$\distance(\Omega\cap \p E_0^{(1)},\Omega\cap \p E_0^{(2)})>0,$ then
$\distance(\Omega\cap \p E^{(1)}(t),\Omega\cap \p E^{(2)}(t))>0$ for 
all $t\in [0,T_0].$
\end{proposition}

\begin{proof}
The proof is an adaptation of   the
classical one (see for instance \cite{Bell:2013}).
\end{proof}

\section*{Acknowledgements}

The first author would like to express his gratitude to the
International Centre for Theoretical Physics (ICTP) in Trieste
for its hospitality and facilities.  The second author is very grateful to the
International Centre for Theoretical Physics (ICTP) and the
Scuola Internazionale Superiore di Studi Avanzati (SISSA) in Trieste,
where this research was made.

\end{document}